\input ./style/arxiv-general.cfg
\documentclass[aop,MSNbibl,dvips]{arximspdf}
\makeatletter
   \@ifpackageloaded{graphicx}{}{\usepackage{graphicx}}
\makeatother
\usepackage{mathbh}

\doi{10.1214/15-AOP1021}% Updated by VTEXPTS2LaTeX.exe, 23.04.2015
\volume{44}
\issue{3}
\pubyear{2016}
\firstpage{2198}
\lastpage{2263}
\docsubty{FLA}

\makeatletter

\def\sklfrac#1#2{(#1/#2)}

\newcommand{\lleft}{\left}
\newcommand{\rrvert}{\vert}
\newcommand{\rright}{\right}
\newcommand{\llvert}{\vert}
\newtheorem{teo}{Theorem}[section]
\newtheorem{cor}[teo]{Corollary}
\newtheorem{lem}[teo]{Lemma}
\newtheorem{prop}[teo]{Proposition}
\newproclaim{rmk}{Remark}
\newproclaim{rmks}{Remarks}
\renewcommand{\mid}{|}
\renewcommand{\P}{\mathbb{P}}
\newcommand{\E}{\mathbb{E}}
\newcommand{\R}{\mathbb{R}}
\newcommand{\N}{\mathbb{N}}
\newcommand{\1}{\mathbh{1}}
\newcommand{\eps}{\varepsilon}
\newcommand{\Z}{\mathbb{Z}}
\newcommand{\Prob}{\operatorname{Prob}}
\newcommand{\Om}{\Omega}
\newcommand{\la}{\lambda}
\newcommand{\calE}{\mathcal{E}}
\newcommand{\calH}{\mathcal{H}}
\newcommand{\calM}{\mathcal{M}}
\newcommand{\cT}{\mathcal{T}}
\makeatother

\begin{document}
\begin{frontmatter}

\title{Intermittency for branching random walk in~Pareto environment}
\runtitle{Intermittency for BRWRE}

\begin{aug}
% Corresponding author: Matthew Roberts - mattiroberts@gmail.com% Updated by VTEXPTS2LaTeX.exe, 24.04.2015 15:48
%Updated by VTEXPTS2LaTeX.exe, 23.04.2015 09:09
\author[A]{\fnms{Marcel}~\snm{Ortgiese}\ead[label=e1]{marcel.ortgiese@uni-muenster.de}}
\and
\author[B]{\fnms{Matthew I.}~\snm{Roberts}\corref{}\thanksref{T1}\ead[label=e2]{mattiroberts@gmail.com}}
\runauthor{M.~Ortgiese and M.~I.~Roberts}
\affiliation{Westf\"alische Wilhelms-Universit\"at M\"unster and
University of Bath}
%\dedicated{}
\address[A]{Institut f\"ur Mathematische Statistik\\
Westf\"alische Wilhelms-Universit\"at M\"unster\\
Einsteinstra\ss{}e 62\\
48149 M\"unster\\
Germany\\
\printead{e1}}
\address[B]{Department of Mathematical Sciences\\
University of Bath\\
Claverton Down\\
Bath BA2 7AY\\
United Kingdom\\
\printead{e2}}
\end{aug}
\thankstext{T1}{Supported in part by an EPSRC postdoctoral fellowship
(EP/K007440/1).}

% HISTORY:
%
\received{\smonth{5} \syear{2014}}% Updated by VTEXPTS2LaTeX.exe,
%23.04.2015 09:09
%
\revised{\smonth{2} \syear{2015}}% Updated by VTEXPTS2LaTeX.exe,
%23.04.2015 09:09

% ABSTRACT
%
\begin{abstract}
We consider a branching random walk on the lattice, where the branching
rates are given by an i.i.d. Pareto random potential. We describe the
process, including a detailed shape theorem, in terms of a system of
growing lilypads. As an application we show that the branching random
walk is intermittent, in the sense that most particles are concentrated
on one very small island with large potential. Moreover, we compare the
branching random walk to the parabolic Anderson model and observe that
although the two systems show similarities, the mechanisms that control
the growth are fundamentally different.
\end{abstract}

% KEYWORDS
\end{frontmatter}

%s1 #&#
\section{Introduction and main results}

%s1.1 #&#
\subsection{Introduction}\label{sec1}
Branching processes in random environments are a classical subject
going back to~\cite{SmithWilkinson69,Wilkinson67}. We are interested
in branching random walks (BRW), where particles branch but also have
spatial positions and are allowed to migrate to other sites.

We consider a particular variant of the model defined on $\Z^d$. Start
with a single particle at the origin. Each particle performs a
continuous-time nearest-neighbor symmetric random walk on $\Z^d$. When
at site $z\in\Z^d$, a particle splits into two new particles at rate
$\xi(z)$, where the potential $(\xi(z), z \in\Z^d)$ is a
collection of nonnegative i.i.d. random variables. The two new
particles then repeat the stochastic behavior of their parent, started
from $z$. This particular model was first described in~\cite{GM90},
although until now analysis has concentrated on the expected number of
particles: see the surveys \cite{GK05,M11,KW14}. In this article we
show that while the study of the actual number of particles is more
technically demanding, it is still tractable, and reveals surprising
and interesting behavior which warrants further investigation.

We begin by recalling what is known about the expected number of
particles. More precisely, we fix a realization of the environment
$(\xi(z), z \in\Z^d)$ and take expectations over migration and
branching. We denote the expected number of particles by
\[
u(z,t) = E^\xi\bigl[ \# \{ \mbox{particles at site } z \mbox{ at
time } t \} \bigr].
\]
The superscript $\xi$ indicates that this expression is still random
due to its dependence on the environment. By considering the different
possibilities in a infinitesimal time step, one can easily see that
$u(z,t)$ solves the following stochastic partial differential equation,
known as the \emph{parabolic Anderson model} (PAM):
\begin{eqnarray*}
\partial_t u(z,t) & =& \Delta u(z,t) + \xi(z) u(z,t)\qquad\mbox{for
}z \in\Z^d, t \geq0,
\\
u(z,0) & =& \1_{\{z = 0 \}}\qquad\mbox{for } z \in\Z^d.
\end{eqnarray*}
Here, $\Delta$ is the discrete Laplacian defined for any function $f\dvtx
\Z^d \rightarrow\R$ as
\[
\Delta f(z) = \sum_{y \sim z} \bigl(f(y) - f(z)\bigr),
\qquad z \in\Z^d,
\]
where we write $y \sim z$ if $y$ is a neighbor of $z$ on the lattice
$\Z^d$. Starting with the seminal work of~\cite{GM90} the PAM has
been intensively studied in the last twenty years. Much interest stems
from the fact that it is one of the more tractable models to exhibit an
effect called \emph{intermittency}, which roughly means that the
solution is concentrated in a few peaks that are spatially well
separated. For the PAM this effect is well understood: see the
surveys~\cite{GK05,M11,KW14}. The size and the number of peaks depends
essentially on the tail of $\xi$, that is, on the decay of $\P(\xi
(0) > x )$ for large $x$. For a bounded potential the size of the
relevant islands grows with $t$. In the intermediate regime, when the
potential is double exponentially distributed, the size of the islands
remains bounded. Finally, it is believed, and in a lot of cases proven,
that for any potential with heavier tails, there is a single island
consisting of a single point containing almost all of the mass. In the
most extreme case when the potential is Pareto distributed, a very
detailed understanding of the evolution of the solution has emerged;
see~\cite{HMS08,KLMS09,MOS11}.

While the expected number of particles, that is, the PAM, is well
understood, a lot less is known for the actual number of particles in
the branching random walk. The only results so far for this particular model
concern higher moments of particles
\[
E^\xi\bigl[ \# \{ \mbox{particles at site } z \mbox{ at time } t
\}^n \bigr].
\]
These were studied in a special case by~\cite{ABMY00} using analytic
methods and for a larger class of potentials and providing finer
asymptotics in~\cite{GKS13} using spine methods for higher moments as
developed in~\cite{HR11}.

%s1.2 #&#
\subsection{Main result}

Motivated by the detailed understanding of the\break  parabolic Anderson model
in the case of Pareto potentials, we from now on assume that $\{ \xi
(z), z \in\Z^d\}$
is a collection of independent and identically distributed Pareto
random variables.
Denoting the underlying probability measure on $(\Om, \mathcal{F})$
by $\Prob$,
we have in particular that for a parameter $\alpha> 0$, and any $z \in
\Z^d$,
\[
\Prob\bigl( \xi(z) > x \bigr) = x^{- \alpha} \qquad\mbox{for all } x \geq1.
\]
Moreover, we assume throughout that $\alpha> d$, which is a necessary
condition for the total mass in
the PAM to remain finite; see \cite{GM90}.

For a fixed environment $\xi$, we denote by $P_y^\xi$ the law of the
branching simple random walk in continuous time with binary branching
and branching rates $\{\xi(z), z \in\Z^d\}$ started with a single
particle at site $y$.
Finally, for any measurable set $F \subset\Om$, we define
\[
\P_y ( F \times\cdot) = \int_F
P_y^\xi( \cdot) \Prob(d \xi).
\]
If we start with a single particle at the origin, we omit the subscript
$y$ and simply write
$P^\xi$ and $\P$ instead of $P_0^\xi$ and $\P_0$.

We define
$Y(z,t)$ to be the set of particles at the point $z$ at time $t$.
Moreover, we let $Y(t)$ be the set of all particles present
at time $t$.
We are interested in the number of particles
\[
N(z,t) = \# Y(z,t)\quad\mbox{and}\quad N(t) = \# Y(t).
\]
The aim of this paper is to understand the long-term evolution of
the branching random walk, and we therefore introduce\vspace*{1pt} a rescaling
of time by a parameter $T > 0$. We also rescale space
and the potential.
If $q=\frac{d}{\alpha-d}$, the right scaling factors for the
potential, respectively, space,
turn out to be
\[
a(T) = \biggl(\frac{T}{\log T} \biggr)^q \quad\mbox{and} \quad
r(T) = \biggl(\frac{T}{\log T} \biggr)^{q+1}.
\]
This scaling is the same as that used in the parabolic Anderson model
(cf. \cite{HMS08,KLMS09})
and guarantees the right balance between the peaks of the potential and
the cost of
reaching the corresponding sites.
We now define the rescaled lattice as
\[
L_T = \bigl\{ z \in\R^d\dvtx  r(T) z \in\Z^d
\bigr\},
\]
and for $z\in\R^d$, $R\geq0$ define $L_T(z,R) = L_T\cap B(z,R)$
where $B(z,R)$ is the open ball of radius $R$ about $z$ in $\R^d$. For
$z \in L_T$, the rescaled potential is given by
\[
\xi_T(z) = \frac{\xi(r(T) z)}{a(T)},
\]
and we set $\xi_T(z) = 0$ for $z \in\R^d \setminus L_T$.
The correct scaling for the number of particles at $z$ is given by
\[
M_T(z,t) = \frac{1}{a(T)T} \log_+ N\bigl(r(T) z,t T\bigr).
\]
We will see that in order to bound the rescaled number of particles
$M_T(z,t)$, we first have to understand at what time $z$ is hit. We
therefore introduce
the hitting time of a point $z \in L_T$ as
\[
H_T(z) = \inf\bigl\{t>0\dvtx  Y\bigl(r(T) z,tT\bigr) \neq\varnothing
\bigr\}.
\]

Our main result states that we can predict the behavior of the
branching random walk
purely in terms of the potential. For this purpose we introduce the
\emph{lilypad model}.

For any site $z \in L_T$, we set
\[
h_T(z) = \mathop{\inf_{y_0,\ldots,y_n \in L_T\dvtx }}_{y_0 = z, y_n = 0}
\Biggl( \sum_{j=1}^{n} q\frac{\llvert   y_{j-1} - y_{j}\rrvert  }{\xi_T(y_{j})}
\Biggr),
\]
where throughout $\llvert  \cdot\rrvert  $ will denote the $\ell_1$-norm on $\R^d$.
We call $h_T(z)$ the first hitting time of $z$ in the lilypad model.
We think of each site $y$ as being home to a lilypad, which grows at
speed $\xi_T(y)/q$. Note that $h_T(0)=0$, so that the lilypad at the
origin begins growing at time $0$, but other lilypads only begin to
grow once they are touched by another lilypad. For convenience, we set
$h_T(z) = h_T([z]_T)$ for any point $z \in\R^d\setminus L_T$, where
\[
[z]_T = \biggl( \frac{\lfloor r(T) z_1 \rfloor}{r(T)}, \ldots, \frac
{\lfloor r(T) z_d \rfloor}{r(T)}
\biggr) \qquad\mbox{for any } z = (z_1, \ldots, z_d) \in
\Z^d.
\]
This system of hitting times is an interesting model in its own right,
describing a first passage percolation model on $\Z^d$.

Although there are no ``particles'' in this system of growing lilypads,
we define
\[
m_T(z,t) = \sup_{y \in L_T}\bigl\{
\xi_T(y) \bigl(t-h_T(y)\bigr)_+ - q\llvert z-y\rrvert
\bigr\},
\]
which we think of as the rescaled number of particles in the lilypad
model. We will show that with high probability its value matches very
closely that of $M_T(z,t)$.

We will give a heuristic explanation for these definitions in
Section~\ref{seheuristics}. For an idea of how the system evolves,
see Figures~\ref{figlily1} to~\ref{figlily6}, where we plot the
growth of the sites hit as time advances. A simulation of the process
can be seen at \surl{http://tiny.cc/lilypads}.

%
%f1 #&#
\begin{figure}

\includegraphics{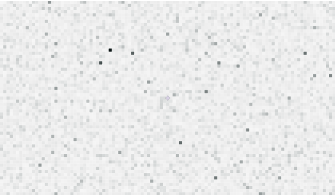}

\caption{The grey squares represent the potential: dark means large
potential. A lilypad starts to grow from near the origin.}\label{figlily1}
\end{figure}

%f2 #&#
\begin{figure}

\includegraphics{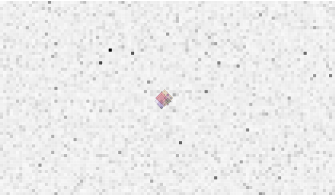}

\caption{Some more points of reasonable potential are hit and a number
of other visible lilypads are launched.}\label{figlily2}
\end{figure}

%f3 #&#
\begin{figure}

\includegraphics{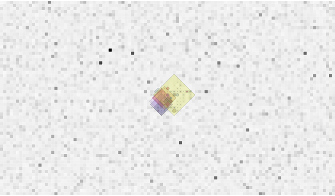}

\caption{A point of larger potential is hit and its lilypad, in
yellow, grows faster.}\label{figlily3}
\end{figure}

%f4 #&#
\begin{figure}

\includegraphics{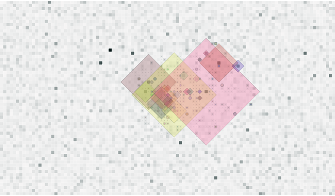}

\caption{Space is covered quickly and the dark spots in the top left
will soon be hit.}
\label{figlily4}
\end{figure}

%f5 #&#
\begin{figure}

\includegraphics{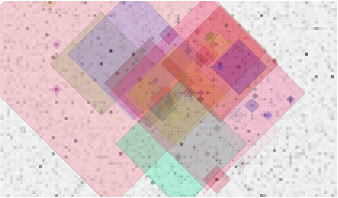}

\caption{The darkest spot is hit and launches a fast-growing pink lilypad.}
\label{figlily5}
\end{figure}

%f6 #&#
\begin{figure}[b]

\includegraphics{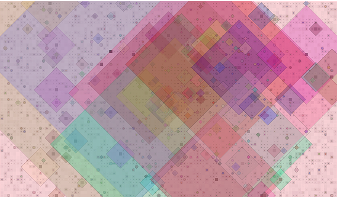}

\caption{The whole visible region is covered.}\label{figlily6}
\end{figure}

%
%th1.1 #&#
\begin{teo}[(Approximation by lilypad model)]\label{teolilypad} For any
$t_\infty> 0$,
as \mbox{$T \rightarrow\infty$},
\begin{eqnarray*}
&&\sup_{t \leq t_\infty} \sup_{z \in L_T} \bigl\llvert
M_T(z,t) - m_T(z,t) \bigr\rrvert \rightarrow0 \qquad
\mbox{in } \P\mbox{-probability}.
\end{eqnarray*}
Moreover, for any $R > 0$, as $T \rightarrow\infty$,
\[
\sup_{z \in L_T(0,R)} \bigl\llvert H_T(z) -
h_T(z) \bigr\rrvert \rightarrow0 \qquad\mbox {in } \P
\mbox{-probability}.
\]
\end{teo}

\begin{rmks*}
(1)~\emph{The lilypad model is well defined.} It is not a
priori clear that the hitting times $\{ h_T(z) \}$ are well defined.
However, we will show in Lemma~\ref{lesmalllily} that any finite
ball gets covered eventually by the lilypad model, and in Lemma~\ref
{le0805-1} that there is no explosion; that is, for any finite time
$t$ there exists $R>0$ such that the lilypad model is entirely
contained within $B(0,R)$ at time $t$.

(2) \emph{Interpretation as a first-passage percolation
model.} As mentioned above it is possible to interpret the lilypad
hitting times as a first-passage percolation model. We connect each
pair of vertices in $L_T$ via two directed edges. We associate to the
directed edge going from $x$ to $y$ the passage time $q \frac
{\llvert  x-y\rrvert  }{\xi_T(x)}$. Then $h_T(z)$ is the first passage time from $0$
to $z$.

(3) \emph{Convergence to a Poissonian model.} From extreme
value theory it can be shown that in a suitable sense
\[
\sum_{z \in L_T} \delta_{(z, \xi_T(z))} \Rightarrow
\Gamma,
\]
where $\Gamma$ is a Poisson point process on $\R^d \times\R^+$ with
intensity measure $d z \times\alpha x^{-(\alpha+1)} \,dx$. See~\cite
{HMS08} for precise statements and an application to the PAM. This
suggests that the lilypad model and therefore also the hitting times
and number of particles in the branching random walk should converge in
distribution to a version of the lilypad model defined in terms of the
Poisson point process. We will carry out the details of this analysis
in a further paper.
\end{rmks*}

%s1.3 #&#
\subsection{Applications of the lilypad model}
Theorem~\ref{teolilypad} tells us that the BRW is well approximated
by the lilypad model. We now describe some consequences of that
approximation. As an easy first application, we describe the support of
the branching random walk. For this we define
\[
S_T(t) = \bigl\{ z \in\R^d\dvtx  H_T
\bigl([z]_T\bigr) \leq t \bigr\} \quad\mbox{and}\quad
s_T(t) = \bigl\{ z \in\R^d\dvtx  h_T
\bigl([z]_T\bigr) \leq t \bigr\},
\]
which we think of as the support of the BRW and the lilypad model, respectively.

%
%th1.2 #&#
\begin{teo}\label{teosupport} If $d_{\mathrm{H}}$ denotes the Hausdorff
distance, for any $t_\infty>0$, as \mbox{$T \rightarrow\infty$},
\[
\sup_{t \leq t_\infty} d_\mathrm{H}\bigl( S_T(t),
s_T(t)\bigr) \rightarrow0 \qquad \mbox{in } \P\mbox{-probability}.
\]
\end{teo}

\begin{rmk*} Note that our definition of the support $S_T(t)$ is not the same
as
\[
\bigl\{ z \in\R^d\dvtx  Y\bigl(r(T)z, tT\bigr) \neq\varnothing\bigr\},
\]
which is the set of sites that are occupied at time $t$. For example,
$S_T(t)$ is
by definition always a connected set, since the underlying random walk
is nearest-neighbor, while the latter
set may be disconnected since particles can jump away from the bulk.
However, the two sets are almost the same in the following sense:
Theorem~\ref{teolilypad} tells us that
very shortly after a site has been visited by the BRW it will be
occupied by many particles, which ensure
that the site will be occupied from then onward.
\end{rmk*}

A more striking application of our description is that the BRW shows
intermittent
behavior: all the mass is concentrated around a single peak of the potential.

%
%th1.3 #&#
\begin{teo}\label{teointermittency} For $t>0$ let $w_T(t)$ be the
point in $L_T$ that maximizes $\{ m_T(z,t) \}$, where in the case of a tie
we choose arbitrarily. Then, for any fixed $t > 0$, and $\eps_T =\frac{3}q \log^{-1/4}T$,
\[
\frac{ \sum_{z \in L_T(w_T(t), \eps_T)} N(r(T)z,tT) }{ \sum_{z \in
L_T} N(r(T)z,tT) } \rightarrow 1 \qquad\mbox{in } \P\mbox{-probability}.
\]
\end{teo}

\begin{rmk*} \emph{One point localization and further extensions}. The
above theorem tells us that with high probability almost all of the
mass is contained within a small ball about $w_T(t)$. In fact, with
high probability almost all of the mass is contained actually at the
single site $w_T(t)$. Proving this is more difficult and will be
carried out in a further paper. Other, even more delicate results are
known for the behavior of the PAM, including the almost sure
fluctuations of the process (see \cite{KLMS09}), and we plan to
address the corresponding questions for the BRW in future work. We will
also postpone to future work a detailed description of further
properties that can be described by the lilypad model. These include a
description of genealogies of particles as well as aging for the process.
\end{rmk*}

%s1.4 #&#
\subsection{The parabolic Anderson model revisited}

We recall that the expected number of particles at a site $z$ at time $t$
is given by $u(z,t)$, the solution of the parabolic Anderson model.
As pointed out in the \hyperref[sec1]{Introduction}, the parabolic
Anderson model has been studied extensively, and a reader familiar with
the literature
will recognize that our predictions in terms of the lilypad model do
not resemble
those for the parabolic Anderson model. This raises the natural question
of how different the actual number of particles is from the expected number.

We will make this comparison more transparent
by first considering the support of the branching random walk. We
already know from
Theorem~\ref{teosupport} that the support is described by the lilypad model.
Without this description, a naive guess for the support of the BRW
would be that a site gets hit roughly as soon as the expected number of
particles, that is, the solution of the PAM, at that site becomes
larger than $1$. We show that this guess is dramatically wrong.

Previous work on the PAM has focused on showing, for example, one-point
localization, but
to understand the expected ``support'' we need information on the
growth at every site, not just those with large potential. It turns out
that by a simple version of our arguments for the BRW, we can also
describe the profile of the PAM.

For this, we define the growth rate of particles and ``hitting time''
at a site $z \in L_T$ for the PAM as
\[
\Lambda_T(z,t) = \frac{1}{a(T)T} \log_+ u\bigl(r(T) z,tT\bigr)
\]
and
\[
%\quad\mbox{and}\quad
\cT_T(z) = \inf\bigl\{ t \geq0\dvtx  u\bigl( r(T) z, tT
\bigr) \geq1 \bigr\}.
\]
In a similar fashion to the lilypad model for the BRW, we can define
the \emph{PAM lilypad model}
by specifying the ``number of particles'' as
\[
\la_T(z,t) = \sup_{ y \in L_T} \bigl\{
\xi_T(y) t - q \llvert y\rrvert - q \llvert z-y\rrvert \bigr\} \vee0.
\]
Moreover, the ``hitting time'' for the PAM lilypad model is given by
\[
\tau_T(z) = \inf_{y \in L_T} \biggl\{ q
\frac{\llvert  y\rrvert  }{\xi_T(y)} + q \frac{\llvert  z-y\rrvert  }{\xi_T(y) } \biggr\}.
\]
We can also describe the support of the PAM and its lilypad model,
which we define, respectively, as
\[
S_T^{\mathrm{PAM}}(t) = \bigl\{ z \in\R^d\dvtx
\cT_T\bigl([z]_T\bigr) \leq t \bigr\}\quad\mbox{and}\quad
s_T^{\mathrm{PAM}}(t) = \bigl\{ z \in\R^d\dvtx
\tau_T\bigl([z]_T\bigr) \leq t \bigr\}.
\]

%
%th1.4 #&#
\begin{teo}\label{teoPAM}
For any $R, t_\infty> 0$, the following hold as $T \rightarrow\infty$:
\begin{longlist}[(iii)]
\item[(i)] $\sup_{t \leq t_\infty} \sup_{z \in L_T} \llvert   \Lambda_T(z,t) -
\lambda_T(z,t) \rrvert   \rightarrow0$ in $\P$-probability.\vspace*{2pt}

\item[(ii)] $\sup_{z \in L_T(0,R)} \llvert   \cT_T(z) - \tau_T(z) \rrvert   \rightarrow0$
in $\P$-probability.\vspace*{2pt}

\item[(iii)] $\sup_{t \leq t_\infty} d_\mathrm{H}(S_T^{\mathrm{PAM}}(t), s_T^{\mathrm{PAM}}(t))
\rightarrow0$ in $\P$-probabililty.
\end{longlist}
\end{teo}

\begin{rmks*}
(1) One can show that
\[
\la_T(z,t) = \sup_{y \in L_T} \bigl\{
\xi_T(y) \bigl( t - \tau_T(y)\bigr)_+ - q \llvert z-y
\rrvert \bigr\},
\]
which is very similar to the definition of
\[
m_T(z,t) = \sup_{y\in L_T} \bigl\{
\xi_T(y) \bigl(t-h_T(y)\bigr)_+ - q\llvert z-y\rrvert
\bigr\}.
\]

(2) We stress that although Theorem~\ref{teoPAM} is a new
result, and we provide a short and self-contained proof, it could be
proved using existing PAM technology.
\end{rmks*}

%
%th1.5 #&#
\begin{teo}\label{teonotconnected}
\textup{(i)} The support of the branching random walk $S_T(t)$ is
connected at all times, while
the support of the PAM is disconnected in the sense that for any $t > 0$,
\[
\liminf_{T \rightarrow\infty} \P\bigl( S_T^{\mathrm{PAM}}(t)
\mbox{ is disconnected}\bigr) > 0.
\]

\textup{(ii)} Let $W_T(t)$ be the maximizer of the branching random walk
and $W_T^{\mathrm{PAM}}(t)$ the maximizer of
the parabolic Anderson model (where possible ties are resolved arbitrarily).
Then for any $t > 0$,
%
%e1 #&#
\begin{equation}
\label{eqmaxsame} \liminf_{T \rightarrow\infty} \P \biggl(\bigl\llvert
W_T(t) - W_T^{\mathrm{PAM}}(t)\bigr\rrvert \leq
\frac{3}q \log^{-1/4}T \biggr) > 0.
\end{equation}
At the same time, for any $\kappa> 0$ and $t>0$,
\[
\liminf_{T \rightarrow\infty} \P\bigl( \bigl\llvert W_T(t) -
W_T^{\mathrm{PAM}}(t)\bigr\rrvert \geq \kappa\bigr) > 0.
\]
\end{teo}

The explanation for this behavior is that in the PAM a site $z$ outside
the current support can have such a high potential value $\xi_T(z)$
that in expectation it becomes optimal to go straight to that site
despite the high cost. This leads to exponentially large values of the
expectation in areas disconnected from the rest of the support; see
Figure~\ref{figdisconnected} for an illustration.

%
%f7 #&#
\begin{figure}

\includegraphics{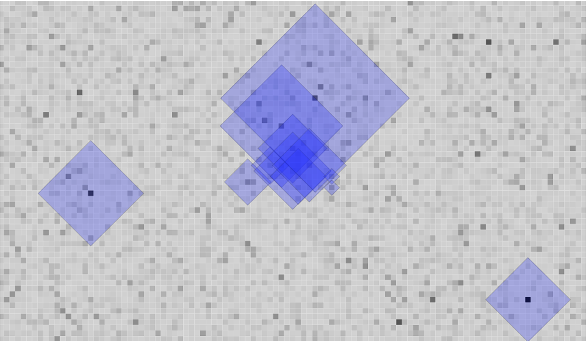}

\caption{The blue regions show the support of the PAM at a particular
time. Note that the support is disconnected.}\label{figdisconnected}
\end{figure}

However, the branching random walk can only spread at a speed that
depends on the values of the potential at sites that it has already
visited. Therefore its support remains connected and particles cannot
jump ahead to profit from larger values of the potential.

For the second part of the theorem we show that there are scenarios
when the BRW can catch up with the PAM. On the other hand, we can show
that there are times when the maximizers are spatially separated. See
also Figures~\ref{figsame},~\ref{figdifferent},~\ref{figsame1d}
and~\ref{figdifferent1d} for an illustration.

%
%f8 #&#
\begin{figure}

\includegraphics{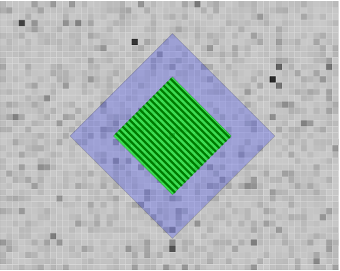}

\caption{Support of BRW (striped green) and PAM (blue) with same
maximizer in $\Z^2$.}\vspace*{12pt}
\label{figsame}
%\end{figure}
%
%%f9 #&#
%\begin{figure}

\includegraphics{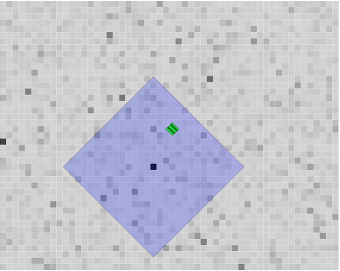}

\caption{Support of BRW (striped green) and PAM (blue) with different
maximizers in $\Z^2$.}\vspace*{12pt}
\label{figdifferent}
%\end{figure}
%
%%f10 #&#
%\begin{figure}

\includegraphics{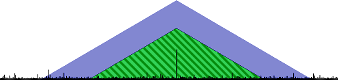}

\caption{Number of particles in BRW (striped green) and PAM (blue)
with same maximizer in $\Z$.}\vspace*{12pt}
\label{figsame1d}
%\end{figure}
%
%%f11 #&#
%\begin{figure}

\includegraphics{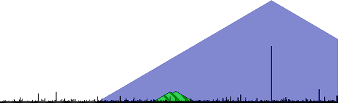}

\caption{Number of particles in BRW (striped green) and PAM (blue)
with different maximizers in~$\Z$.}\label{figdifferent1d}
\end{figure}

%s1.5 #&#
\subsection{Related work}
There are several natural ways of letting a random environment
influence the evolution of branching random walks. One possibility,
\mbox{introduced} in~\cite{DF83}, is to model spatial heterogeneity by
associating to each site a randomly sampled offspring distribution.
Alternatively, the offspring distribution can also vary in time as an
space--time i.i.d. sequence (see, e.g.,~\cite
{BGK05,Yoshida08,HY09,CY11}), or even both the motion of the particles
and the offspring distribution can be influenced by the environment;
see, for example,~\cite{CP07a}.\vadjust{\goodbreak}

Closely related to our model is a branching random walk on $\Z^d$ in
discrete time with a spatial i.i.d. offspring distribution. Here, much
more is known about the number of particles. In their early work on
this subject for $d=1$, Greven and den~Hollander \cite{GH92} and Baillon et al.~\cite
{BCGH93} start with an infinite population and describe the local
and global growth rates in terms of a variational problem (depending on
a drift in the underlying random walk). Many other authors address the
question of survival (see, e.g.,~\cite{BGK09,GMPV10}) and recurrence
vs. transience; see, for example,~\cite{CMP98,CP07a,Mueller08b,Mueller08a}.

Since our interest is in the effect of heavy-tailed environments, we
assume that the branching rates are bounded away from zero and thus
avoid the issue of recurrence and transience. Indeed we see that as
soon as a site is occupied, there are almost immediately exponentially
many particles, and we focus on analyzing the growth of the branching
process by describing when sites are hit and how the number of
particles evolves thereafter. We find that for our choice of potential
the sites that are hit, as well as the local growth rates, are---even
after rescaling appropriately---random.
Furthermore, we will show that in our case the growth rates for the
actual number of particles deviates dramatically from those for the
expected number.
This is in sharp contrast to existing shape theorems for branching
random walks with spatial i.i.d. offspring distribution; see~\cite
{CP07a,CP07b}. More precisely, in~\cite{CP07a} it is proved that under
a uniform ellipticity condition (and only in the recurrent
case) the
rescaled set of visited sites is well approximated by a deterministic,
convex and compact set almost surely.
This is strengthened in~\cite{CP07b} under the assumption that the
environment is uniformly elliptic, and the number of immediate offspring
is uniformly bounded.
Here the authors show that the local growth rates of the quenched
expectation (i.e., without averaging over the environment) as well as
that of the actual number of particles are described by the same
deterministic function.

Compared to work on the PAM, we see several similarities, most
prominently intermittency. However, there are also stark differences,
both in the results---Theorem~\ref{teonotconnected} gives a
snapshot, but this reflects just a small part of the contrast described
more fully by the clearly distinct lilypad models---and in the
methodology. Indeed, one of the main difficulties in studying
intermittency in the PAM is that one must control all possible ``good
islands,'' whereas for the BRW we must control not just all possible
good islands but all possible \emph{paths}, or sequences of good
islands. This increases the difficulty significantly and substantial
technological innovation is required.

%s1.6 #&#
\subsection{Heuristics}\label{seheuristics}
It is already known from work on the PAM that if we look at the
expected number of particles at each site in $\Z^d$, the system
essentially behaves as follows: the first particle chooses an optimal
site $z$ [which will be at distance of order $r(T)$ from the origin],
runs there in a short time (order $\ll T$) and sits there for time of
order $T$ to take advantage of the large potential at $z$.

Our first question is whether the branching system follows the same
tactic. The answer is no: the probability of one particle running
distance $r(T)$ in time $\ll T$ is extremely small, and so the behavior
outlined above is effectively impossible. In expectation there is no
problem since the enormous reward more than compensates for the small
probability of the event, but without taking expectations it is clear
that in order to cover large distances, we need to have lots of
particles already present in the system.

Suppose that we have some particles at a site $z$, and that $\xi
(z)=Aa(T)$. How long does it take those particles to reach another site
$y$? If $\llvert  z-y\rrvert  =Rr(T)$, where $R\asymp1$, then the probability\vspace*{1pt} that a
single random walk started at $z$ is at $y$ at time $tT$ is
approximately $e^{-a(T)T qR}$. (The dependence on $t$ is of smaller
order, which explains why particles in the PAM run large distances in
small times.) Thus we need of the order of $e^{a(T)T q R}$ particles at
$z$ before we can reach $y$. Particles breed at rate $\xi(z)$, so
ignoring\vspace*{1pt} the motion for a moment, by time $tT$ we should have of the
order of $e^{At a(T)T}$ particles at $z$. Thus we expect that it takes
time $t\approx qR/A$ to reach $y$ from $z$.

Given the calculations above, we are drawn to the idea that once a site
$z$ is hit, particles move outward from $z$ at speed proportional to
$\xi(z)$. We imagine a growing ``lilypad'' of particles centered at
$z$ and growing outward at a constant speed. Each site hit by $z$'s
lilypad then launches its own lilypad which grows at rate proportional
to its potential. Of course if $\xi(z)$ is large, then most of the
sites hit by $z$'s lilypad have smaller potential, so their lilypads
grow more slowly and have no discernible effect. Only when $z$'s
lilypad touches a point of greater potential do we start uncovering new
terrain at a faster rate.

In reality this does not accurately describe how particles behave
because if $\xi(z)$ is large, then particles wait at $z$ until the
last possible second before running quickly to their desired
destination. Besides this, our rough calculations required the
potential at $z$
and also the distance between $y$ and $z$ to be large.
In particular we should worry about the system at small times, since
when we start with one particle at the origin, there might be no points
of large potential nearby.
Nevertheless, this collection of deterministic (given the environment)
growing lilypads does give a caricature of the dynamics of the system
that is surprisingly accurate and useful.

Now suppose that we want to know when particles first hit a fixed site
$z$. In order to hit $z$, we must find a point $y_1$ of large potential
whose lilypad has touched~$z$. We must then ensure that $y_1$ is hit
sufficiently early, so we must find a suitable point $y_2$ whose
lilypad has touched $y_1$: working backward in this way, we construct a
sequence of points leading back toward the origin, and by looking at
their potentials---together with their positions relative to one
another---we can decide when $z$ should be hit.

%s1.7 #&#
\subsection{Organization of the paper}
We begin with some simple estimates on random walks and branching
processes in Section~\ref{seestimates}. In Section~\ref{selilypad}
we develop some initial estimates on the behavior of the system of
lilypads outlined above. We then move on, in Section~\ref{seearly},
to give upper bounds on the number of particles in the branching random
walk, and then provide lower bounds in Section~\ref{selate}. These
are tied together in Section~\ref{seproofs} to prove Theorems~\ref
{teolilypad},~\ref{teosupport} and~\ref{teointermittency}. The
relatively straightforward proof of Theorem~\ref{teoPAM} is given in
Section~\ref{sePAM}, and then in Section~\ref{secompare} we compare
the BRW with the PAM by proving Theorem~\ref{teonotconnected}.

%s1.8 #&#
\subsection{Frequently used notation and terminology}

We suppose that under $P^\xi$, and under an auxiliary probability
measure $P$, we have a simple random walk $(X(u))_{u\geq0}$, started
from $0$, independent of the environment and of the branching random
walk above.

We fix $c_d,C_d >0$ such that for any $R,T>0$ with $Rr(T)>1$,
\[
c_d R^d r(T) \leq\#L_T(0,R) \leq
C_d R^d r(T).
\]

Sometimes, for events $A$ and $B$, we say ``on $A$, $\P$-almost surely
$B$ occurs.'' By this we mean that $\P(A\cap B^c)=0$.

Given $v \in Y(t)$, and $s \leq t$, we write $X_v(s)$ for the position
of the unique ancestor of $v$ that was present
%alive
at time $s$.

At the end of the article we include a glossary of frequently used
notation for reference.

%s2 #&#
\section{Simple estimates on random walks and branching processes}\label{seestimates}
We collect here a few basic results that will be needed later. Lemmas
\ref{yulelem} and~\ref{2m} will be easy results about the growth of
branching processes, and Lemmas~\ref{rwlem},~\ref{rwlem2} and~\ref
{rwlem3} give us control over simple random walks. We also give a
Chernoff bound in Lemma~\ref{lechernoff} and an estimate on the
largest values of the potential in Lemma~\ref{le0306-1}.

First we check that branching processes do not grow much slower than
they should. The following result is very basic, but will still be
useful occasionally.

%
%le2.1 #&#
\begin{lem}\label{yulelem}
Let $(\Upsilon_t)_{t\geq0}$ be a Yule process (a continuous-time
Galton--Watson process with 2 children at every branch) branching at
rate $r$ under an auxiliary probability measure $P$. For any $r' < r$,
there exists a constant $c$ such that
\[
P\bigl(\Upsilon_t < \exp\bigl(r' t\bigr)\bigr) \leq
c \exp\bigl(\bigl(r'-r\bigr)t/2\bigr)\qquad\mbox{for all } t\geq0.
\]
\end{lem}

\begin{pf}
Let $T_0=0$, and for $n\geq1$ let $T_n$ be the $n$th birth time of the
process, and define $V_n = T_n - T_{n-1}$. Then the random variables
$(V_n, n\geq1)$ are independent, and $V_n$ is exponentially
distributed with parameter $rn$. Thus, using Markov's inequality,
\[
P(T_n > t) = P \Biggl(\sum_{j=1}^n
V_j > t \Biggr) \leq E\bigl[e^{\sklfrac {r}2\sum_{j=1}^n V_j}\bigr]e^{-rt/2}
\leq\prod_{j=1}^n \biggl(1-
\frac
{1}{2j}\biggr)^{-1} e^{-rt/2}.
\]
However,
\[
\prod_{j=1}^n \biggl(1-
 \frac{1}{2j} \biggr)^{-1} = \exp
\Biggl(-\sum_{j=1}^n
\log \biggl(1- \frac{1}{2j} \biggr) \Biggr) \leq
\exp \Biggl(\sum_{j=1}^n
\biggl( \frac{1}{2j} + \frac{1}{2j^2} \biggr)
\Biggr) \leq c n^{1/2}
\]
for some constant $c$. Taking $n = \lfloor e^{r' t}\rfloor$, we get
\[
P\bigl(\Upsilon_t < e^{r' t}\bigr) \leq P(T_n
> t) \leq c e^{r't/2 - rt/2}.
\]
\upqed
\end{pf}

Although the previous lemma is occasionally useful, we will need a
slightly different estimate in other places. Since our particles can
move around, it will often be more useful to be able to know that the
number of particles at a single site does not grow much slower than it should.

%
%le2.2 #&#
\begin{lem}\label{2m}
Suppose that $\xi(0)\geq4d$. Then $P^{\xi}(N(0,t) < \frac
{1}{2}e^{(\xi(0)-2d)t}) \leq15/16$ for all $t\geq0$.
\end{lem}

\begin{pf}
Let $\mathcal{N}(t)$ be the set of particles that have not left $0$ by
time $t$, so that
$N(0,t) \geq\mathcal{N}(t)$. Clearly $\mathcal{N}$ is a birth--death
process in which each particle breeds at rate $\xi(0)$ and dies at
rate $2d$. Note that $E^{\xi}[\mathcal N(t)] = e^{(\xi(0)-2d)t}$ and,
by the Paley--Zygmund inequality,
\[
P^{\xi}\biggl(\mathcal N(t) \geq\frac{1}{2}E^{\xi}\bigl[
\mathcal N(t)\bigr]\biggr) \geq \frac{E^{\xi}[\mathcal N(t)]^2}{4E^{\xi}[\mathcal N(t)^2]}.
\]
Thus it suffices to show that $E^{\xi}[\mathcal N(t)^2] \leq4e^{2(\xi
(0)-2d)t}$. By choosing $\delta>0$ small and conditioning on what
happens by time $\delta$, if $\mathcal N^1(t)$ and $\mathcal N^2(t)$
are independent copies of $\mathcal N(t)$, then
\begin{eqnarray*}
E^{\xi}\bigl[\mathcal N(t+\delta)^2\bigr] &=&
E^{\xi}\bigl[\mathcal N(t)^2\bigr]\bigl(1-\delta \bigl(
\xi(0)+2d\bigr)\bigr)
\\
&&{} + E^{\xi}\bigl[\bigl(\mathcal N^1(t)+\mathcal
N^2(t)\bigr)^2\bigr]\delta\xi(0) + O\bigl(
\delta^2\bigr)
\end{eqnarray*}
so
\[
\frac{d}{dt}E^{\xi}\bigl[\mathcal N(t)^2\bigr] =
\bigl(\xi(0)-2d\bigr)E^{\xi}\bigl[\mathcal N(t)^2\bigr] + 2
\xi(0)e^{2(\xi(0)-2d)t}.
\]
By solving this ODE we obtain
\[
E^{\xi}\bigl[\mathcal N(t)^2\bigr] = \frac{2\xi(0)}{\xi(0)-2d}e^{2(\xi
(0)-2d)t}+
\biggl(1-\frac{2\xi(0)}{\xi(0)-2d} \biggr)e^{(\xi(0)-2d)t},
\]
which, when $\xi(0)\geq4d$, is at most $4e^{2(\xi(0)-2d)t}$, as required.
\end{pf}

Recall that $X(t),t\geq0$ is a continuous-time random walk on $\Z^d$.
We give a lower bound on the probability that $X(sT)=r(T)z$. Define
\[
\calE^1_T(s,R) = \frac{R}{\log T} (\log R - \log s)
+ \frac{2\,ds}{a(T)}.
\]

%
%le2.3 #&#
\begin{lem}\label{rwlem}
For $z\in L_T$ and $s>0$, $T>e$,
\[
P\bigl(X(sT) = r(T)z\bigr) \geq\exp \bigl(-a(T)T\bigl(q\llvert z\rrvert +
\calE^1_T\bigl(s,\llvert z\rrvert \bigr)\bigr) \bigr).
\]
\end{lem}

\begin{pf}
Fix a path of length $r(T)\llvert  z\rrvert  $ from $0$ to $r(T)z$. To reach $r(T)z$,
it suffices to make exactly $r(T)\llvert  z\rrvert  $ jumps by time $t$, all along our
chosen path. Thus
\[
P\bigl(X(t) = z\bigr) \geq\frac{1}{(2d)^{r(T)\llvert  z\rrvert  }} e^{-2\,dsT}
\frac
{(2\,dsT)^{r(T)\llvert  z\rrvert  }}{(r(T)\llvert  z\rrvert  )!}.
\]
Using the fact that $n! \leq n^n = \exp(n\log n)$, the above is at least
\begin{eqnarray*}
&& \exp \bigl(-2\,dsT + r(T)\llvert z\rrvert \log(sT) - r(T)\llvert z\rrvert \log
\bigl(r(T)\llvert z\rrvert \bigr) \bigr)
\\
&&\qquad \geq\exp \bigl(-r(T)\llvert z\rrvert \log\bigl(T^q\llvert z\rrvert
/s\bigr) - 2\,dsT \bigr)
\\
&&\qquad  = \exp \biggl(-a(T)T\biggl(q\llvert z\rrvert + \frac{\llvert  z\rrvert  }{\log T}\bigl(\log
\llvert z\rrvert - \log s\bigr) + \frac{2\,ds}{a(T)}\biggr) \biggr).
\end{eqnarray*}
\upqed
\end{pf}

Now we need an upper bound on the probability that $X(t)=z$. In order
to reach~$z$, a random walk must jump at least $\llvert  z\rrvert  $ times, and this
bound will be enough for us, so for $s\geq0$ and $R>0$, define
\[
J_T(s,R) = P\bigl(X(u) \mbox{ jumps at least } Rr(T) \mbox{ times
before time } sT\bigr),
\]
and let
\[
\calE^2_T(s,R) = \frac{R}{\log T}\bigl(\log s - \log
R + 1 + \log(2d) + (q+1) \log\log T\bigr).
\]

%
%le2.4 #&#
\begin{lem}\label{rwlem2} For any $R>0$ and $s>0$, $T>e$,
\[
J_T(s,R) \leq\exp\bigl\{ - a(T) T \bigl( q R -
\calE^2_T(s,R) \bigr)\bigr\}.
\]
\end{lem}

\begin{pf} The number of jumps that $X(u)$ makes up to time $sT$ is
Poisson distributed with parameter $2\,dsT$, so that
\[
J_T(s,R) \leq\frac{(2\,dsT)^{Rr(T)}}{(R r(T))!}.
\]
By Stirling's formula $n! \geq\exp(n\log n - n)$, giving a new upper
bound of
\begin{eqnarray*}
&& \exp \bigl( Rr(T) \bigl(- \log\bigl( Rr(T)\bigr) + 1 + \log (2\,dsT)\bigr) \bigr)
\\
&&\qquad  = \exp \bigl( - q r(T) R\log T
\\
&&\hspace*{18pt}\quad\qquad{} + r(T) R \bigl( 1 + \log s + \log(2d)+ (q+1)\log\log T - \log R\bigr) \bigr)
\\
&&\qquad  = \exp \bigl( - a(T) T \bigl( q R - \calE^2_T(s,R)
\bigr) \bigr),
\end{eqnarray*}
where we used the definitions of $r(T)$ and $a(T)$ as well as $\calE^2_T$.
\end{pf}

Our third estimate on random walks is slightly different. Instead of
looking at the probability that a random walk moves a long way in a
relatively short time, we now want to ensure that the probability a
random walk moves a short distance in a relatively long time is
reasonably large. This is a consequence of a standard local central
limit theorem; see, for example, Theorem 2.1.3 of \cite{LL10}.

%
%le2.5 #&#
\begin{lem}\label{rwlem3}
There exists a constant $c>0$ such that provided $\llvert  z\rrvert  \leq\sqrt t$,
\[
P\bigl(X(t)=z\bigr) \geq c t^{-d/2}.
\]
\end{lem}

The following well-known version of the Chernoff bound will also be
very useful.

%
%le2.6 #&#
\begin{lem}\label{lechernoff}
Suppose that $Z_1,\ldots,Z_k$ are independent Bernoulli random
variables, and let $Z = \sum_{i=1}^r Z_i$. Then
\[
P \biggl(Z\leq\frac{E[Z]}{2} \biggr)\leq\exp \biggl(-\frac{E[Z]}{8}
\biggr).
\]
\end{lem}

Finally, we give some simple estimates on the maximum of the
environment within a ball. For $R>0$, let
\[
\bar\xi_T(R) = \max_{z\in L_T(0,R)} \xi_T(z).
\]

%
%le2.7 #&#
\begin{lem}\label{le0306-1}
\textup{(i)} For any $T>e$, any $R>0$ and $\nu>0$,
\[
\P \bigl( \bar\xi_T(R) \leq\nu \bigr) \leq e^{- c_d R^d \nu
^{-\alpha}}.
\]

\textup{(ii)} Provided that $Rr(T)\geq1$, for any $N\geq1$ and any
$\nu>0$,
\[
\P \bigl( \#\bigl\{z \in L_T(0,R)\dvtx  \xi_T(z) \geq\nu
\bigr\} \geq N \bigr) \leq \biggl(\frac{C_d e R^d \nu^{-\alpha}}{N} \biggr)^N.
\]
\end{lem}

\begin{pf}
(i) We may assume without loss of generality that $\nu a(T)\geq1$;
otherwise all points satisfy $\xi_T(z)>\nu$. By independence,
\begin{eqnarray*}
\P \bigl( \bar\xi_T(R) \leq\nu \bigr) & =& \P\bigl(\xi
_T(0) \leq\nu\bigr)^{\#L_T(0,R)}
\\
& \leq&\bigl(1-\bigl(\nu a(T)\bigr)^{-\alpha}\bigr)^{c_d R^d r(T)^d} \leq
e^{ - c_d R^d
\nu^{-\alpha}},
\end{eqnarray*}
where the last inequality follows from the inequality $1-x\leq e^{-x}$
and the fact that $r(T)^d = a(T)^\alpha$.

(ii) The number of points in $L_T(0,R)$ with (rescaled) potential
larger than $\nu$ is dominated by a binomial random variable with $C_d
e R^d r(T)^d$ trials of success probability $\nu^{-\alpha
}a(T)^{-\alpha}$. Thus
\begin{eqnarray*}
\P \bigl( \#\bigl\{z \in L_T(0,R)\dvtx  \xi_T(z) \geq\nu
\bigr\} \geq N \bigr) &\leq& \pmatrix{C_d R^d r(T)^d
\cr N} \nu^{-N\alpha}a(T)^{-N\alpha}
\\
&\leq&\frac{(C_d R^d r(T)^d)^N}{N!} \nu^{-N\alpha} a(T)^{-N\alpha}.
\end{eqnarray*}
Since $r(T)^d = a(T)^{\alpha}$, and $N!\geq N^N e^{-N}$, we get the result.
\end{pf}

%s3 #&#
\section{First properties of the lilypad model}\label{selilypad}
There are several fairly simple facts about the environment that will
be useful to us later. We begin with some almost self-evident
observations in Section~\ref{sealtform} that nonetheless take some
time to prove rigorously: Lemma~\ref{leattain} tells us that the
infimum in the definition of our lilypad hitting times $h_T(z)$ is
attained, Lemma~\ref{leequiv} proves that the infimum may be broken
up into more manageable chunks, while Lemma~\ref{leincrease} records
properties of the potential along an optimal
path. In Section~\ref{selilypadbounds} we bound the growth of our
lilypad model: Lemma~\ref{lesmalllily} gives upper bounds on the
time required to cover a ball about the origin, and Corollary~\ref
{cosmalllily} gives a more explicit bound on the time to cover a
particular small ball. Then Lemma~\ref{le0805-1} ensures that the
lilypad model does not quickly exit large balls and thus does not
explode in finite time. We will use many of these results often,
usually without reference.

%s3.1 #&#
\subsection{An alternative formulation for hitting times in the lilypad model}\label{sealtform}

As mentioned earlier, we want to show that the hitting times in the
lilypad model have two equivalent formulations. We define
\[
h_T(z) = \mathop{\inf_{ y_0,\ldots,y_n \in L_T\dvtx }}_{y_0 = z, y_n = 0}
\Biggl\{ \sum_{j=1}^{n} q
\frac{\llvert   y_{j-1} - y_{j}\rrvert  }{\xi_T(y_{j})} \Biggr\}
\]
and claim that
\[
h_T(z) = \inf_{y\neq z} \biggl\{h_T(y)
+ q\frac{\llvert  z-y\rrvert  }{\xi_T(y)} \biggr\}.
\]
First we check that the infimum is attained.

%
%le3.1 #&#
\begin{lem}\label{leattain}
For any $T>0$, the infimum in the definition of $h_T(z)$ is attained
for some sequence $y_0,\ldots,y_n$, $\P$-almost surely.
\end{lem}

\begin{pf}
Note that for $\lambda= \frac{1}2 \frac{\alpha-d}{\alpha}$, by the
definition of $a(T)$ and $r(T)$
\begin{eqnarray*}
\P\bigl(\exists y\in B\bigl(0,e^k\bigr)\dvtx  \xi_T(y)>
e^{(1-\lambda
)k} \bigr) & \leq& C_d \bigl(r(T)e^{k}
\bigr)^d \bigl(a(T)c_k e^{(1-\lambda)k}
\bigr)^{-\alpha}
\\
& =& C_d e^{-\sklfrac{1}2 (\alpha-d) k}.
\end{eqnarray*}
Therefore, by the Borel--Cantelli lemma, there exists $K$ such that for
all $k\geq K$,  $\max_{y\in B(0,e^k)} \xi_T(y) \leq e^{(1-\lambda
)k}$. By increasing $K$ if necessary we may assume also that $\frac
{h_T(z)+1}{q} \leq e^{\lambda K}$.
Suppose for contradiction that there exists a sequence
$(y_j)_{j=0,\ldots,n}$ with $y_0 = z$ and $y_n = 0$ such that for at
least one $j$,
$y_j \notin L_T(0, e^K)$ and
\[
\sum_{j=0}^{n} q\frac{\llvert  y_{j-1}-y_j\rrvert  }{\xi_T(y_j)} \leq
h_T(z) + \frac{1}2.
\]
Define $\ell= \max\{j\dvtx  y_j \notin B(0,e^K)\}$, so that by assumption
$\ell\in\{0, \ldots, n-1\}$. Then by the triangle inequality,
\begin{eqnarray*}
h_T(z) + \frac{1}2 &\geq& \sum_{j=\ell+1}^n
q\frac{\llvert  y_{j-1}-y_j\rrvert  }{\xi
_T(y_j)} \geq q \frac{ \llvert  y_\ell\rrvert  }{ \max_{y \in L_T(0, e^{K})} \xi_T(y)} \geq q \frac{ e^{K}}{e^{(1-\lambda)K}}
\\
&=&
qe^{\lambda K}.
\end{eqnarray*}
This contradicts our choice of $K$, and we deduce that
\[
h_T(z) = \mathop{\inf_{y_0,\ldots,y_n \in L_T(0,e^{K})\dvtx }}_{y_0 = z,
y_n = 0}
\Biggl\{ \sum_{j=1}^{n} q
\frac{\llvert   y_{j-1} - y_{j}\rrvert  }{\xi
_T(y_{j})} \Biggr\}.
\]
This infimum is over a finite set, so the minimum is attained.
\end{pf}

We can now prove our alternative formulation of the hitting times.

%
%le3.2 #&#
\begin{lem}\label{leequiv}
$\P$-almost surely, for any $z\neq0$ and $T>0$,
\[
h_T(z) = \inf_{y\neq z} \biggl\{h_T(y)
+ q\frac{\llvert  z-y\rrvert  }{\xi_T(y)} \biggr\}.
\]
\end{lem}

\begin{pf}
Fix $z\neq0$. First suppose there exists $y$ such that
\[
h_T(y) + \frac{\llvert  z-y\rrvert  }{\xi_T(y)} < h_T(z).
\]
Then by Lemma~\ref{leattain}, there exist $n$ and $y_0,\ldots,y_n$
such that $y_0=y$, $y_n=0$ and $h_T(y) = \sum_{j=1}^n
q\llvert  y_{j-1}-y_j\rrvert  /\xi_T(y_j)$. Defining $y'_0 = z$ and for $i=0,\ldots,n$ letting $y'_{i+1} = y_i$, we have by definition of $h_T(z)$
\[
h_T(z)\leq\sum_{j=1}^{n+1} q
\frac{\llvert  y'_{j-1} - y'_j\rrvert  }{\xi_T(y'_j)} = h_T(y) + q\frac{\llvert  z-y\rrvert  }{\xi_T(y)} <
h_T(z).
\]
This is a contradiction, so we have established that
\[
h_T(z) \leq\inf_{y\neq z} \biggl
\{h_T(y) + q\frac{\llvert  y-z\rrvert  }{\xi
_T(y)} \biggr\}.
\]
For the opposite inequality, choose (by Lemma~\ref{leattain}) $n$ and
distinct $y_0,\ldots,y_n$ such that $y_0 = z$, $y_n=0$ and
\[
h_T(z) = \sum_{j=1}^n q
\frac{\llvert  y_{j-1}-y_j\rrvert  }{\xi_T(y_j)}.
\]
We claim that $h_T(y_1) = \sum_{j=2}^n q \llvert  y_{j-1}-y_j\rrvert  /\xi_T(y_j)$.
If not, then there exist $m$ and $x_0,\ldots,x_m$ such that $x_0=y_1$,
$x_m=0$ and
\[
h_T(y_1) = \sum_{j=1}^m
q\frac{\llvert  x_{j-1}-x_j\rrvert  }{\xi_T(x_j)} < \sum_{j=2}^n q
\frac{\llvert  y_{j-1}-y_j\rrvert  }{\xi_T(y_j)}.
\]
Let $y'_0 = y_0 = z$, $y'_1 = y_1$, and for $j=2,\ldots,m+1$, $y_j' =
x_{j-1}$. Then
\begin{eqnarray*}
h_T(z) & \leq&\sum_{i=1}^{m+1}
q\frac
{\llvert  y'_{i-1}-y'_i\rrvert  }{\xi_T(y'_i)} = q\frac{\llvert  y_0-y_1\rrvert  }{\xi_T(y_1)} + \sum_{j=1}^m
q\frac{\llvert  x_{j-1}-x_j\rrvert  }{\xi_T(x_j)}
\\
&<& \sum_{j=1}^n q\frac{\llvert  y_{j-1}-y_j\rrvert  }{\xi_T(y_j)} =
h_T(z).
\end{eqnarray*}
This is a contradiction, so our claim that $h_T(y_1) = \sum_{j=2}^n q
\llvert  y_{j-1}-y_j\rrvert  /\xi_T(y_j)$ holds. Then
\[
h_T(z) = h_T(y_1) + q\frac{\llvert  z-y_1\rrvert  }{\xi_T(y_1)},
\]
which completes the proof.
\end{pf}

%
%le3.3 #&#
\begin{lem}\label{leincrease}
Let $T>0$, $z \in L_T$, and suppose that
\[
h_T(z) = \sum_{j = 1}^n q
\frac{\llvert  y_{j-1} - y_j\rrvert  }{\xi_T(y_j)},
\]
for distinct points $y_j, j = 1, \ldots, n$ with $y_0 = z$ and $y_n = 0$.
Then, for any $k \in\{1, \ldots, n\}$,
\[
h_T(y_k) = \sum_{j = k+1}^n
q \frac{\llvert  y_{j-1} - y_j\rrvert  }{\xi_T(y_j)}.
\]
Moreover, the sequence $(\xi_T(y_j), j \geq1)$ is nonincreasing.
\end{lem}

\begin{pf} We have already shown in the proof of Lemma~\ref{leequiv} that
\[
h_T(y_1) = \sum_{j = 2}^n
q \frac{\llvert  y_{j-1} - y_j\rrvert  }{\xi_T(y_j)}.
\]
Iterating the argument shows the first statement.

For the second statement, suppose that there exists $k \in\{1,\ldots,
n-1\}$ such that
$\xi_T(y_k) < \xi_T(y_{k+1})$. We show that it is then faster to
reach $z$ without traveling
via~$y_k$. Indeed, by the triangle inequality
\begin{eqnarray*}
&&\sum_{j \in\{1, \ldots,n \}, j \neq k, k+1} q \frac{\llvert  y_{j-1} -
y_j\rrvert  }{\xi_T(y_j)} + q
\frac{\llvert  y_{k-1} - y_{k+1}\rrvert  }{\xi_T(y_{k+1})}
\\
&&\qquad \leq\sum_{j \in\{1, \ldots,n \}, j \neq k, k+1} q \frac{\llvert  y_{j-1} -
y_j\rrvert  }{\xi_T(y_j)} + q
\frac{\llvert  y_{k-1} - y_k \rrvert   + \llvert  y_{k} - y_{k+1}\rrvert  }{\xi
_T(y_{k+1})}
\\
&&\qquad < \sum_{j = 1}^n q \frac{\llvert  y_{j-1} - y_j\rrvert  }{\xi_T(y_j)} =
h_T(z),
\end{eqnarray*}
contradicting the definition of $h_T(z)$.
\end{pf}

%s3.2 #&#
\subsection{Bounding the lilypad model}\label{selilypadbounds}

We want to make sure that the lilypad model behaves relatively
sensibly: that small balls are covered quickly, but large balls are
not. We begin with the former statement; more precisely, we show that
for any time $t>0$, we can find a radius $R>0$ such that with high
probability, $B(0,R)$ is covered by time $t$. For $R>0$, let
\[
\bar h_T(R) = \sup_{z\in B(0,R)}h_T(z).
\]

%
%le3.4 #&#
\begin{lem}\label{lesmalllily} For all $k\in\N$, $T>e$ and $\gamma
\in(d/\alpha,1)$,
\[
\P \biggl( \bar h_T\bigl(2^{-k}\bigr) >
\frac{4q}{1-2^{\gamma-1}}2^{(\gamma
-1)k} \biggr) \leq\sum_{j=k}^{\infty}
e^{-c_d 2^{(\alpha\gamma-d)j}}.
\]
Moreover, for any $R>0$, $T>e$ and $\gamma\in(d/\alpha,1)$,
\[
\P \biggl( \bar h_T(R) > 2^{\gamma(k+1)}qR +
\frac{4q}{1-2^{\gamma
-1}}2^{(\gamma-1)k} \biggr) \leq\sum_{j=k}^{\infty}
e^{-c_d
2^{(\alpha\gamma-d)j}}.
\]
\end{lem}

\begin{pf}
Let $B_k = L_T(0,2^{-k})$. Define $A_k = \{ \exists z \in B_k\dvtx  \xi
_T(z) \geq2^{-\gamma k} \}$.
Then by Lemma~\ref{le0306-1}(i), for any $T>e$,
\[
\P\bigl(A_k^c\bigr) \leq e^{-c_d 2^{(\alpha\gamma-d )k}}.
\]
Thus
\[
\sum_{j = k+1}^{\infty} \P\bigl(A_j^c
\bigr) \leq\sum_{j=k+1}^{\infty} e^{-
c_d 2^{(\alpha\gamma-d)j}}.
\]
However, if $A_{k+1}, A_{k+2},\ldots$ all occur, then we may choose
$y_j \in B(0,2^{-(k+j)})$ such that $\xi_T(y_j) \geq2^{-\gamma
(k+j)}$ for each $j\geq1$. Clearly there exists $n$ such that $y_j=0$
for all $j\geq n$. Take $z\in B(0,2^{-k})$, and let $y_0 = z$. Then for
$j\geq1$,
\[
q\frac{\llvert  y_{j-1}-y_j\rrvert  }{\xi_T(y_{j})} \leq q\frac{2\cdot
2^{-(k+j-1)}}{2^{-\gamma(k+j)}} = 4q\cdot2^{(\gamma-1)(k+j)},
\]
so by the definition of $h_T(z)$,
\[
h_T(z) \leq4q\sum_{j=1}^n
2^{(\gamma-1)(k+j)} \leq\frac
{4q}{1-2^{\gamma-1}}2^{(\gamma-1)k}
\]
which proves our first claim. For the second, it suffices to observe
that if $x\in B(0,R)$ and we can find the above sequence
$y_1,y_2,\ldots,$ then by Lemma~\ref{leequiv}
\begin{eqnarray*}
h_T(x) &\leq &h_T(y_1) + q
\frac{\llvert  x-y_1\rrvert  }{\xi_T(y_1)}  \leq 4q\sum_{j=2}^n
2^{(\gamma-1)(k+j)} + q\frac{R+2^{-(k+1)}}{2^{-\gamma
(k+1)}}
\\
& \leq&2^{\gamma(k+1)}qR + \frac{4q}{1-2^{\gamma-1}}2^{(\gamma-1)k}.
\end{eqnarray*}
\upqed
\end{pf}

By choosing $\gamma= (d/\alpha+1)/2$ and $k=\frac{\psi+1}{(1-\gamma
)\log2}\log\log T$, we get the following corollary.

%
%co3.5 #&#
\begin{cor}\label{cosmalllily}
For large $T$ and any $\psi>0$,
\[
\P \biggl( \bar h_T\bigl(\log^{-2(\psi+1)(q+1)}T\bigr) >
\frac{4q}{1-2^{\gamma
-1}}\log^{-(\psi+1)}T \biggr) \leq T^{-1}.
\]
\end{cor}

Now we show that conversely, we can find a radius $R>0$ such that with
high probability the lilypad model does not exit $B(0,R)$ by time $t$.

%
%le3.6 #&#
\begin{lem}\label{le0805-1}
For any $t> 0$, provided that $Rr(T)\geq1$,
\[
\bigl\{\exists z\in L_T\setminus B(0,R)\dvtx  h_T(z)
\leq t\bigr\} \subseteq \Bigl\{ \max_{y\in L_T(0,R)}
\xi_T(y) \geq qR/t \Bigr\}.
\]
As a result,
\[
\P\bigl( \exists z \in L_T\setminus B(0,R)\dvtx  h_T(z)
\leq t\bigr) \leq C_d e q^{-\alpha}R^{d-\alpha}t^\alpha.
\]
\end{lem}

%
%re1 #&#
\begin{rmk*}
In particular, the lilypad model does not explode in finite time.
\end{rmk*}

\begin{pf*}{Proof of Lemma \ref{le0805-1}}
Let $D_T = L_T(0,R+1)\setminus L_T(0,R)$, the boundary of $L_T(0,R)$.
Let $\tilde z$ be the point with smallest lilypad hitting time in
$D_T$, that is, $h_T(\tilde z) = \min_{y \in D_T} h_T(y)$. Then from
the definition of $h_T$, any point $z\in L_T\setminus B(0,R)$ satisfies
$h_T(z) \geq h_T(\tilde z)$. For the same reason, in the definition of
$h_T(\tilde z)$, we can restrict the infimum to points $y_i$ within
$L_T(0,R)$ so that if we set $y_0 = \tilde z$,
\begin{eqnarray*}
h_T(\tilde z) & =& \mathop{\inf_{y_1,\ldots,y_n\in L_T(0,R)\dvtx }}_{y_n =
0}
\sum_{j = 1}^n q \frac{\llvert  y_{i-1} - y_i\rrvert  }{\xi_T(y_{i})} \geq q
\frac{\llvert  \tilde z\rrvert  }{\max_{y\in L_T(0,R)} \xi_T(y) }
\\
& \geq&\frac{qR}{\max_{y\in L_T(0,R)} \xi_T(y) }.
\end{eqnarray*}
Then, by the above estimate, we have that
\[
\bigl\{\exists z \in L_T\setminus B(0,R)\dvtx  h_T(z)
\leq t\bigr\} \subseteq\bigl\{ h_T(\tilde z) \leq t \bigr\}
\subseteq \biggl\{\max_{y\in L_T(0,R)} \xi _T(y) \geq
\frac{qR}{t} \biggr\}.
\]
Lemma~\ref{le0306-1}(ii), with $N=1$, then tells us that
\[
\P \biggl(\max_{y\in L_T(0,R)} \xi_T(y) \geq
\frac{qR}{t} \biggr) \leq C_d e q^{-\alpha}
R^{d-\alpha} t^{\alpha}
\]
which gives the desired bound.
\end{pf*}

%s4 #&#
\section{Upper bounds}\label{seearly}
In this section we come back to the branching random walk. We will
check that particles do not arrive anywhere earlier than they should and---as a consequence---that the number of particles at any site is
not too large. Our main tool will be the many-to-one lemma, also known
as the Feynman--Kac formula, which we introduce in Section~\ref
{semto}. We then apply this to bound the hitting times in terms of an
object $G_T$, which we go on to study in Section~\ref{seG}. The
tactic will be to give a recursive bound on $G_T$ along any sequence of
points of increasing potential, and then in Section~\ref
{secparameters} we fix a particular sequence and calculate the
resulting estimate. Finally we apply this hard work in Sections~\ref
{sesubearly} and~\ref{setoomany} to show, respectively, that
particles do not arrive early and that there are not too many particles.

%s4.1 #&#
\subsection{The many-to-one lemma}\label{semto}
We introduce a standard tool, sometimes called the many-to-one lemma
(in the branching process literature) and sometimes the Feynman--Kac
formula (in the Parabolic Anderson and statistical physics literature).
It gives us a way of calculating expected numbers of particles in our
branching random walk by considering the behavior of a single random
walk. Recall that under $P^{\xi}$, $(X(u))_{u\geq0}$ is a simple
random walk on $\Z^d$ independent of our branching random walk.

%
%le4.1 #&#
\begin{lem}[(Many-to-one lemma/Feynman--Kac formula)]\label{lemto}
If $f$ is measurable, then $\Prob$-almost surely, for any $s>0$,
\[
E^{\xi} \biggl[\sum_{v\in Y(s)} f\bigl(
\bigl(X_v(u)\bigr)_{u\in[0,s]}\bigr) \biggr] = E^{\xi}
\biggl[\exp \biggl(\int_0^s \xi\bigl(X(u)
\bigr) \,du \biggr) f\bigl(\bigl(X(u)\bigr)_{u\in[0,s]}\bigr) \biggr].
\]
\end{lem}

The interested reader may find a proof in \cite{ikeda1969branching},
or for a more modern approach \cite{HR11}. Both references give far
more general versions of this lemma, and in fact we will need one such
generalization. It is not too surprising, given that the many-to-one
lemma involves the equality of two expectations, that there is a
martingale hidden away here. For the more general version, essentially
we want to stop the martingale at a stopping time, rather than at a
fixed time $s$; but while the concept of a stopping time is simple
enough for our single random walk $(X(u))_{u\geq0}$, we need something
a bit more general for our branching random walk. This is where the
concept of a \emph{stopping line} enters. There is a whole theory
built around this idea, but we will need only the simplest part of it,
which can be deduced rather easily, avoiding a detailed discussion.
Indeed, fix $T>1$ and a point $z\in L_T$, and imagine that any particle
that hits $r(T)z$ is absorbed there, alive but no longer moving or
breeding. When working with this alternative system, we will attach a
superscript~$\sim$ to our notation, so, for example, $\tilde Y(s)$ will
be the set of particles present at time $s$ in the alternative system.

We make two observations about the alternative system. First, the
many-to-one lemma still holds, but since particles stop breeding as
soon as they hit $z$, if we define
\[
H^*_T(z) = \inf\bigl\{t\geq0\dvtx  X(tT) = r(T)z\bigr\},
\]
then we have
\begin{eqnarray*}
&& E^{\xi} \biggl[\sum_{v\in\tilde Y(tT)}  f\bigl(\bigl(
\tilde X_v(u)\bigr)_{u\in[0,tT]}\bigr) \biggr]
\\
&&\qquad  = E^{\xi} \biggl[\exp \biggl(T\int_0^{\tilde H^*_T(z)\wedge t}
\xi \bigl(\tilde X(uT)\bigr) \,du \biggr) f\bigl(\bigl(\tilde X(u)
\bigr)_{u\in[0,tT]}\bigr) \biggr],
\end{eqnarray*}
where $(\tilde X(u))_{u \geq0}$ is a simple random walk absorbed at
$r(T)z$. Second, notice that we may take the obvious coupling so that
the two systems are identical until $H_T(z)$; in particular $H_T(z)\leq
t$ if and only if $\tilde H_T(z)\leq t$.

Before we apply our two observations to prove a key lemma, we will introduce
some notation that we will use throughout this section.
We will work with a fixed large $T$, distinct points $z_1, z_2,\ldots
\in L_T$ and constants $t,t_1,t_2,\ldots\in\R^+$.
We are interested in bounding the event
\[
%\label{Adef}
A_T(j,z,t) = \left\{ \matrix{ H_T(z) \leq
t, H_T(z_i)\geq H_T(z)\wedge
t_i~\forall i\leq j,
\vspace*{3pt}\cr
H_T(z_i)\geq
H_T(z)~\forall i>j} \right\}.
\]
Similarly, we define $A^*_T(j,z,t)$ for the random walk by replacing
$H_T$ by $H^*_T$ in the above definition.
Informally, $A_T(j,z,t)$ [resp., $A^*_T(j,z,t)$] is the event that $z$
is hit by time $t$ by the rescaled branching random walk (resp., the
rescaled random walk), none of the $z_{i}, i>j$ are hit before $z$ and
those $z_i, i\leq j$ that are hit before $z$ are not hit before~$t_i$.

We will bound the probability of the event $A_T(j,z,t)$ in terms of the
following key quantity:
%
%e2 #&#
\begin{equation}
\label{eqGdef} G_T(j,z,s,t) = E^{\xi} \biggl[\exp \biggl(T
\int_0^{H^*_T(z)\wedge s} \xi\bigl(X(uT)\bigr) \,du \biggr)
\1_{A^*_T(j,z,t)} \biggr],
\end{equation}
where the need for an extra parameter $s \geq0$ becomes apparent later on.

%
%le4.2 #&#
\begin{lem}\label{leMtO}
$\P$-almost surely, for any $T>0$, distinct points $z_1, z_2,\ldots
\in L_T$, any $z\in L_T$, any $t,t_1,t_2,\ldots\in\R$ and any $j\geq0$,
%\begin{eqnarray*}
%
\[
P^{\xi}\bigl( A_T(j,z,t) \bigr) \leq E^{\xi}
\biggl[\exp \biggl(T\int_0^{H^*_T(z)} \xi\bigl(X(uT)
\bigr)\,du \biggr) \1_{A_T^*(j,z,t)} \biggr] = G(j,z,t,t).
\]
\end{lem}

\begin{pf}
All statements below hold $\P$-almost surely. By our second
observation above,
\begin{eqnarray*}
&& P^{\xi}\bigl(H_T(z)\leq t, H_T(z_i)
\geq H_T(z)\wedge t_i~\forall i\leq j,
H_T(z_i)\geq H_T(z)~\forall i>j\bigr)
\\
&&\qquad = P^{\xi}\bigl(\tilde H_T(z)\leq t, \tilde
H_T(z_i)\geq\tilde H_T(z)\wedge
t_i ~\forall i\leq j, \tilde H_T(z_i)
\geq\tilde H_T(z)~\forall i>j\bigr).
\end{eqnarray*}
Now, if some particle is to hit $z$ without hitting any of the $z_i$
too early (where ``too early'' is interpreted appropriately depending
on whether $i\leq j$), there must be a first particle to do so; so
writing $H^v_T(z)$ for the (rescaled) first time that particle $v$---or
one of its ancestors or descendants---hits $z$,
\begin{eqnarray*}
&&P^{\xi}\bigl(\tilde H_T(z) \leq t, \tilde
H_T(z_i)\geq\tilde H_T(z)\wedge
t_i ~\forall i\leq j, \tilde H_T(z_i)
\geq\tilde H_T(z)~\forall i>j\bigr)
\\
&&\qquad  \leq P^{\xi}\bigl(\exists v\in\tilde Y(tT)\dvtx  \tilde
H^v_T(z) \leq t, \tilde
H^v_T(z_i)\geq\tilde
H^v_T(z)\wedge t_i~\forall
i\leq j,
\\
&&\hspace*{179pt} \tilde H^v_T(z_i)
\geq\tilde H^v_T(z)~\forall i>j\bigr)
\\
&&\qquad  \leq E^{\xi} \biggl[\sum_{v\in\tilde Y(tT)}
\1_{\{\tilde H^v_T(z)
\leq t, \tilde H^v_T(z_i)\geq\tilde H^v_T(z)\wedge t_i
~\forall i\leq j, \tilde H^v_T(z_i)\geq\tilde
H^v_T(z)~\forall i>j\}} \biggr].
\end{eqnarray*}
We now apply the many-to-one lemma for the alternative system to the
last expectation to see that it equals
\[
E^{\xi} \biggl[\exp \biggl(T\int_0^{\tilde H^*_T(z)}
\xi \bigl(X(uT)\bigr)\,du \biggr) \1_{\biggl\{\fontsize{8.36}{12}\selectfont{\matrix{\tilde H^*_T(z) \leq t, \tilde
H^*_T(z_i)\geq\tilde H^*_T(z)\wedge t_i ~\forall i\leq
j,\cr \tilde H^*_T(z_i)\geq\tilde H^*_T(z)~\forall i>j }}\biggr\}} \biggr].
\]
But now we can use our second observation again to remove all the $\sim
$ superscripts from the above statement and deduce the desired result.
\end{pf}

Our first aim is to show that the probability that we hit $z$ early is small.
We use Lemma~\ref{leMtO}, and our tactic is to let $z_1,z_2,\ldots$
be the set of points of large potential in increasing order of $\xi$,
and work by induction on the largest $j$ such that we hit $z_j$ before
$z$. It will then be important how long we spend at $z_j$, and to
control this we will need to ``decouple'' the time in the $\exp(\cdot
)$ part of Lemma~\ref{leMtO} from the time in the indicator function,
which explains the extra parameter $s$ in the definition of~$G_T$.
We will concentrate on bounding $G_T$ in the next section, but a clue
as to how we will use it comes via the following easy corollary of
Lemma~\ref{leMtO}.

%
%co4.3 #&#
\begin{cor}\label{coPG}
$\P$-almost surely, for any $z\in L_T$, any $s,t \geq0$ and any
\mbox{$j\geq0$},
\begin{eqnarray*}
&&P^{\xi} \bigl(H_T(z) \leq s\wedge t, H_T(z_i)\geq H_T(z)\wedge
t_i ~\forall i\leq j, H_T(z_i)
\geq H_T(z)~\forall i>j\bigr)
\\
&&\qquad  \leq G_T(j,z,s,t).
\end{eqnarray*}
\end{cor}

\begin{pf}
By Lemma~\ref{leMtO},
\begin{eqnarray*}
&& P^{\xi}\bigl(H_T(z) \leq s\wedge t, H_T(z_i)\geq H_T(z)\wedge
t_i ~\forall i\leq j, H_T(z_i)
\geq H_T(z)~\forall i>j\bigr)
\\
&&\qquad  \leq G_T(j,z,s\wedge t,s\wedge t),
\end{eqnarray*}
but $G_T(j,z,s,t)$ is increasing in both $s$ and $t$.
\end{pf}

%s4.2 #&#
\subsection{Bounding $G_T$}\label{seG}
The work above will allow us to reduce the problem of proving upper
bounds to bounding $G_T(j,z,s,t)$. As mentioned above, we want to work
by induction, and the following result allows us to bound $G_T(j,\cdot
,\cdot,\cdot)$ in terms of $G_T(j-1,\cdot,\cdot,\cdot)$. Recall that
\[
J_T(t,R) = P\bigl(X(u) \mbox{ jumps at least } Rr(T) \mbox{ times
before time } tT\bigr).
\]

%
%le4.4 #&#
\begin{lem}\label{leG}
Suppose that $j\geq1$ and that $\xi_T(y)\leq\xi_T(z_1)\leq\xi
_T(z_2)\leq\cdots\leq\xi_T(z_j)$ for all $y\notin\{z_1,z_2,\ldots
\}$. Then $\P$-almost surely, for any $z$ and any $s,t\geq0$,
\begin{eqnarray*}
G_T(j,z,s,t)& \leq& G_T(j-1,z,s,t)
\\
&&{}+ G_T(j-1,z_j,t_j,t)e^{a(T)T\xi
_T(z_j)(s-t_j)_+}J_T
\bigl(t,\llvert z-z_j\rrvert \bigr),
\end{eqnarray*}
where for $x\in\R$, $x_+ = x\vee0$.
\end{lem}

\begin{pf}
The main idea is that either we hit $z_j$ before hitting $z$, or we do
not. In the latter case, we reduce to $G_T(j-1,z,s,t)$, and in the
former case, our best tactic is to get to $z_j$ as quickly as possible
(since it has larger potential than any other point we are allowed to
visit) and stay there for as long as we can. Getting there as quickly
as possible gives us $G_T(j-1,z_j,t_j,t)$, and staying there until time
$s$ gives us the exponential factor; then we must also at some point
run to $z$, which costs us $J_T(t,\llvert  z-z_j\rrvert  )$.

By default, all statements below hold $\Prob$-almost surely. If
$z=z_j$, then $G_T(j,z,s,t)\leq G_T(j-1,z,s,t)$, so the inequality
trivially holds. We may therefore assume that $z\neq z_j$. Note then
that either $H^*_T(z_j) > H^*_T(z)$ or $H^*_T(z_j) < H^*_T(z)$. In the
former case $A^*_T(j-1,z,t)$ occurs, and therefore
\begin{eqnarray*}
G_T(j,z,s,t) &\leq& G_T(j-1, z,s,t)
\\
&&{} + E^{\xi} \biggl[\exp \biggl(T\int_0^{H^*_T(z)\wedge s}
\xi\bigl(X(uT)\bigr) \,du \biggr) \1_{A^*_T(j,z,t)}\1_{\{H^*_T(z_j)<H^*_T(z)\}} \biggr].
\end{eqnarray*}
On $A^*_T(j,z,t)$, since $\xi_T(z_j)\geq\xi_T(z_i)$ for all $i\leq
j$ and $t_j\leq H^*_T(z_j) < H^*_T(z_i)$ for all $i>j$, we have
\begin{eqnarray*}
&& T\int_0^{H^*_T(z)\wedge s} \xi\bigl(X(uT)\bigr) \,du
\\
&&\qquad \leq T
\int_0^{H^*_T(z_j)\wedge t_j} \xi\bigl(X(uT)\bigr) \,du + a(T)T
\xi_T(z_j) (s-t_j)_+.
\end{eqnarray*}
Note also that
\begin{eqnarray*}
&&A^*_T(j,z,t)\cap\bigl\{H^*_T(z_j)<H^*_T(z)
\bigr\}
\\
&&\qquad \subseteq A^*_T(j-1,z_j,t)\cap\bigl
\{t_j\leq H^*_T(z_j)<H^*_T(z)
\leq t\bigr\}.
\end{eqnarray*}
Thus
%
%e3 #&#
%e4 #&#
%e5 #&#
\begin{eqnarray}
\label{eqG1}
G_T(j,z,s,t) &\leq& G_T(j-1,z,s,t)\nonumber
\\
&&{} + E^{\xi} \biggl[\exp \biggl(T\int_0^{H^*_T(z_j)\wedge t_j}
\xi \bigl(X(uT)\bigr) \,du \biggr) \1_{A^*_T(j-1,z_j,t)}
\\
&&\hspace*{29pt}{}\times\exp\bigl(a(T)T\xi_T(z_j) (s-t_j)_+
\bigr) \1_{\{H^*_T(z_j)<H^*_T(z)\leq
t\}} \biggr].\nonumber
\end{eqnarray}
If $(\mathcal G_u,u\geq0)$ is the natural filtration for $X$, we
observe that
\[
P^{\xi} \bigl(H^*_T(z_j)<
H^*_T(z)\leq t \mid\mathcal G_{H^*_T(z_j)T} \bigr) \leq
J_T\bigl(t,\llvert z-z_j\rrvert \bigr).
\]
Inserting this into (\ref{eqG1}) (the first part inside the
expectation is $\mathcal G_{H^*_T(z_j)T}$-measurable), we obtain
\begin{eqnarray*}
G_T(j,z,s,t)&\leq& G_T(j-1,z,s,t)
\\
&&{} + E^{\xi} \biggl[\exp \biggl(T\int_0^{H^*_T(z_j)\wedge t_j}
\xi\bigl(X(uT)\bigr) \,du \biggr) \1 _{A^*_T(j-1,z_j,t)} \biggr]
\\
&&\hspace*{10pt}{} \times\exp\bigl(a(T)T\xi_T(z_j)
(s-t_j)_+\bigr) J_T\bigl(t,\llvert
z-z_j\rrvert \bigr).
\end{eqnarray*}
We now recognize the expectation above as $G_T(j-1,z_j,t_j,t)$, which
gives us exactly the expression required.
\end{pf}

Now we have a way of reducing $j$ until it hits $0$, and so we need a
bound on $G_T(0,z,s,t)$. Recall that $\bar h_T(R) = \max_{z\in
L_T(0,R)} h_T(z)$. The lemma below gives a simple bound when $s$ is
slightly smaller than $h_T(z)$ and $z$ is outside a ball about the
origin. It may be useful to imagine applying it when $R$ is small and
$z$ is a long way from $B(0,R)$, so that $\bar h_T(R)<\delta$ and
$q(\gamma-1)\llvert  z\rrvert   + q\gamma R < 0$. (This is exactly what we shall do later.)

%
%le4.5 #&#
\begin{lem}\label{leVW}
Let $\gamma\in(0,1)$, $\delta>0$ and $t>0$. $\P$-almost surely, if
$\xi_T(y)\leq\bar\xi_T(R)$ for all $y\notin\{z_1,z_2,\ldots\}$,
then for $z\notin B(0,R)$, if $\gamma h_T(z)-\delta\geq0$,
\begin{eqnarray*}
&& G_T\bigl(0, z,\gamma h_T(z)-\delta,t\bigr)
\\
&&\qquad \leq\exp \bigl(a(T)T \bigl(\bar\xi_T(R) \bigl(\gamma\bar
h_T(R)-\delta\bigr) + q(\gamma-1)\llvert z\rrvert +q\gamma R +
\calE^2_T\bigl(t,\llvert z\rrvert \bigr) \bigr) \bigr).
\end{eqnarray*}
\end{lem}

\begin{pf}
We again work $\P$-almost surely throughout. Whenever $y\notin\{
z_1,\break  z_2,\ldots\}$, we have $\xi_T(y)\leq\bar\xi_T(R)$, and
therefore on $A^*_T(0,z,t)$ we have
\[
T\int_0^{H^*_T(z)\wedge(\gamma h_T(z)-\delta)} \xi\bigl(X(uT)\bigr) \,du \leq
a(T)T \bar\xi_T(R) \bigl(\gamma h_T(z)-\delta\bigr).
\]
Also,
\[
h_T(z) \leq\min_{y\in B(0,R)} \biggl
\{h_T(y) + q\frac{\llvert  z-y\rrvert  }{\xi
_T(y)} \biggr\} \leq\bar
h_T(R) + q\frac{\llvert  z\rrvert  +R}{\bar\xi_T(R)}.
\]
Thus
\begin{eqnarray*}
&& G_T\bigl(0,  z,\gamma h_T(z)-\delta,t\bigr)
\\
&&\qquad \leq E^{\xi} \bigl[\exp \bigl(a(T)T \bigl(\bar\xi_T(R)
\bigl(\gamma\bar h_T(R) -\delta\bigr) + q\gamma\llvert z\rrvert + q
\gamma R \bigr) \bigr) \1 _{A^*_T(0,z,t)} \bigr]
\\
&&\qquad  = \exp\bigl(a(T)T \bigl(\bar\xi_T(R) \bigl(\gamma\bar
h_T(R) - \delta\bigr) + q\gamma \llvert z\rrvert + q\gamma R\bigr)
\bigr) P^{\xi}\bigl(A^*_T(0,z,t)\bigr).
\end{eqnarray*}
However, $P^{\xi}(A^*_T(0,z,t))$ is at most the probability that our
random walk jumps $\llvert  z\rrvert  r(T)$ times by time $tT$, which is exactly
$J_T(t,\llvert  z\rrvert  )$. Applying Lem\-ma~\ref{rwlem2} completes the proof.
\end{pf}

%s4.3 #&#
\subsection{Fixing parameters}\label{secparameters}
Until now we have worked with general points $z_i$ and times $t_i$. Now
we want to specialize to our particular situation. We suppose that we
are given a fixed time $t_\infty\geq1$ and proceed to fix a variety
of parameters which we will use to ensure that the probability that
particles arrive at any point $z$ substantially before $h_T(z)\wedge
t_\infty$ is small. We choose:
\begin{itemize}
\item$\psi\in(\frac{1}2, 1)$;
\item$\delta_T = 1/(3\log^\psi T)$;
\item$\gamma_T = 1-1/(3t_\infty\log^\psi T)$;\vspace*{2pt}
\item$\theta_T = \log^{-2(\psi+1)(q+1)}T$, so that by Corollary
\ref{cosmalllily}, $\P( \bar h_T(\theta_T)>\delta_T/2)\to0$;
\item$\eta_T = (1-\gamma_T)\theta_T/3 = \log^{-\psi-2(\psi
+1)(q+1)}T/(9t_\infty)$;
\item$\rho_T = \log\log T$, so that by Lemma~\ref{le0306-1}(ii),
$\P (\bar\xi_T(\rho_T) \geq\frac{q(\rho_T-2)\gamma
_T}{t_\infty+\delta_T} )\to0$;
\item$\nu_T = \log^{-2\psi d/\alpha- 2(\psi+1)q}T$ so that by
Lemma~\ref{le0306-1}(i),
\[
\P \bigl(\bar\xi_T(\eta_T) < \nu_T
\bigr) \to0;
\]
\item$K_T = \log^{3\psi d+2(\psi+1)q\alpha}T$, so that by Lemma
\ref{le0306-1}(ii),
\[
\P\bigl(\#\bigl\{z\in B(0,\rho_T)\dvtx  \xi_T(z) \geq
\nu_T\bigr\}>K_T\bigr) \to0;
\]
\item$\beta_T = \log^{-5\psi d - 4(\psi+1)(q+1)\alpha} T$.
\end{itemize}
We also define
\begin{eqnarray*}
\Gamma_T &=& \biggl\{\bar h_T(\theta_T)
\leq\delta_T/2, \bar\xi_T(
\eta_T) \geq\nu_T,
\\
&&\hspace*{7pt} \#\bigl\{z\in B(0,\rho_T)\dvtx  \xi_T(z) \geq
\nu_T\bigr\}\leq K_T, \bar
\xi_T(\rho_T) < \frac{q(\rho_T-2)\gamma_T}{t_\infty
+\delta_T} \biggr\}.
\end{eqnarray*}
We think of $\Gamma_T$ as a good event on which the environment
behaves sensibly. Note from above that $\P(\Gamma_T)\to1$.

We now let
\[
Z = \bigl\{z\in L_T(0,\rho_T)\dvtx
\xi_T(z)> \bar\xi_T(\eta_T)\bigr\},
\qquad \kappa(T) = \#Z,
\]
and
\[
Z' = \bigl\{z\notin L_T(0,\rho_T)\dvtx
\xi_T(z)>\bar\xi_T(\eta_T)\bigr\}.
\]
We label the elements of $Z$ as $z_1,\ldots,z_{\kappa(T)}$ such that
$\xi_T(z_1)\leq\cdots\leq\xi_T(z_{\kappa(T)})$, and the elements
of $Z'$ arbitrarily as $z_{\kappa(T)+1},z_{\kappa(T)+2},\ldots.$ Let
$t_i = (\gamma_T h_T(z_i) - \delta_T)_+$ for each $i$. Note that
$z_i$ and $t_i$ only depend on the environment $\xi$ so that we are
allowed to apply the
results in Sections~\ref{semto} and~\ref{seG}.
(Of course $z_i$ and $t_i$ also depend on $T$, but keeping track of
this would make our notation unwieldy.)

We can now translate our general results about $G_T$ from the previous
section to get bounds for our particular choice of $z_i$ and $t_i$.

%
%le4.6 #&#
\begin{lem}\label{leG0}
On $\Gamma_T$, $\P$-almost surely, for any $z\notin B(0,\eta_T)$ and
any $t>0$,
\[
G_T\bigl(0,z,\bigl(\gamma_T h_T(z)-
\delta_T\bigr)_+,t\bigr) \leq\exp \bigl(a(T)T \bigl(
\tfrac{2}{3}q(\gamma_T-1)\llvert z\rrvert +
\calE^2_T\bigl(t,\llvert z\rrvert \bigr) \bigr)
\bigr).
\]
\end{lem}

\begin{pf} Note that if $\gamma_T h_T(z)-\delta_T<0$, then
\begin{eqnarray*}
G_T\bigl(0,z,\bigl(\gamma_T h_T(z)-
\delta_T\bigr)_+,t\bigr) & \leq& P^{\xi}
\bigl(A^*_T(0,z,t)\bigr)
\\
& \leq&\exp\bigl(a(T)T\bigl(-q\llvert z\rrvert +\calE^2_T
\bigl(t,\llvert z\rrvert \bigr)\bigr)\bigr),
\end{eqnarray*}
where the first inequality comes directly from the definition of $G_T$
and the second is from Lemma~\ref{rwlem2}.
In particular, on $\Gamma_T$ this bound applies if $z\in B(0,\theta
_T)\setminus B(0,\eta_T)$.
Thus it remains to consider $z\notin B(0,\theta_T)$ such that $\gamma
_T h_T(z) - \delta_T \geq0$. Since on $\Gamma_T$ we have $\bar
h_T(\eta_T)\leq\bar h_T(\theta_T) <\delta_T$, by Lemma~\ref{leVW}
(with $R=\eta_T$),
\begin{eqnarray*}
&& G_T\bigl(0,z,\bigl(\gamma_T h_T(z)-
\delta_T\bigr)_+,t\bigr)
\\
&&\qquad \leq\exp\bigl(a(T)T\bigl(q(\gamma
_T-1)\llvert z\rrvert + q\gamma_T\eta_T
+ \calE^2_T\bigl(t,\llvert z\rrvert \bigr)\bigr)\bigr).
\end{eqnarray*}
However, we chose $\eta_T = (1-\gamma_T)\theta_T/3 \leq(1-\gamma
_T)\llvert  z\rrvert  /3$, and the result follows.
\end{pf}

Now we want to apply Lemma~\ref{leG} to bound $G_T(j,\cdot,\cdot,\cdot)$ for $j\geq1$. Note that we cannot induct directly on $G_T$
since Lemma~\ref{leG} relates $G_T(j,z,\cdot,\cdot)$ to
$G_T(j-1,z_j,\cdot,\cdot)$ rather than $G_T(j-1,z,\cdot,\cdot)$.
However, we can work with
\[
\bar G_T = \max_{k\leq\kappa(T)} G_T\bigl(
\kappa(T),z_k,t_k,t_\infty\bigr).
\]
Since $\psi< 1$, we can choose $T_1$ such that
%
%e6 #&#
\begin{equation}
\label{eqbdE2} \calE^2_T(t_\infty,R) \leq q(1-
\gamma_T) R/3 \qquad \forall R\geq
\beta_T, T\geq T_1
\end{equation}
and
%
%e7 #&#
\begin{equation}
\label{eqbd1plus} 1+\exp \bigl(a(T)T\cdot q\beta _T \bigr) \leq\exp
\bigl(a(T)T\cdot2q\beta_T \bigr) \qquad\forall T
\geq T_1.
\end{equation}

%
%le4.7 #&#
\begin{lem}\label{leGK}
On $\Gamma_T$, for all $T\geq T_1$, $\P$-almost surely,
\[
\bar G_T \leq e^{a(T)T (q(\gamma_T-1)\eta_T/3 + 2 K_T q \beta_T)}.
\]
\end{lem}

\begin{pf}
By Lemma~\ref{leG}, $\P$-almost surely, for $j,k \leq\kappa(T)$,
\begin{eqnarray*}
&&G_T(j,z_k, t_k,t_\infty)
\\
&&\qquad  \leq G_T(j-1,z_k,t_k,t_\infty)
\\
&&\quad\qquad{} + G_T(j-1,z_j,t_j,t_\infty)e^{a(T)T\xi
_T(z_j)(t_k-t_j)_+}J_T
\bigl(t_\infty,\llvert z_k-z_j\rrvert \bigr)
\\
&&\qquad \leq G_T(j-1,z_k,t_k,t_\infty)
\\
&&\quad\qquad {}+ G_T(j-1,z_j,t_j,t_\infty)e^{a(T)T \gamma_T
q\llvert  z_k-z_j\rrvert  }J_T
\bigl(t_\infty,\llvert z_k-z_j\rrvert \bigr),
\end{eqnarray*}
where we use that if $t_k - t_j > 0$, then by Lemma~\ref{leequiv} and
since $t_j \geq\gamma_T h_T(z_j) - \delta$,
\[
(t_k - t_j)_+ \leq\gamma_T
\bigl(h_T(z_k) - h_T(z_j)
\bigr) \leq\gamma_T q \frac{\llvert  z_k-z_j\rrvert  }{\xi_T(z_j)}.
\]

Now, if $\llvert  z_k-z_j\rrvert  <\beta_T$, then [since trivially $J_T(t_\infty,\llvert  z_k-z_j\rrvert  )\leq1$]
\[
e^{a(T)T \gamma_T q\llvert  z_k-z_j\rrvert  }J_T\bigl(t_\infty,\llvert
z_k-z_j\rrvert \bigr) \leq e^{a(T)T q
\beta_T};
\]
on the other hand, if $\llvert  z_k-z_j\rrvert  \geq\beta_T$, then by Lemma~\ref{rwlem2}
\[
e^{a(T)T \gamma_T q\llvert  z_k-z_j\rrvert  }J_T\bigl(t_\infty,\llvert
z_k-z_j\rrvert \bigr) \leq e^{a(T)T
(\gamma_T q\llvert  z_k-z_j\rrvert   - q\llvert  z_k-z_j\rrvert   + \calE^2(t_\infty,\llvert  z_k-z_j\rrvert  ))}
\]
so that if $T\geq T_1$ by~(\ref{eqbdE2}),
\[
e^{a(T)T \gamma_T q\llvert  z_k-z_j\rrvert  }J_T\bigl(t_\infty,\llvert
z_k-z_j\rrvert \bigr) \leq1 \leq e^{a(T)T q \beta_T}.
\]

Either way, we can conclude that for any $j\leq\kappa(T)$ and $T\geq
T_1$ by~(\ref{eqbd1plus}),
\begin{eqnarray*}
\max_{k\leq\kappa(T)} G_T(j,z_k,t_k,t_\infty)
& \leq&\max_{k\leq
\kappa(T)} G_T(j-1,z_k,t_k,t_\infty)
\bigl(1+ e^{a(T)T q\beta_T}\bigr)
\\
& \leq&\max_{k\leq\kappa(T)} G_T(j-1,z_k,t_k,t_\infty)
e^{a(T)T 2
q\beta_T}.
\end{eqnarray*}
Iterating this inequality $\kappa(T)$ times beginning with 
$\max_{k\leq\kappa(T)} G_T(\kappa(T),z_k,\break  t_k,t_\infty)$ gives
%
%e8 #&#
\begin{equation}
\label{eqGK} \bar G_T \leq\max_{k\leq\kappa(T)}
G_T(0,z_k,t_k,t_\infty)
e^{a(T)T 2 \kappa(T) q\beta_T},
\end{equation}
but on $\Gamma_T$, $\kappa(T)\leq K_T$. Then applying Lemma~\ref
{leG0} gives the result.
\end{pf}

Very similar arguments allow us to get an estimate on $G_T$ for \emph
{any} point outside $B(0,\eta_T)$.

%
%le4.8 #&#
\begin{lem}\label{leGgen}
On $\Gamma_T$, $\P$-almost surely, for any $z\notin B(0,\eta_T)$ and
any $T\geq T_1$,
\begin{eqnarray*}
&& G_T\bigl(\kappa(T),  z,\bigl(\gamma_T
h_T(z)-\delta _T\bigr)_+,t_\infty\bigr)
\\
&&\qquad \leq e^{a(T)T q(\gamma_T-1)\eta_T/3} + K_T e^{a(T)T (q(\gamma
_T-1)\eta_T/3 + (2 K_T + 1) q \beta_T)}.
\end{eqnarray*}
\end{lem}

\begin{pf}
Essentially we just apply Lemma~\ref{leG} again to relate
$G_T(j,z,\cdot,\cdot)$ to $G_T(j-1,z_j,\cdot,\cdot)$, which we can
now control using Lemma~\ref{leGK}. Indeed, by Lemma~\ref{leG}, for
any $j\leq\kappa(T)$ and $s\geq0$,
\begin{eqnarray*}
&& G_T(j,z,s,t_\infty)
\\
&&\qquad \leq G_T(j-1,z,s,t_\infty) + \bar G_T
\max_{k\leq\kappa(T)} e^{a(T)T \xi_T(z_k)(s-t_k)_+}J_T
\bigl(t_\infty,\llvert z-z_k\rrvert \bigr).
\end{eqnarray*}
When $s=(\gamma_T h_T(z) - \delta_T)_+$, we get $(s-t_k)_+ \leq
\gamma_T q\frac{\llvert  z-z_k\rrvert  }{\xi_T(z_k)}$, so
\[
G(j,z,s,t_\infty) \leq G_T(j-1,z,s,t_\infty) +
\bar G_T \max_{k\leq
\kappa(T)} e^{a(T)T \gamma_T q\llvert  z-z_k\rrvert  }J_T
\bigl(t,\llvert z-z_k\rrvert \bigr).
\]
As in the proof of Lemma~\ref{leGK}, considering two cases (when
$\llvert  z-z_k\rrvert  <\beta_T$ and when $\llvert  z-z_k\rrvert  \geq\beta_T$), we get
\[
e^{a(T)T \gamma_T q\llvert  z-z_k\rrvert  }J_T\bigl(t_\infty,\llvert
z-z_k\rrvert \bigr) \leq e^{a(T)T q
\beta_T}.
\]
Thus
\begin{eqnarray*}
&& G\bigl(j,z,\bigl(\gamma_T h_T(z)-
\delta_T\bigr)_+,t_\infty\bigr)
\\
&&\qquad \leq G_T\bigl(j-1,z,\bigl(\gamma_T
h_T(z) - \delta_T\bigr)_+,t_\infty\bigr) +
\bar G_T e^{a(T)T q\beta_T}.
\end{eqnarray*}
Iterating $\kappa(T)$ times gives
\begin{eqnarray*}
&&G\bigl(\kappa(T), z,\bigl(\gamma_T h_T(z)-\delta
_T\bigr)_+,t_\infty\bigr)
\\
&&\qquad \leq G_T\bigl(0,z,\bigl(\gamma_T
h_T(z)-\delta_T\bigr)_+,t_\infty\bigr) +
\kappa (T)\bar G_T e^{a(T)T q\beta_T};
\end{eqnarray*}
then applying Lemmas~\ref{leG0} and~\ref{leGK} [together with the
fact that on $\Gamma_T$, $\kappa(T)\leq K_T$] completes the proof.
\end{pf}

%s4.4 #&#
\subsection{Particles do not arrive too early}\label{sesubearly}
We are finally in a position to prove our first real result, that
$H_T(z)$ does not occur significantly before $h_T(z)$ for any $z$.

%
%pr4.9 #&#
\begin{prop}\label{prearly}
For any $t_\infty>0$, there exists $T_2$ such that for all $T\geq T_2$,
\[
\P \bigl(\exists z\dvtx  H_T(z)\leq\bigl(\gamma_T
h_T(z)-\delta_T\bigr)\wedge t_\infty \bigr)
\leq\P\bigl(\Gamma_T^c\bigr) + e^{-T}\to0.
\]
\end{prop}

\begin{pf}
All that remains is to tie together the threads developed above. Note
that by Lemma~\ref{le0805-1},
%
%e9 #&#
%e10 #&#
\begin{eqnarray}
\label{eqUT4bd}
&& \bigl\{\exists z\in L_T\setminus B(0,
\rho_T-2) \dvtx  \gamma_T h_T(z)-\delta
_T \leq t_\infty\bigr\}
\nonumber\\[-8pt]\\[-8pt]\nonumber
&&\qquad \subseteq \biggl\{ \bar\xi_T(\rho_T-2)\geq
\frac{q(\rho
_T-2)\gamma_T}{t_\infty+\delta_T} \biggr\} \subseteq\Gamma_T^c.
\end{eqnarray}
Since our random walks only make nearest neighbor steps, particles must
enter $L_T(0,\rho_T)\setminus L_T(0,\rho_T-2)$ before they can exit
$B(0,\rho_T)$. Thus if there exists $z$ outside $B(0,\rho_T)$ such
that $H_T(z)\leq t_\infty$, then there must exist $z$ within
$L_T(0,\rho_T)\setminus L_T(0,\rho_T-2)$ such that $H_T(z)\leq
t_\infty$. Thus on $\Gamma_T$,
\begin{eqnarray*}
&& \bigl\{\exists z\dvtx  H_T(z)\leq\bigl(\gamma_T
h_T(z)-\delta_T\bigr)\wedge t_\infty\bigr\}
\\
&&\qquad \subseteq\bigl\{\exists z\in L_T(0,\rho_T)\dvtx
H_T(z)\leq\bigl(\gamma_T h_T(z)-
\delta_T\bigr)\wedge t_\infty\bigr\}.
\end{eqnarray*}
If a point is hit early, then there must be a \textit{first} point that is
hit early; thus, recalling that
\begin{eqnarray*}
%V_T(z)
&& A_T\bigl(\kappa(T), z, \bigl(\gamma_T
h_T(z)-\delta_T\bigr)_+\wedge t_\infty\bigr)
\\
&&\qquad = \bigl\{H_T(z)\leq\bigl(\gamma_T
h_T(z)-\delta_T\bigr)_+\wedge t_\infty,
H_T(z_i)\geq t_i\wedge
H_T(z)~\forall i\leq\kappa (T),
\\
&&\hspace*{208pt} H_T(z_i)\geq H_T(z)~\forall i>
\kappa(T)\bigr\},
\end{eqnarray*}
we have reduced the problem to showing that on $\Gamma_T$,
\[
P^{\xi} \biggl(\bigcup_{z\in L_T(0,\rho_T)}
A_T\bigl(\kappa(T), z, \bigl(\gamma_T
h_T(z)-\delta_T\bigr)_+\wedge t_\infty\bigr)
\biggr) \leq e^{-T}.
\]
However, by Corollary~\ref{coPG}, for any $z\in L_T$,
\[
P^{\xi}\bigl(A_T\bigl(\kappa(T), z, \bigl(
\gamma_T h_T(z)-\delta_T\bigr)_+\wedge
t_\infty\bigr)\bigr) \leq G_T\bigl(\kappa(T),z,\bigl(
\gamma_T h_T(z)-\delta _T
\bigr)_+,t_\infty\bigr),
\]
and by Lemma~\ref{leGgen}, for any $z\notin B(0,\eta_T)$, on $\Gamma
_T$, $\P$-almost surely
\begin{eqnarray*}
&& G_T\bigl(\kappa(T), z,\bigl(\gamma_T
h_T(z)-\delta _T\bigr)_+,t_\infty\bigr)
\\
&&\qquad \leq e^{a(T)T q(\gamma_T-1)\eta_T/3} + K_T e^{a(T)T (q(\gamma
_T-1)\eta_T/3 + (2 K_T + 1) q \beta_T)}.
\end{eqnarray*}
Now,
\[
\frac{q(\gamma_T-1)\eta_T}{3} = -\frac{q}{81 t_\infty^2 \log
^{2\psi+ 2(\psi+1)(q+1)}T}
\]
and
\[
(2K_T+1)q\beta_T \leq\frac{3q}{\log^{2\psi d + 2(\psi
+1)(q+1)\alpha} T},
\]
so by taking $T$ large (not depending on the environment $\xi$), we
can certainly ensure that
\[
e^{a(T)T (q(\gamma_T-1)\eta_T/3 + (2 K_T + 1) q \beta_T)}\leq e^{-2T}.
\]
Thus for large $T$, for any $z\notin B(0,\eta_T)$, on $\Gamma_T$, $\P
$-almost surely
\[
G_T\bigl(\kappa(T),z,\gamma_T h_T(z)-
\delta_T,t_\infty\bigr) \leq(K_T+1)e^{-2T}.
\]
Also, on $\Gamma_T$, if $z\in B(0,\eta_T)$, then $h_T(z)<\delta_T$
and hence $P^{\xi}(A_T(\kappa(T), z$, $(\gamma_T h_T(z)-\delta
_T)_+\wedge t_\infty))=0$.
We deduce that for large $T$,
\begin{eqnarray*}
&& P^{\xi} \biggl(\bigcup_{z\in L_T(0,\rho_T)}
A_T\bigl(\kappa(T), z, \bigl(\gamma_T
h_T(z)-\delta_T\bigr)_+\wedge t_\infty\bigr)
\biggr)
\\
&&\qquad \leq C_d \rho_T^d r(T)^d
(K_T+1) e^{-2T} %\\
%& \leq&
e^{-T}, %\to0.\qedhere
\end{eqnarray*}
which converges to $0$ as $T \rightarrow\infty$.
\end{pf}

%s4.5 #&#
\subsection{There are not too many particles}\label{setoomany}
Now that we have bounded the probability that any of our points is hit
early, we check that the number of particles at any site cannot be too
large (given that no point is hit early). We work with the same
parameters as above, and the same choice of $z_i, t_i, i\geq1$. Indeed
our whole tactic will be very similar, except that instead of looking
at the expected number of particles at $z$ at time $s\wedge H_T(z)$, we
will instead just look at time $s$ [conditional on not having hit $z$
substantially before $h_T(z)$].

Define the events
\[
\calH_T = \bigl\{H_T(y) > \bigl(
\gamma_T h_T(y) - \delta_T\bigr)\wedge
t_\infty ~\forall y\bigr\}
\]
and for $s > 0$
\[
\calH^*_T(s) = \bigl\{H^*_T(y) > \bigl(
\gamma_T h_T(y) - \delta_T\bigr)\wedge s
~\forall y\bigr\}.
\]
We know from Proposition~\ref{prearly} that $\P(\calH_T)\to1$ as
$T\to\infty$.

We begin with a lemma that allows us to control the number of particles
at $z$ by linking to something we already know a lot about: $G_T$.

%
%le4.10 #&#
\begin{lem}\label{leexpno1}
On $\Gamma_T$, $\P$-almost surely, for any $s\leq t_\infty$ and any $z$,
\begin{eqnarray*}
&& E^{\xi} \Bigl[\sup_{u\in(s-\delta_T,s]}\#\bigl\{v\in Y(uT)\dvtx
X_v(uT) = r(T)z\bigr\}\1_{\calH_T} \Bigr]
\\
&&\qquad \leq e^{a(T)T \bar\xi_T(\eta_T) s} J_T\bigl(s,\llvert z\rrvert \bigr)
\\
&&\quad\qquad{} + \sum_{j=1}^{\kappa(T)}
G_T(j,z_j,t_j,t_\infty)
e^{a(T)T\xi_T(z_j)(s- t_j)+a(T)T\xi
_T(z_{\kappa(T)})\delta_T} J_T\bigl(s,\llvert z-z_j\rrvert
\bigr)
\\
&&\hspace*{61pt}{}\times \1_{\{t_j \leq s \}}.
\end{eqnarray*}
\end{lem}

\begin{pf}
The plan is as follows: we apply the many-to-one lemma to turn our
expectation over the branching random walk into an expectation
involving only one random walk. Then either we do not hit any $z_j$
before time $s$, in which case our potential is small, or there is a
last $z_j$ that we hit. For each $j$ we then use similar calculations
to those in the proof of Lemma~\ref{leG}.

For $s>0$, let
\[
\tilde A^*_T(0,z,s) = \bigl\{\exists u\in(s-\delta_T,s]
\dvtx  X(uT) = r(T)z, H_T^*(z_i)>s-
\delta_T ~\forall i\bigr\}\cap\calH ^*_T(s),
\]
and for $j\geq1$, let
\begin{eqnarray*}
\tilde A^*_T (j,z,s) &=& \lleft\{ \matrix{\exists u\in(s-
\delta_T,s]\dvtx  X(uT) = r(T)z,
\vspace*{3pt}\cr
H_T^*(z_j)<s-
\delta_T, H_T^*(z_i)>s-
\delta_T~\forall i>j} \rright\}\cap\calH^*_T(s).
\end{eqnarray*}
Informally, $\tilde A^*_T(j,z,s)$ says that we are at $z$ around time
$s$, we traveled via $z_j$ (unless $j=0$) and not via $z_i$ for $i>j$,
and no-one was hit early.

First note that
\begin{eqnarray*}
&& E^{\xi} \Bigl[\sup_{u\in(s-\delta_T,s]}\#\bigl\{v\in Y(uT)\dvtx
X_v(uT) = r(T)z\bigr\}\1_{\calH_T} \Bigr]
\\
&&\qquad \leq E^{\xi} \biggl[\sum_{v\in Y(sT)}
\1_{\{\exists u\in(s-\delta
_T,s]\dvtx  X_v(uT)=r(T)z, H^v_T(y)>(\gamma_T h_T(y)-\delta
_T)\wedge s~\forall y\}} \biggr].
\end{eqnarray*}
By the many-to-one lemma, this equals
\[
E^{\xi} \biggl[\exp \biggl(T\int_0^s
\xi\bigl(X(uT)\bigr) \,du \biggr) \1_{\{
\exists u\in(s-\delta_T,s]\dvtx  X(uT)=r(T)z\}\cap\calH^*_T(s)} \biggr].
\]
However, either we do not hit any $z_j$, or there is a last $j$ such
that we hit $z_j$ before time $s-\delta_T$, so the above is at most
\[
\sum_{j=0}^\infty E^{\xi} \biggl[
\exp \biggl(T\int_0^s \xi\bigl(X(uT)\bigr) \,du
\biggr) \1_{\tilde A^*_T(j,z,s)} \biggr].
\]
If $t_j = \gamma_T h_T(z_j)-\delta_T >s$, then $\tilde
A^*_T(j,z,s)=\varnothing$. As in (\ref{eqUT4bd}), on $\Gamma_T$ we
have $\gamma_T h_T(z_j)-\delta_T>t_\infty$ for all $j>\kappa(T)$,
so $\tilde A^*_T(j,z,s)=\varnothing$ for all $j>\kappa(T)$. Thus we may
restrict the sum above to $j\leq\kappa(T)$ such that $h_T(z_j)\leq
(s+\delta_T)/\gamma_T$. We then know that on $\tilde A^*_T(j,z,s)$,
between times $t_j T$ and $(s-\delta_T)T$ our potential is at most
$a(T)\xi_T(z_j)$, and between times $(s-\delta_T)T$ and $sT$ our
potential is at most $a(T)\xi_T(z_{\kappa(T)})$. This tells us that
\begin{eqnarray*}
&& E^{\xi} \biggl[\exp \biggl(T\int_0^s
\xi\bigl(X(uT)\bigr) \,du \biggr) \1_{\tilde
A^*_T(j,z,s)} \biggr]
\\
&&\qquad \leq E^{\xi} \biggl[\exp \biggl(T\int_0^{t_j}
\xi\bigl(X(uT)\bigr) \,du + a(T)T\xi_T(z_j)
(s-t_j)
\\
&&\hspace*{171pt}{} + a(T)T\xi_T(z_{\kappa(T)})\delta_T \biggr)
\1_{\tilde
A^*_T(j,z,s)} \biggr].
\end{eqnarray*}
Recall that we defined
\[
A^*_T(j,z,t) = \left\{ \matrix{ H^*_T(z) \leq t,
H^*_T(z_i)\geq H^*_T(z)\wedge
t_i~\forall i\leq j,
\vspace*{3pt}\cr
H^*_T(z_i)\geq
H^*_T(z)~\forall i>j} \right\}.
\]
Note that for $j\geq1$,
\begin{eqnarray*}
\tilde A^*_T(j,z,s)
&\subseteq& A^*_T(j,z_j,s)
\\
&&{} \cap\bigl
\{H_T^*(z_j) \in(t_j, s-
\delta_T], \exists u\in(s-\delta_T,s]\dvtx
X(uT)=r(T)z\bigr\}.
\end{eqnarray*}
Further, recalling that $\mathcal G_u, u\geq0$ is the natural
filtration of our random walk $X$, we have
\begin{eqnarray*}
&& P^{\xi}\bigl(H_T^* (z_j)\leq s-
\delta_T,\exists u\in l(s-\delta_T,s]\dvtx  X(uT)=r(T)z \mid \mathcal G_{(H_T^*(z_j)\wedge(s-\delta))T}\bigr)
\\
&&\qquad  \leq J_T\bigl(s,\llvert z-z_j\rrvert \bigr).
\end{eqnarray*}
Thus since $e^{T\int_0^{t_j} \xi(X(uT)) \,du}\1_{A^*_T(j,z_j,s)}\1_{\{
H_T^*(z_j) > t_j\}}$ is $\mathcal G_{H_T^*(z_j)}$-measurable,
\begin{eqnarray*}
&&E^{\xi} \biggl[\exp \biggl(T\int_0^{t_j}
\xi\bigl(X(uT)\bigr) \,du + a(T)T\xi _T(z_j)
(s-t_j) \biggr)
\\
&&\hspace*{76pt}{}\times\exp \bigl(a(T)T\xi_T\bigl(z_\kappa(T)
\bigr)\delta_T \bigr)\1_{\tilde A^*_T(j,z,s)} \biggr]
\\
&&\qquad  \leq E^{\xi} \biggl[\exp \biggl(T\int_0^{t_j}
\xi\bigl(X(uT)\bigr) \,du \biggr)\1 _{A^*_T(j,z_j,s)}\1_{\{H_T^*(z_j) > t_j\}} \biggr]
\\
&&\quad\qquad{}\times\exp \bigl(a(T)T\xi_T(z_j)
(s-t_j) + a(T)T\xi _T(z_{\kappa(T)})
\delta_T \bigr) J_T\bigl(s,\llvert z-z_j
\rrvert \bigr).
\end{eqnarray*}
However,
\begin{eqnarray*}
&& E^{\xi} \biggl[\exp \biggl(T\int_0^{t_j \wedge(s-
\delta_T)}
\xi\bigl(X(uT)\bigr) \,du \biggr)\1_{A^*_T(j,z_j,s)} \1_{\{H_T^*(z_j) > t_j\}} \biggr]
\\
&&\qquad \leq G_T(j,z_j,t_j,t_\infty),
\end{eqnarray*}
so putting all of this together,
\begin{eqnarray*}
&& E^{\xi}\bigl[\#\bigl\{v\in Y(sT)\dvtx  X_v(sT) = r(T)z\bigr\}
\1_{\calH_T}\bigr]
\\
&&\qquad \leq E^{\xi} \biggl[\exp \biggl(T\int_0^s
\xi\bigl(X(uT)\bigr) \,du \biggr) \1 _{\tilde A^*_T(0,z,s)} \biggr]
\\
&&\quad\qquad{} + \sum_{j=1}^{\kappa(T)}
G_T(j,z_j,t_j,t_\infty)
e^{a(T)T\xi_T(z_j)(s-t_j) + a(T)T\xi
_T(z_{\kappa(T)})\delta_T} J_T\bigl(s,\llvert
z-z_j\rrvert \bigr)
\\
&&\hspace*{63pt}{}\times \1_{\{t_j \leq s\}}.
\end{eqnarray*}
Finally, since $\xi_T(y)\leq\bar\xi_T(\eta_T)$ for all $y\notin\{
z_1,z_2,\ldots\}$, we have
\[
E^{\xi} \biggl[\exp \biggl(T\int_0^s
\xi\bigl(X(uT)\bigr) \,du \biggr) \1_{\tilde
A^*_T(0,z,s)} \biggr] \leq
e^{a(T)T \bar\xi_T(\eta_T) s}J_T\bigl(s,\llvert z\rrvert \bigr),
\]
and the result follows.
\end{pf}

We now use our knowledge of $G_T$ to get a bound in terms of $m_T(z,s)$.

%
%le4.11 #&#
\begin{lem}\label{leexpno2}
There exists $T_3$ such that for any $T\geq T_3$, on $\Gamma_T$, $\P
$-almost surely, for any $s\leq t_\infty$ and any $z$,
\begin{eqnarray*}
E^{\xi} \Bigl[\sup_{u\in(s-\delta_T,s]}\#\bigl\{v\in Y(uT)\dvtx
X_v(uT) = r(T)z\bigr\}\1_{\calH_T} \Bigr]
\\
\leq\exp \biggl( a(T)T\biggl(m_T(z,s) + \frac{1}2
\log^{-1/2} T\biggr) \biggr).
\end{eqnarray*}
\end{lem}

\begin{pf}
Note that, as in (\ref{eqUT4bd}), we may assume that $z\in B(0,\rho
_T)$; otherwise the expectation is 0. Clearly our starting point is
Lemma~\ref{leexpno1}. First we show that
\[
e^{a(T)T \bar\xi_T(\eta_T) s} J_T\bigl(s,\llvert z\rrvert \bigr) \leq\exp
\bigl( a(T)T\bigl(m_T(z,s) + \tfrac{1}4
\log^{-1/2} T\bigr) \bigr).
\]
To do this, choose $y\in B(0,\eta_T)$ such that $\xi_T(y) = \bar\xi
_T(\eta_T)$. Then applying Lemma~\ref{rwlem2},
\begin{eqnarray*}
&&e^{a(T)T \bar\xi_T(\eta_T) s} J_T\bigl(s,\llvert z\rrvert \bigr)
\\
&&\qquad \leq\exp \bigl( a(T)T \bigl(\xi_T(y) s - q\llvert z\rrvert +
\calE^2_T\bigl(s,\llvert z\rrvert \bigr)\bigr) \bigr)
\\
&&\qquad \leq\exp \bigl( a(T)T \bigl(\xi_T(y) \bigl(s-h_T(y)
\bigr) - q\llvert z-y\rrvert
\\
&&\hspace*{52pt}\quad\qquad{} + \xi_T(y)h_T(y) + q\llvert y\rrvert
+ \calE^2_T\bigl(s,\llvert z\rrvert \bigr)\bigr) \bigr)
\\
&&\qquad \leq\exp \bigl( a(T)T \bigl(m_T(z,s) + \xi_T(y)h_T(y)
+ q\llvert y\rrvert + \calE ^2_T\bigl(s,\llvert z
\rrvert \bigr)\bigr) \bigr).
\end{eqnarray*}
However, on $\Gamma_T$, we have $h_T(y)\leq\delta_T/2 = 1/(6\log
^\psi T)$, and $\psi>\frac{1}2$, so on $\Gamma_T$,
\[
\xi_T(y) h_T(y) \leq\max_{y' \in L_T(0,\rho_T)}
\xi_T\bigl(y'\bigr) \delta _T /2 \leq
\frac{q}{2 t_\infty} \rho_T \delta_T \leq
\frac{1}8 \log^{-1/2} T,
\]
for $T$ large.
Also $\llvert  y\rrvert  \leq\eta_T \leq\log^{-\psi} T$, and since $\llvert  z\rrvert  <\rho_T$,
for large $T$ we have $\calE^2_T(s,\llvert  z\rrvert  )\leq(q+2)(\log\log T)^2/\log
T$. Thus
\[
e^{a(T)T \bar\xi_T(\eta_T) s} J_T\bigl(s,\llvert z\rrvert \bigr) \leq\exp
\bigl( a(T)T\bigl(m_T(z,s) + \tfrac{1}4
\log^{-1/2} T\bigr) \bigr),
\]
as claimed.

We now move on to bounding
\[
\sum_{j=1}^{\kappa(T)} G_T(j,z_j,t_j,t_\infty)
e^{a(T)T\xi
_T(z_j)(s-t_j)+a(T)T\xi_T(z_{\kappa(T)})\delta_T} J_T\bigl(s,\llvert z-z_j\rrvert
\bigr)\1 _{\{t_j \leq s\}}.
\]
By Lemma~\ref{rwlem2}, $J_T(s,\llvert  z-z_j\rrvert  )\leq\exp (-a(T)T(
q\llvert  z-z_j\rrvert   - \calE^2_T(s,\llvert  z-z_j\rrvert  ) )$, so the above is at most
\begin{eqnarray*}
&& \sum_{j=1}^{\kappa(T)} G_T(j,z_j,t_j,t_\infty)
\1_{\{h_T(z_j)\leq
(s+\delta_T)/\gamma_T\}}
\\
&&\hspace*{15pt}{}\times  e^{a(T)T(\xi_T(z_j)(s-h_T(z_j) - q\llvert  z-z_j\rrvert  ) + \xi
_T(z_j)(1-\gamma_T) h_T(z_j) +2 \xi_T(z_{\kappa(T)})\delta_T +
\calE^2_T(s,\llvert  z-z_j\rrvert  ))}.
\end{eqnarray*}
Since we are assuming $z\in B(0,\rho_T)$, and $j\leq\kappa(T)$ so
$z_j\in B(0,\rho_T)$, we have $\llvert  z-z_j\rrvert  \leq2\rho_T$ so as above for
large $T$ we have $\calE^2_T(s,\llvert  z-z_j\rrvert  )\leq(q+2)(\log\log T)^2/\log
T$. Also, if $h_T(z_j)\leq s\leq t_\infty$, then on $\Gamma_T$ we
have $\xi_T(z_j)(1-\gamma_T) h_T(z_j) +2\xi_T(z_{\kappa(T)})\delta
_T \leq(\log\log T)/\log^\psi T$. Thus, as $\psi> \frac{1}2$, the
above is at most
\[
\sum_{j=1}^{\kappa(T)} G_T(j,z_j,t_j,t_\infty)
e^{a(T)T(\xi
_T(z_j)(s-h_T(z_j) - q\llvert  z-z_j\rrvert  ) + \sklfrac{1}8 \log^{-1/2} T)}.
\]
However, $\xi_T(z_j)(s-h_T(z_j)) - q\llvert  z-z_j\rrvert   \leq m_T(z,s)$, and on
$\Gamma_T$, $\kappa(T)\leq K_T$, and Lemma~\ref{leGK} tells us that
for each $j\leq\kappa(T)$,
\[
G_T(j,z_j,t_j,t_\infty)\leq
\exp \bigl(a(T)T\bigl(q(\gamma_T-1)\eta_T/3 +
2K_Tq\beta_T\bigr) \bigr).
\]
As in the proof of Proposition~\ref{prearly}, it is easy to check
that this is at most $e^{-2T}\leq1$. Putting all of this together, we get
\begin{eqnarray*}
&&\sum_{j=1}^{\kappa(T)} G_T(j,z_j,t_j,t_\infty)
e^{a(T)T\xi
_T(z_j)(s-t_j+\delta_T)} J_T\bigl(s,\llvert z-z_j\rrvert
\bigr)\1_{\{t_j \leq s\}}
\\
&&\qquad \leq K_T \exp \biggl(a(T)T \biggl( m_T(z,s) +
\frac{1}8 \log^{-1/2} T \biggr) \biggr)
\\
&&\qquad \leq\exp \biggl(a(T)T \biggl( m_T(z,s) + \frac{1}4
\log^{-1/2} T \biggr) \biggr).
\end{eqnarray*}

Finally, by Lemma~\ref{leexpno1} and the above two calculations,
%\begin{eqnarray*}
%
\begin{eqnarray*}
&& E^{\xi} \Bigl[\sup_{u\in(s-\delta_T,s]}\#\bigl\{v\in Y(uT)\dvtx
X_v(uT) = r(T)z\bigr\}\1_{\mathcal H_T} \Bigr]
\\
&&\qquad \leq \exp \biggl( a(T)T\biggl(m_T(z,s) + \frac{1}4
\log^{-1/2} T\biggr) \biggr)
\\
&&\quad\qquad{}  + \exp \biggl(a(T)T \biggl( m_T(z,s) + \frac{1}4
\log^{-1/2} T \biggr) \biggr)
\\
&&\qquad \leq\exp \biggl( a(T)T\biggl(m_T(z,s) +\frac{1}2
\log^{-1/2} T \biggr) \biggr).
\end{eqnarray*}
\upqed
\end{pf}

We have now done the hard work, so we can show that the number of
particles at each point $z$ behaves more or less as it should.

%
%pr4.12 #&#
\begin{prop}\label{prtoomany}
There exists $T_4$ such that for all $T\geq T_4$,
\begin{eqnarray*}
&& \P\bigl(\exists u\in(0,t_\infty],\exists z\in
L_T\dvtx  M_T(z,u)> m_T(z,u)+
\log^{-1/2} T\bigr)
\\
&&\qquad \leq 2\P\bigl(\Gamma_T^c\bigr) + 2e^{-T}
\to0.
\end{eqnarray*}
\end{prop}

\begin{pf}
Since, for any $z$, $m_T(z,u)$ is increasing in $u$, Markov's
inequality and Lemma~\ref{leexpno2} tell us that if $s\leq t_\infty$
and $T\geq T_3$, then on $\Gamma_T$
\begin{eqnarray*}
&&P^{\xi}\bigl(\exists u\in(s-\delta_T,s]\dvtx  N
\bigl(r(T)z,uT\bigr)>e^{a(T)T(m_T(z,u)+\log^{-1/2} T)}, \calH_T\bigr)
\\
&&\qquad \leq E^{\xi} \Bigl[\sup_{u\in(s-\delta_T,s]} N\bigl(r(T)z,uT\bigr)
\1 _{\mathcal H_T} \Bigr] e^{-a(T)T(m_T(z,s-\delta_T)+\log^{-1/2}T)}
\\
&&\qquad \leq e^{a(T)T(m_T(z,s) + \sklfrac{1}2 \log^{-1/2} T - m_T(z,s-\delta
_T)-\log^{-1/2} T)}.
\end{eqnarray*}
By the definition of $m_T$, on $\Gamma_T$ we have since $\psi> \frac{1}2$
\begin{eqnarray*}
m_T(z,s) & \leq& m_T(z,s-\delta_T) +
\sup_{y\dvtx h_T(y)\leq s} \xi_T(y)\delta_T
\\
& \leq& m_T(z,s-\delta_T) + (\log\log T)/\bigl(3
\log^\psi T\bigr)
\\
& \leq& m_T(z,s-\delta_T) + \frac{1}4
\log^{-1/2}T,
\end{eqnarray*}
so
\begin{eqnarray*}
&& P^{\xi}\bigl(\exists u\in(s-\delta_T,s]\dvtx  N\bigl(r(T)z,uT
\bigr)>e^{a(T)T(m_T(z,u)+\log^{-1/2} T)}, \calH_T\bigr)
\\
&&\qquad \leq e^{-\sklfrac{1}4a(T)T\log^{-1/2} T}.
\end{eqnarray*}
Written in terms of $M_T$, this is (on $\Gamma_T$)
\begin{eqnarray*}
&& P^{\xi}\biggl(\exists u\in(s-\delta_T,s]\dvtx
M_T(z,u)> m_T(z,u)+\frac
{1}{\log^{1/2} T},
\calH_T\biggr)
\\
&&\qquad \leq e^{-\sklfrac{1}4 a(T)T\log^{-1/2} T}.
\end{eqnarray*}
Taking a union over $s=\delta_T, 2\delta_T,\ldots, \lceil t_\infty
/\delta_T\rceil\delta_T$ and $z\in L_T(\rho_T)$, on $\Gamma_T$ we have
\begin{eqnarray*}
&& P^{\xi}\bigl(\exists u\in(0,t], \exists z\in L_T(
\rho_T)\dvtx  M_T(z,u)> m_T(z,u)+
\log^{-1/2} T, \calH_T\bigr)
\\
&&\qquad \leq\frac{t_\infty+1}{\delta_T} C_d r(T)^d\rho_T^d
e^{-\sklfrac{1}4a(T)T \log^{-1/2} T} \leq e^{-T}.
\end{eqnarray*}
On $\Gamma_T$, we have $\gamma_T h_T(z)-\delta_T \geq t_\infty$ for
all $z\in L_T\setminus L_T(\rho_T)$ [see (\ref{eqUT4bd})], and
therefore on $\calH_T$ there cannot exist $z\in L_T\setminus L_T(\rho
_T)$ such that $M_T(z,u)> 0$ for any $u\leq t_\infty$. This allows us
to change $z\in L_T(\rho_T)$ to $z\in L_T$ in the estimate above.
Thus, applying Proposition~\ref{prearly}, for large $T$,
\begin{eqnarray*}
&& \P\bigl(\exists u\in(0,t_\infty], \exists z\in L_T\dvtx
M_T(z,u)> m_T(z,u)+\log^{-1/2} T\bigr)
\\
&&\qquad \leq\P\bigl(\Gamma_T^c\bigr) + \P\bigl(
\calH_T^c\bigr) + e^{-T} \leq2\P\bigl(
\Gamma_T^c\bigr) + 2e^{-T} \to0.
\end{eqnarray*}
\upqed
\end{pf}

%s5 #&#
\section{Lower bounds}\label{selate}
We now turn our attention to lower bounds on hitting times and the
number of particles. The key is to check that if $y$ has reasonably
large potential, and we start with lots of particles at $y$, then we
can travel to $z$ in time roughly $q\llvert  z-y\rrvert  /\xi_T(y)$. We do this in
Lemma~\ref{le2ptgrowth}. Since we are not likely to start from a site
with large potential, we must also check that things behave well near
the origin, which is carried out in Lemmas~\ref{lenear0} and~\ref
{levia}. These results are then applied in Section~\ref{sesublate}
to check that $H_T(z)$ is not too much larger than $h_T(z)$, and in
Section~\ref{setoofew} to ensure that there are never too few
particles at a site.

Let $\mu_T = \log^{1/4} T$. Then for $z\in L_T$, define
\[
H'_T(z) = \inf\bigl\{t>0\dvtx  N\bigl( r(T)z, tT\bigr) >
\exp(\mu_T)\bigr\}.
\]
To avoid the randomness that occurs when we only have a few particles,
we work with $H'_T(z)$ instead of $H_T(z)$.

%s5.1 #&#
\subsection{Preliminary estimates for the lower bounds}

We start by checking that if we start with lots of particles at a site
$y$, the
number of particles grows as expected.

%
%le5.1 #&#
\begin{lem}\label{le1ptgrowth}
Take\vspace*{2pt} $y\in L_T$ such that $\xi_T(y)\geq\frac{\mu_T^3}{a(T)}$ and
choose $p\in[1/2,1]$. Then for any $s\geq0$, for large $T$ (depending
only on $d$), for any $s \geq0$,
\begin{eqnarray*}
&& P^{\xi} \biggl(N \biggl(r(T)y, H'_T(y)T+
\biggl(1+ \frac
{p}{\mu_T^2}\biggr)sT \biggr) \leq\frac{1}{64}e^{a(T)T \xi_T(y)s
(1+p/(2\mu_T^2)) + \mu_T}
\biggr)
\\
&&\qquad < \frac{1}2e^{-\log^2 T}.
\end{eqnarray*}
\end{lem}

\begin{pf}
By definition there are $e^{\mu_T}$ particles at $r(T)y$ at time
$H'_T(y)T$. By Lemma~\ref{2m}, we expect at least $\frac{1}{16}\exp
(\mu_T)$ of these to have at least
\[
\frac{1}2 \exp \biggl(\bigl(\xi_T(y)a(T)-2d\bigr)
\biggl(1+ \frac{p}{\mu
_T^2}\biggr)sT \biggr)
\]
descendants at $r(T)y$ at time $H_T'(y)T + (1+\frac{p}{\mu_T^2})sT$.
Note that when\vspace*{-5pt} $T$ is large, since $\xi_T(y)\geq\frac{\mu
_T^3}{a(T)}$, we have
\begin{eqnarray*}
&& \bigl(\xi_T(y)a(T)-2d\bigr) \biggl(1+ \frac
{p}{\mu_T^2}\biggr)
\\
&&\qquad  \geq\xi_T(y)a(T)
+ \xi_T(y)a(T) \frac
{p}{2\mu_T^2} +  p
\mu_T - 2d\biggl(1+ \frac{p}{\mu
_T^2}\biggr)
\\
&&\qquad  \geq \xi_T(y)a(T) \biggl(1+ \frac{p}{2\mu_T^2}\biggr).
\end{eqnarray*}
The result now follows from Lemma~\ref{lechernoff} since $\exp
(-\frac{1}{8}\cdot\frac{1}{16}\exp(\mu_T) ) \leq\break \frac{1}2 \exp
(-\log^2 T)$ when $T$ is large.
\end{pf}

Now that we know that the number of particles at $y$ grows as expected,
we can check that particles move from $y$ to $z$ in time roughly
$q\llvert  z-y\rrvert  /\xi_T(y)$.

%
%le5.2 #&#
\begin{lem}\label{le2ptgrowth}
Take $y, z\in L_T$ such that $y\neq z$ and $\frac{\mu_T^3}{a(T)} \leq
\xi_T(y)\leq\exp(\mu_T)$. Suppose that $s\geq q\llvert  z-y\rrvert  /\xi_T(y)$.
For large $T$ (depending only on $d$ and $q$),
\[
P^{\xi}\biggl(N \biggl(r(T)z, H'_T(y)T+
\biggl(1+ \frac
{1}{\mu_T^2}\biggr)sT \biggr) \leq e^{a(T)T (\xi_T(y)s-q\llvert  z-y\rrvert   )
+\mu_T}
\biggr) < e^{-\log^2 T}.
\]
\end{lem}

Note in particular that if we apply this result at time $s=q\llvert  z-y\rrvert  /\xi
_T(y)$, then we already see $\exp(\mu_T)$ particles.

\begin{pf*}{Proof of Lemma \ref{le2ptgrowth}}
By Lemma~\ref{le1ptgrowth}, for large $T$,
\begin{eqnarray*}
&&P^{\xi} \biggl(N \biggl(r(T)y, H'_T(y)T+
\biggl(1+ \frac
{1}{2\mu_T^2}\biggr)sT \biggr) \leq\frac{1}{64}e^{a(T)T \xi_T(y)s
(1+1/(4\mu_T^2)) + \mu_T}
\biggr)
\\
&&\qquad < \frac{1}2e^{-\log^2 T}.
\end{eqnarray*}
By Lemma~\ref{rwlem},
\begin{eqnarray*}
&& P \biggl(X \biggl(\frac{s T}{2\mu_T^2} \biggr) = r(T) (z-y) \biggr)
\\
&&\qquad \geq \exp
\biggl(-a(T)T\biggl(q\llvert z-y\rrvert + \calE^1_T
\biggl(\frac{s}{2\mu
_T^2},\llvert z-y\rrvert \biggr)\biggr) \biggr).
\end{eqnarray*}
In words, we have a large number of particles at $r(T)y$ just before
the time we are interested in, and each has a reasonable probability of
being at $r(T)z$ at the time we are interested in. Applying Lemma~\ref
{lechernoff}, we are done, provided that
\begin{eqnarray*}
&& \tfrac{1}{128}e^{a(T)T \xi_T(y)s (1+1/(4\mu_T^2)) +
\mu_T - a(T)T(q\llvert  z-y\rrvert   + \calE^1_T(s/(2\mu_T^2),\llvert  z-y\rrvert  ))}
\\
&&\qquad \geq e^{a(T)T (\xi_T(y)s-q\llvert  z-y\rrvert   ) +\mu_T},
\end{eqnarray*}
which reduces to showing that
\[
\tfrac{1}{128}e^{a(T)T (\xi_T(y)s/(4\mu_T^2) - \calE^1_T(s/(2\mu
_T^2),\llvert  z-y\rrvert  ))} \geq1.
\]
However, since $\llvert  z-y\rrvert   \leq\xi_T(y)s/q \leq s\exp(\mu_T)/q$ and $\xi
_T(y)\geq\frac{\mu_T^3}{a(T)}$, for large $T$,
\begin{eqnarray*}
\calE^1_T \biggl(\frac{s}{2\mu_T^2},\llvert z-y\rrvert
\biggr) & =& \frac
{\llvert  z-y\rrvert  }{\log T} \biggl(\log\llvert z-y\rrvert - \log \biggl(
\frac{s}{2\mu
_T^2} \biggr) \biggr) + \frac{ds}{\mu_T^2 a(T)}
\\
& \leq&\frac{\xi_T(y)s}{q\log T}(\mu_T - \log q + \log2 + 2\log\mu
_T) + \frac{\xi_T(y)s}{q\log T}\cdot\frac{dq}{\mu_T}
\\
& \leq&\frac{2\xi_T(y)s}{q\mu_T^3}.
\end{eqnarray*}
Thus for large $T$,
\[
\tfrac{1}{128}e^{a(T)T (\xi_T(y)s/(4\mu_T^2) - \calE^1_T(s/(2\mu
_T^2),\llvert  z-y\rrvert  ))} \geq\tfrac{1}{128}e^{a(T)T \xi_T(y)s/(8\mu_T^2)},
\]
but since $y\neq z$, we have $\llvert  y-z\rrvert   \geq1/r(T)$ so that
\[
a(T)T \frac{\xi_T(y)s}{8\mu_T^2} \geq a(T)T \frac{q\llvert  z-y\rrvert  }{8\mu
_T^2} \geq\frac{q\mu_T^2}{8}
\geq\log (128 ).
\]
\upqed
\end{pf*}

We now know that if there are lots of particles at $y$ and $y$ has
reasonable potential, then we can travel from $y$ to any other point
$z$ in a suitable time. We now make sure that there are some
points---indeed, all points close enough to the origin---with lots of
particles. We make no attempt to optimize our argument and use only
simple estimates.

%
%le5.3 #&#
\begin{lem}\label{lenear0}
For any $\phi>0$, $\P$-almost surely, for all $T>e$,
\[
P^{\xi}\bigl(\exists z\in B\bigl(0,\log^\phi T\bigr)\dvtx  N
\bigl(z,5\log^{2\phi}T\bigr) < e^{\lambda\log^\phi T}\bigr) < c
e^{-\lambda\log^\phi T},
\]
where $\lambda= \frac{\log2}{8(d+1)}$ and $c$ is a constant
depending only on $d$ and $\phi$.
\end{lem}

\begin{pf}
We write $P^1$ to mean the law of a BRW with branching rate $1$
everywhere. Since $\P(\xi(z)\geq1)=1$, if we can prove the lemma
under $P^1$, then by an easy coupling it must hold for $\P$-almost
every environment $\xi$. Also by adjusting $c$ it suffices to consider
large~$T$.

Note that by Lemma~\ref{yulelem},
\[
P^1 \bigl(N \bigl(4\lambda\log^\phi T \bigr) <
e^{2\lambda\log
^\phi T} \bigr) \leq ce^{-\lambda\log^\phi T}.
\]
However, since $\llvert  X(4\lambda\log^{\phi} T)\rrvert  $ is stochastically
dominated by a Poisson random variable of parameter $8d\lambda\log
^{\phi} T$, we have
\begin{eqnarray*}
P \bigl(\bigl\llvert X\bigl(4\lambda\log^{\phi} T\bigr)\bigr\rrvert >
\log^{\phi} T \bigr) &\leq& E\bigl[e^{\llvert  X(4\lambda\log^{\phi} T)\rrvert  \log2}\bigr]e^{-(\log
2)\log^{\phi} T}
\\
&=& e^{8d\lambda\log^\phi T - (\log2)\log^\phi T} = e^{-8\lambda\log
^\phi T},
\end{eqnarray*}
and applying the many-to-one lemma,
\begin{eqnarray*}
&& P^1 \bigl(\exists v\in Y \bigl(4\lambda\log^{\phi} T \bigr)\dvtx  \bigl\llvert X_v\bigl(4\lambda\log^{\phi} T\bigr)\bigr
\rrvert > \log^{\phi} T \bigr)
\\
&&\qquad\leq e^{4\lambda\log^{\phi} T} P \bigl(\bigl\llvert X\bigl(4\lambda\log^{\phi}
T\bigr)\bigr\rrvert > \log^{\phi} T \bigr) \leq e^{-4\lambda\log^{\phi} T}.
\end{eqnarray*}
So (adjusting $c$ as necessary) with probability at least
$1-ce^{-\lambda\log^\phi T}$ we have at least $e^{2\lambda\log^\phi
T}$ particles spread over $B(0,\log^\phi T)$ at time $4\lambda\log
^\phi T$.

Take one such particle $v$. We now wait a further time $5\log^{2\phi}
T-4\lambda\log^\phi T$, which is at least $(2\log^\phi T)^2$ when
$T$ is large. For any $z\in B(0,\log^\phi T)$, by Lemma~\ref
{rwlem3} the probability that $v$ has a descendant at $z$ at this
time is at least $c'\log^{-d\phi} T$ for some constant $c'$. Thus, by
Lemma~\ref{lechernoff},
\begin{eqnarray*}
&& P^1 \bigl(N\bigl(z,5\log^{2\phi}T\bigr) <
\bigl(c'/2\bigr)e^{2\lambda
\log^\phi T}\log^{-d\phi} T \bigr)
\\
&&\qquad< \exp \bigl(-\bigl(c'/8\bigr)e^{2\lambda\log^\phi T}
\log^{-d\phi} T \bigr).
\end{eqnarray*}
Taking a union over all $z\in B(0,\log^\phi T)\cap\Z^d$, we get the
desired result.
\end{pf}

We\vspace*{1pt} have now established that with high probability every site within
$B(0,\log^\phi T)$ has lots of particles by time $5\log^{2\phi} T$.
Using Lemma~\ref{le0306-1} we can ensure that at least one of these
sites---call it $z_0$---has reasonably large potential. The small
problem we face is that in the definition of $h_T$, our trail of points
starts from $0$ and not from $z_0$. The following lemma helps us to get
around this fact, essentially by stating that travelling via $z_0$ does
not cost much.

%
%le5.4 #&#
\begin{lem}\label{levia}
Fix $T>e$. Suppose that $z\in L_T$ and $h_T(z)=\sum_{j=1}^n
q\llvert  y_{j-1}-y_j\rrvert  /\xi_T(y_j)$ where $y_0=z$ and $y_n=0$. Let
\[
n' = \cases{ 0, &\quad if $\displaystyle
\xi_T(y_j)< \mu _T^3/a(T)
~\forall j\geq1$,
\vspace*{3pt}\cr
\displaystyle\max\bigl\{j\geq1\dvtx
\xi_T(y_j)\geq\mu_T^3/a(T)
\bigr\}, &\quad otherwise.}
\]
For any $z_0\in L_T$ such that $\xi_T(z_0)\geq2\mu_T^3/a(T)$, we have
\[
q\frac{\llvert  y_{n'}\rrvert  }{\xi_T(z_0)} \leq h_T(z_0) + q
\frac{\llvert  z_0\rrvert  }{\xi_T(z_0)}.
\]
\end{lem}

\begin{pf}
First note that since $\max_{j>n'} \xi_T(y_j) \leq\frac{1}2\xi
_T(z_0)$, by Lemma~\ref{leincrease} and the triangle inequality,
\[
h_T(y_{n'}) = \sum_{j=n'+1}^n
q\frac{\llvert  y_{j-1}-y_j\rrvert  }{\xi_T(y_j)} \geq 2q\frac{\llvert  y_{n'}\rrvert  }{\xi_T(z_0)},
\]
but also,
\[
h_T(y_{n'}) \leq h_T(z_0) +
q\frac{\llvert  y_{n'} - z_0\rrvert  }{\xi_T(z_0)} \leq h_T(z_0) + q
\frac{\llvert  y_{n'}\rrvert  }{\xi_T(z_0)} + q\frac{\llvert  z_0\rrvert  }{\xi_T(z_0)}.
\]
Combining these two statements, we see that
\[
q\frac{\llvert  y_{n'}\rrvert  }{\xi_T(z_0)} \leq h_T(z_0) + q
\frac{\llvert  z_0\rrvert  }{\xi_T(z_0)}.
\]
\upqed
\end{pf}

%s5.2 #&#
\subsection{Particles do not arrive too late}\label{sesublate}

We are now ready to prove our main result for this section, namely that
the probability that anyone arrives late is small. As we hinted
earlier, we will apply Lemma~\ref{le2ptgrowth} at time $q\llvert  z-y\rrvert  /\xi
_T(y)$ to check that we move from $y$ to $z$ in the time allotted. (The
extra work to consider more general $s$ was not wasted, however: it
will be used when we check that the \textit{number} of particles grows as
claimed.) The rest of the proof simply involves tying up some loose ends.

%
%pr5.5 #&#
\begin{prop}\label{prlate}
As $T\to\infty$,
\[
\P \bigl(\exists z\dvtx  H'_T(z)\wedge
t_\infty> h_T(z) + (t_\infty+1)
\log^{-1/2}T \bigr) \to0.
\]
\end{prop}

\begin{pf}
Let $\phi= \frac{3 \alpha}{2d}$ % 2/d + 3\alpha/4 > 2$
and $\lambda=\frac{\log2}{8(d+1)}$. We consider the following four events:
\begin{itemize}
\item$\Gamma_T$: in particular [see (\ref{eqUT4bd})],
$h_T(z)>t_\infty$ for all $z\notin B(0,\rho_T)$;
\item$N(r(T)z, 5\log^{2\phi}T)\geq\exp(\lambda\log^\phi T)$ for
all $z\in L_T(0,r(T)^{-1}\log^\phi T)$;
\item there exists $z_0 \in L_T(0,r(T)^{-1}\log^\phi T)$ such that
$\xi_T(z_0) > \frac{2\mu_T^3}{a(T)}$ and $h_T(z_0)\leq\frac
{1}{2\mu_T^2}$;
\item for all $y\neq z$ in $L_T(0,\rho_T)$ such that $\xi_T(y)\geq
\frac{\mu_T^3}{a(T)}$, we have $H'_T(z)\leq H'_T(y) + (1+\mu
_T^{-2})q\frac{\llvert  z-y\rrvert  }{\xi_T(y)}$.
\end{itemize}
We recall that $\P(\Gamma_T)\to1$ as $T\to\infty$. By Lemma~\ref
{lenear0}, the probability of the second event also tends to $1$ as
$T\to\infty$. By Lemma~\ref{le0306-1} and our choice of $\phi$,
together with Corollary~\ref{cosmalllily}, the probability of the
third event also tends to $1$ as $T\to\infty$. Finally, by applying
Lemma~\ref{le2ptgrowth} when $s=q\llvert  z-y\rrvert  /\xi_T(y)$ [note that on
$\Gamma_T$, there exists $c$ such that $\xi_T(y)\leq c\log\log T$
for\vspace*{1pt} all $y\in L_T(0,\rho_T)$] together with the fact that there are at
most $c_d^2 r(T)^{2d}\rho_T^{2d}\ll\exp(\log^2 T)$ pairs of points
$y,z\in L_T(0,\rho_T)$, the probability of the fourth event tends to
$1$ as $T\to\infty$. Thus it suffices to prove that on the
intersection of these four events, for every $z$, we have
$H'_T(z)\wedge t_\infty\leq h_T(z) + (t_\infty+1)/\mu_T^2$. In
particular, we can assume that $h_t(z) \leq t_\infty$.\vadjust{\goodbreak}

Choose $n$ and distinct $y_0,\ldots,y_n$ such that $y_0=z$, $y_n=0$
and $h_T(z) = \sum_{j=1}^n q\llvert  y_{j-1}-y_j\rrvert  /\xi_T(y_j)$. Let $n'$ be as
in Lemma~\ref{levia}, and note that by Lemma~\ref{leincrease}, $\xi
_T(y_j) \geq\xi_T(y_{n'})\geq\mu_T^3/a(T)$ for $1 \leq j \leq n'$.
Since we are working on $\Gamma_T$, we may assume that $y_j\in
B(0,\rho_T)$ for all $j$. Then from the fourth event above,
\begin{eqnarray*}
H_T'(z) & =& \sum_{j=1}^{n'}
\bigl(H_T'(y_{j-1}) - H_T'(y_j)
\bigr) + H_T'(y_{n'}) -
H_T'(z_0) + H_T'(z_0)
\\
& \leq&\sum_{j=1}^{n'} \bigl(1+
\mu_T^{-2}\bigr)q\frac{\llvert  y_{j-1}-y_j\rrvert  }{\xi
_T(y_j)} + \bigl(1+
\mu_T^{-2}\bigr)q\frac{\llvert  y_{n'}-z_0\rrvert  }{\xi_T(z_0)} +
H_T'(z_0)
\\
& \leq&\bigl(1+\mu_T^{-2}\bigr)h_T(z) +
\bigl(1+\mu_T^{-2}\bigr)q\frac
{\llvert  y_{n'}\rrvert  +\llvert  z_0\rrvert  }{\xi_T(z_0)} +
H_T'(z_0).
\end{eqnarray*}
By\vspace*{1pt} Lemma~\ref{levia} we have $q\llvert  y_{n'}\rrvert  /\xi_T(z_0)\leq h_T(z_0) +
q\llvert  z_0\rrvert  /\xi_T(z_0)$. We know from the second event above that
$H'_T(z_0)\leq\frac{5\log^{2\phi}T}{T}$, and from the third event that
$z_0$ is chosen such that $\llvert  z_0\rrvert   \leq\frac{\log^\phi T}{r(T)}$,
$h_T(z_0)\leq\frac{1}{2\mu_T^2}$ and $\xi_T(z_0) > \frac{2 \mu
_T^3}{a(T)}$. Thus when $T$ is large,%\looseness=-1
\begin{eqnarray*}
H_T'(z) & \leq&\bigl(1+\mu_T^{-2}
\bigr) \biggl(h_T(z) + h_T(z_0) + 2q
\frac
{\llvert  z_0\rrvert  }{\xi_T(z_0)} \biggr) + H_T'(z_0)
\\
& \leq&\bigl(1+\mu_T^{-2}
\bigr)h_T(z) + \bigl(1+ \mu_T^{-2}\bigr)
\biggl(\frac{1}{2\mu
_T^2} + 2 q \frac{(\log^\phi T)}{r(T)} \frac{a(T)}{ 2 \mu_T^3} \biggr) +
5 \frac{\log^{2\phi} T }{ T}
\\
& \leq&\bigl(1+\mu_T^{-2}\bigr)h_T(z) +
\mu_T^{-2} \leq h_T(z) + (t_\infty
+1)/\mu_T^{2},
\end{eqnarray*}
where we recall that $a(T) / r(T) = \log T / T$ and $\mu_T = \log
^{1/4} T$.
\end{pf}

%s5.3 #&#
\subsection{There are not too few particles}\label{setoofew}
We now want to show that with high probability, there are at least
about $m_T(z,s)$ particles at each site $z$, for all $s\leq t_\infty$.
Again our main tool will be Lemmas~\ref{le1ptgrowth} and~\ref
{le2ptgrowth}. Since these lemmas apply only at fixed times, and we
want to be sure that there is \textit{never} a time when the number of
particles is too small, we start by translating into continuous time.

%
%le5.6 #&#
\begin{lem}\label{lects2ptgrowth}
Take $y, z\in L_T$ such that $\frac{\mu_T^3}{a(T)} \leq\xi_T(y)
\leq\mu_T$. Then for large $T$,
\begin{eqnarray*}
&& P^{\xi} \biggl(\exists s\in \biggl[q\frac{\llvert  z-y\rrvert  }{\xi_T(y)},t_\infty
\biggr]\dvtx
\\
&&\hspace*{21pt}{} N \bigl(r(T)z, H'_T(y)T+\bigl(1+
\mu_T^{-2}\bigr)sT \bigr) < e^{a(T)T (\xi_T(y)s-q\llvert  z-y\rrvert   )} \biggr)
\\
&&\qquad < \exp\biggl(- \frac{1}2\log^2 T\biggr).
\end{eqnarray*}
\end{lem}

\begin{pf}
Within this proof only, we will use the shorthand
\[
N_s = N \bigl(r(T)z, H'_T(y)T+\bigl(1+
\mu_T^{-2}\bigr)sT \bigr)
\]
and
\[
E_{s,p} = \exp \bigl(a(T)T \bigl(\xi_T(y)s-q\llvert z-y
\rrvert \bigr) + p\mu _T \bigr).
\]
For each $j\geq0$, let
\[
s_j = q\frac{\llvert  z-y\rrvert  }{\xi_T(y)} + \frac{j}{4a(T)T}.
\]
Our plan is to apply Lemmas~\ref{le1ptgrowth} and~\ref{le2ptgrowth}
with $s=s_j$ for each $j$, and then to show that the number of
particles cannot drop suddenly between $s_j$ and $s_{j+1}$ for any $j$.

For any $j$,
\[
a(T)T\xi_T(y) (s_{j+1}-s_j) =
\xi_T(y)/4\leq\mu_T/4.
\]
Thus if $k= \lfloor4a(T)Tt_\infty\rfloor$,
\begin{eqnarray*}
&& P^\xi\bigl(\exists u\in\bigl[q\llvert z-y\rrvert /
\xi_T(y),t_\infty\bigr]\dvtx  N_u <
E_{u,0}\bigr)
\\
&&\qquad  \leq\sum_{j=0}^{k} P^\xi
\bigl(\exists u\in[s_j,s_{j+1}\bigr)\dvtx  N_u <
E_{s_j,
1/4})
\\
&&\qquad \leq\sum_{j=0}^{k} P^\xi(N_{s_j}
< E_{s_j,1/2})
\\
&&\quad\qquad{} + \sum_{j=0}^{k} P^\xi
\bigl(\exists u\in[s_j,s_{j+1}\bigr)\dvtx  N_u <
E_{s_j, 1/4} \mid N_{s_j} \geq E_{s_j,1/2}).
\end{eqnarray*}
A simple application of Lemma~\ref{le1ptgrowth} (if $y=z$) or Lemma
\ref{le2ptgrowth} (if $y\neq z$) tells us that the first sum is at
most $(4a(T)Tt_\infty+1)\exp(-\log^2 T)$ when $T$ is large, and so
it suffices to prove the same for the second sum.

Let
\[
I_j = \bigl[H'_T(y)T+\bigl(1+
\mu_T^{-2}\bigr)s_jT, H'_T(y)T+
\bigl(1+\mu _T^{-2}\bigr)s_{j+1}T \bigr).
\]
Note that when $T$ is large, $\llvert  I_j\rrvert  \leq\frac{1}{2a(T)}$ for each $j$.
Given that $N_{s_j} \geq E_{s_j,1/2}$, and the probability that a
particle does not move during an interval of length $\frac{1}{2a(T)}$
is $\exp(-d/a(T))$, we expect at least $E_{s_j,1/2} \exp(-d/a(T))$
particles to remain at $r(T)z$ throughout the interval $I_j$. By Lemma
\ref{lechernoff}, the event that
the number of particles that actually stay is less
than $\frac{1}2 E_{s_j,1/2} \exp(-d/a(T))$ % actually do so
has probability at most $\exp(-\frac{1}8 E_{s_j,1/2} \exp(-d/a(T)))
< \exp(-\log^2 T)$. Since
$\frac{1}2 E_{s_j,1/2} \exp(-d/a(T)) > E_{s_j,1/4}$ when $T$ is large,
this completes the proof.
\end{pf}

%
%pr5.7 #&#
\begin{prop}\label{prgrowth}
As $T\to\infty$,
\[
\P\bigl(\exists s\in[0,t_\infty], z\dvtx  M_T(z,s) <
m_T(z,s) - \log^{-1/4} T\bigr) \to0.
\]
\end{prop}

\begin{pf} We may assume without loss of generality that $t_\infty\geq1$.
We work on $\Gamma_T$ and assume that
the event
\[
\calH'_T = \bigl\{ H'_T(y)
\wedge t_\infty\leq h_T(y) + (t_\infty+1)\mu
_T^{-2} \mbox{ for all } y \bigr\}
\]
holds. We know that $\P(\Gamma_T)\to1$, and Proposition~\ref
{prlate} tells us
that $\P( \calH_T') \rightarrow1$, so it suffices to prove our
result under these conditions. In particular we may restrict to $z\in
L_T(0,\rho_T)$.

Fix such a site $z$. Note that on $\Gamma_T$,
\begin{eqnarray*}
&&P^{\xi}\bigl(\exists s\in[0,t_\infty]\dvtx  M_T(z,s)
< m_T(z,s) - \mu _T^{-1},
\calH'_T\bigr)
\\
&&\qquad\leq \sum_{y\in L_T(0,\rho_T)} P^{\xi} \bigl(\exists s
\in [0,t_\infty]\dvtx N\bigl(r(T)z, sT\bigr)\vee1
\\
&&\hspace*{97pt}  < e^{a(T)T(\xi
_T(y)(s-h_T(y))_+-q\llvert  z-y\rrvert  - \mu_T^{-1})},
\calH_T' \bigr).
\end{eqnarray*}
Observe that if $\xi_T(y)<\frac{\mu_T^3}{a(T)}$,
then $\xi_T(y)(t_\infty-h_T(y))-q\llvert  z-y\rrvert  - \mu_T^{-1}<0$ for large $T$,
so the probability above is zero. Recall also that on $\Gamma_T$,
there exists a constant $c$ such that $\xi_T(y)\leq c\log\log T$ for
all $y\in L_T(0,\rho_T)$. We deduce that we may restrict the sum above
to $y$ such that $\frac{\log^{3/4}T}{a(T)} \leq\xi_T(y) \leq c\log
\log T$.

Fix such a site $y\in L_T$. If $s\leq H_T'(y)\wedge t_\infty+ (1+\mu
_T^{-2})q\frac{\llvert  z-y\rrvert  }{\xi_T(y)}$, then on $\calH_T'$ we have
\[
s\leq h_T(y) + (t_\infty+1)\mu_T^{-2}
+ \bigl(1+\mu_T^{-2}\bigr)q\frac
{\llvert  z-y\rrvert  }{\xi_T(y)},
\]
so
\begin{eqnarray*}
&& \xi_T(y) \bigl(s-h_T(y)\bigr) -  q\llvert z-y
\rrvert - \mu_T^{-1}
\\
&&\qquad \leq(t_\infty+1)\mu_T^{-2}
\xi_T(y) + \mu_T^{-2}q\llvert z-y\rrvert -
\mu _T^{-1}
\\
&&\qquad \leq(t_\infty+1)\mu_T^{-2}c\log\log T + 2
\mu_T^{-2}q\log\log T - \mu_T^{-1}
\end{eqnarray*}
which is negative for large $T$. Thus
\begin{eqnarray*}
&& P^{\xi} \biggl(\exists s\in \biggl[0,H_T'(y)
\wedge t_\infty+ \bigl(1+\mu _T^{-2}\bigr)q
\frac{\llvert  z-y\rrvert  }{\xi_T(y)} \biggr]\dvtx
\\
&&\hspace*{21pt} N\bigl(r(T)z, sT\bigr)\vee1 < e^{a(T)T(\xi_T(y)(s-h_T(y))_+-q\llvert  z-y\rrvert  - \mu
_T^{-1})}, \calH_T'
\biggr) = 0.
\end{eqnarray*}
As a result, it suffices to look at $s\geq H_T'(y) + (1+\mu
_T^{-2})q\frac{\llvert  z-y\rrvert  }{\xi_T(y)}$, and by substituting in $u = \frac
{s-H'_T(y)}{1+\mu_T^{-2}}$ we get that
\begin{eqnarray*}
&& P^{\xi} \biggl(\exists s\in \biggl[H_T'(y)
+ \bigl(1+\mu_T^{-2}\bigr)q\frac
{\llvert  z-y\rrvert  }{\xi_T(y)},t_\infty
\biggr]\dvtx
\\
&&\hspace*{21pt} N\bigl(r(T)z, sT\bigr)\vee1 < e^{a(T)T(\xi
_T(y)(s-h_T(y))_+-q\llvert  z-y\rrvert  - \mu_T^{-1})},
\calH_T' \biggr)
\\
&&\qquad \leq P^{\xi} \biggl(\exists u\in \biggl[q\frac{\llvert  z-y\rrvert  }{\xi
_T(y)},t_\infty
\biggr]\dvtx  N\bigl(r(T)z, H_T'(y)T+\bigl(1+
\mu_T^{-2}\bigr)uT\bigr)
\\
&&\hspace*{31pt}\qquad < e^{a(T)T(\xi_T(y)(u + u\mu_T^{-2} + H'_T(y)-h_T(y))-q\llvert  z-y\rrvert  - \mu
_T^{-1})},%\\
\calH_T',
H_T'(y) < t_\infty \biggr)
\\
&&\qquad \leq P^{\xi} \biggl(\exists u\in \biggl[q\frac{\llvert  z-y\rrvert  }{\xi
_T(y)},t_\infty
\biggr]\dvtx
\\
&&\hspace*{54pt} N\bigl(r(T)z, H_T'(y)T+\bigl(1+
\mu_T^{-2}\bigr)uT\bigr) < e^{a(T)T(\xi
_T(y)u - q\llvert  z-y\rrvert  )} \biggr).
\end{eqnarray*}
By Lemma~\ref{lects2ptgrowth}, this is at most $\exp(-\frac{1}2\log
^2 T)$, and since there are at most $c_d^2 r(T)^2 \rho_T^2 \ll\exp
(\frac{1}2\log^2 T)$ suitable pairs of points $y,z$, the result follows.
\end{pf}

%s6 #&#
\section{Proofs of Theorems \texorpdfstring{\protect\ref{teolilypad}}{1.1}, \texorpdfstring{\protect\ref{teosupport}}{1.2} and \texorpdfstring{\protect\ref{teointermittency}}{1.3}}\label{seproofs}

It now remains to draw the results of the previous sections together.

\begin{pf*}{Proof of Theorem~\ref{teolilypad}}
The fact that
\[
\sup_{t\leq t_\infty}\sup_{z\in L_T} \bigl\llvert
M_T(z,t)-m_T(z,t)\bigr\rrvert \to0
\]
in $\P$-probability follows immediately from Propositions~\ref
{prtoomany} and~\ref{prgrowth}. We therefore concentrate on showing
that $\sup_{z\in L_T(0,R)} \llvert  H_T(z)-h_T(z)\rrvert  \to0$ in $\P$-probability.

Fix any $R,\delta,\eps>0$. Clearly it suffices to prove the theorem
when $t_\infty$ is large, and by Lemma~\ref{lesmalllily} by making
$t_\infty$ large we may ensure that $\P(\exists z\in L_T(0,R)\dvtx
h_T(z)\geq t_\infty-\delta) < \eps/2$. Now, by Proposition~\ref
{prlate}, we may choose $T_\infty$ large enough such that for any
$T\geq T_\infty$, we have
\[
\P\bigl(\exists z\dvtx  H'_T(z)\wedge t_\infty>
h_T(z) + \delta\bigr) < \eps/2.
\]
Then for $T\geq T_\infty$,
\begin{eqnarray*}
&&\P\bigl(\exists z\in L_T(0,R)\dvtx  H_T(z) >
h_T(z)+\delta\bigr)
\\
&&\qquad  \leq\P\bigl(\exists z\in L_T(0,R)\dvtx  H'_T(z)
\wedge t_\infty> h_T(z)+\delta\bigr)
\\
&&\quad\qquad{} + \P\bigl(\exists z\in L_T(0,R)\dvtx
h_T(z) \geq t_\infty- \delta\bigr) < \eps.
\end{eqnarray*}
For the lower bound on $H_T(z)$, by increasing $T_\infty$ if necessary
we may assume that for any $T\geq T_\infty$ we have $(1-\gamma
_T)t_\infty+ \delta_T \leq\delta$, where $\gamma_T$ and $\delta
_T$ are as in Section~\ref{secparameters}. By Proposition~\ref
{prearly} we may also ensure that for $T\geq T_\infty$ we have
\[
\P\bigl(\exists z\dvtx  H_T(z) < \bigl(\gamma_T
h_T(z)-\delta_T\bigr)\wedge t_\infty \bigr)<
\eps/2.
\]
Then for $T\geq T_\infty$,
\begin{eqnarray*}
&& \P\bigl(\exists z\in L_T(0,R)\dvtx  H_T(z) <
h_T(z)-\delta\bigr)
\\
&&\qquad \leq\P\bigl(\exists z\in L_T(0,R)\dvtx  h_T(z) >
t_\infty\bigr)
\\
&&\quad\qquad{} + \P\bigl(\exists z\dvtx  H_T(z) < \gamma_T
h_T(z) + (1-\gamma _T)t_\infty- \delta,
h_T(z)\leq t_\infty\bigr)
\\
&&\qquad \leq \eps/2 + \P\bigl(\exists z\dvtx  H_T(z) < \bigl(
\gamma_T h_T(z)-\delta _T\bigr)\wedge
t_\infty\bigr)
\\
&&\qquad < \eps.
\end{eqnarray*}
\upqed
\end{pf*}

In order to prove Theorem~\ref{teosupport} we first show that, with
high probability, the lilypad model does not change rapidly over small
time intervals.

%
%le6.1 #&#
\begin{lem}\label{lelilysupport}
For any $t_\infty,\delta,\eps>0$, there exists $\eta>0$ such that
for all large $T$,
\[
\P \biggl(s_T(t+\eta)\subseteq\bigcup
_{y\in s_T(t)} B(y,\delta) ~\forall t\leq t_\infty
\biggr)\geq1-\eps.
\]
\end{lem}

\begin{pf}
By Lemma~\ref{le0805-1} we may choose $R$ such that for any $T>e$,
\[
\P\bigl(\exists z\in L_T\setminus B(0,R)\dvtx  h_T(z)
\leq t_\infty+1\bigr) < \eps/2.
\]
Then by Lemma~\ref{le0306-1} we may choose $\Upsilon>0$ such that
for any $T>e$,
\[
\P \Bigl(\max_{z\in L_T(0,R)} \xi_T(z) > \Upsilon \Bigr)
< \eps/2.
\]
Now choose $T_\infty>e$ such that $1/r(T_\infty)<\delta/4$, and
choose $\eta< (q\delta/2\Upsilon)\wedge1$. As a result of the above
bounds, for any $T\geq T_\infty$ we have $\P(\exists z\in
s_T(t_\infty+\eta)\dvtx  \xi_T(z)>\Upsilon) < \eps$.

Fix $T>T_\infty$ and $t\leq t_\infty$, and take $x\in s_T(t+\eta
)\setminus s_T(t)$. We will show that if $\xi_T(z)\leq\Upsilon$ for
all $z\in s_T(t+\eta)$, then $d(x,s_T(t))\leq\delta$. Let $w=[x]_T$,
so that $w\in L_T$ and $t\leq h_T(w) \leq t+\eta$. Take $y_0,\ldots
,y_n \in L_T$ such that $y_0=w$, $y_n=0$, $h_T(y_j)\leq t+\eta$ for
all $j$, and
\[
h_T(w) = \sum_{j=1}^n q
\frac{\llvert  y_{j-1}-y_j\rrvert  }{\xi_T(y_j)}.
\]
Let $k = \min\{j\dvtx  y_j\in s_T(t)\}$. Then choose $y\in L_T$ such that
$\llvert  y_{k-1}-y_k\rrvert   = \llvert  y_{k-1}-y\rrvert   + \llvert  y-y_k\rrvert  $ and $y\notin s_T(t)$, but
$d(y,s_T(t))\leq1/r(T)$. That is, $y$ is the first point in $L_T$ on
the geodesic between $y_k$ and $y_{k-1}$ that is outside $s_T(t)$.
(There may be more than one such point, but any will do.) Now,
\begin{eqnarray*}
h_T(w) & =& \sum_{j=1}^{k-1} q
\frac{\llvert  y_{j-1}-y_j\rrvert  }{\xi_T(y_j)} + q\frac{\llvert  y_{k-1}-y_k\rrvert  }{\xi_T(y_k)} + \sum_{j=k+1}^n
q\frac
{\llvert  y_{j-1}-y_j\rrvert  }{\xi_T(y_j)}
\\
& =& \sum_{j=1}^{k-1} q\frac{\llvert  y_{j-1}-y_j\rrvert  }{\xi_T(y_j)} +
q\frac
{\llvert  y_{k-1}-y\rrvert  }{\xi_T(y_k)} + \sum_{j=k+1}^n q
\frac{\llvert  y_{j-1}-y_j\rrvert  }{\xi
_T(y_j)} + q\frac{\llvert  y-y_k\rrvert  }{\xi_T(y_k)}
\\
& \geq&\sum_{j=1}^{k-1} q\frac{\llvert  y_{j-1}-y_j\rrvert  }{\xi_T(y_j)}
+ q\frac
{\llvert  y_{k-1}-y\rrvert  }{\xi_T(y_k)} + h_T(y)
\\
& \geq& q\frac{\llvert  w-y\rrvert  }{\Upsilon} + h_T(y),
\end{eqnarray*}
where the last line followed from the triangle inequality plus the
assumption that $\xi_T(z)\leq\Upsilon$ for all $z\in s_T(t+\eta)$.
Thus $\llvert  w-y\rrvert  \leq(h_T(w)-h_T(y))\Upsilon/q$, and since $t\leq h_T(y)$
and $h_T(w)\leq t+\eta$, we have $\llvert  w-y\rrvert  \leq\eta\Upsilon/q < \delta
/2$. But now $d(x,s_T(t))\leq\llvert  x-w\rrvert  +\llvert  w-y\rrvert  +d(y,s_T(t)) \leq2/r(T) +
\delta/2 < \delta$ which proves our claim.
\end{pf}

\begin{pf*}{Proof of Theorem~\ref{teosupport}}
Fix $t_\infty, \delta,\eps>0$, and apply Lemma~\ref
{lelilysupport} to choose $\eta>0$ such that for all large $T$,
\[
\P \biggl(s_T(t+\eta)\subseteq\bigcup
_{y\in s_T(t)} B(y,\delta) ~\forall t\leq t_\infty
\biggr)\geq1-\eps/3.
\]

Fix $t_\infty>0$, and choose $T_\infty$ large enough so that the
above holds and, using Proposition~\ref{prearly},
\[
\P\bigl(\exists z\dvtx  H_T(z)\leq\bigl(h_T(z) - \eta
\bigr)\wedge(t_\infty+1) \bigr) < \eps/3 \qquad\mbox{for all } T\geq T_\infty.
\]
On the event $\{H_T(z)> (h_T(z) - \eta)\wedge(t_\infty+1) ~\forall z\}$, if $H_T(z)\leq t\leq t_\infty$, we must have
$h_T(z)<H_T(z)+\eta$ and thus $S_T(t)\subseteq s_T(t+\eta)$.

Increasing $T_\infty$ again as required, we can ensure that for all
$T\geq T_\infty$, by Proposition~\ref{prlate},
\[
\P\bigl(\exists z\dvtx  H_T(z)\wedge t_\infty>
h_T(z) + \eta\bigr) < \eps/3.
\]
On the event $\{H_T(z)\wedge t_\infty\leq h_T(z) + \eta~\forall z\}$, if $h_T(z)\leq t-\eta$, then $H_T(z)\leq t$, and
thus $s_T(t-\eta)\subset S_T(t)$.

We have therefore established that with probability at least $1-\eps$,
for any $t\leq t_\infty$,
\[
S_T(t)\subseteq s_T(t+\eta) \subseteq\bigcup
_{y\in s_T(t)} B(y,\delta)
\]
and
\[
s_T(t)\subset\bigcup_{y\in s_T(t-\eta)} B(y,
\delta) \subset\bigcup_{y\in S_T(t)} B(y,\delta),
\]
which proves the theorem.
\end{pf*}

Before proving Theorem~\ref{teointermittency}, we do most of the work
in the following lemma. For $z\in L_T$ and $t\geq0$, let
\[
\tilde m(z,t) = \xi_T(z) \bigl(t-h_T(z)\bigr)_+.
\]
Intuitively, $\tilde m(z,t)$ is a rescaled count of how many particles
should be born at $z$ by time $t$.

%
%le6.2 #&#
\begin{lem}\label{lelocal}
For any $t,\eps>0$ there exists $\delta>0$ such that for all large $T$,
\[
\P\lleft(\exists z_1,z_2\in L_T\dvtx
\matrix{
\tilde m_T(z_1,t)\geq
\tilde m_T(z_2,t)\geq\tilde m_T(y,t) ~\forall y\neq z_1,
\vspace*{3pt}\cr
\bigl\llvert \tilde m_T(z_1,t) - \tilde
m_T(z_2,t)\bigr\rrvert <\delta}
 \rright) < \eps.
\]
\end{lem}

\begin{pf}
Fix $t,\eps>0$, and assume that $T$ is large. By Lemma~\ref
{le0805-1} we may choose $R>0$ such that
\[
\P\bigl(\exists z\notin B(0,R)\dvtx  \tilde m_T(z,t)>0\bigr) < \eps/4.
\]
By Lemma~\ref{le0306-1} we may then choose $\Upsilon>0$ such that
\[
\P\bigl(\exists z\in L_T(0,R)\dvtx  \xi_T(z)>\Upsilon
\bigr) < \eps/4.
\]
We can also find $\eta>0$ such that, by Lemmas~\ref{le0306-1} and
\ref{lesmalllily},
\[
\P\bigl(\tilde m_T(z,t)\leq\eta~\forall z\bigr) <
\eps/4.
\]
For $z_1,z_2\in B(0,R)$, we are interested in the event
\begin{eqnarray*}
\calM_T(z_1,z_2) &=& \bigl\{ \bigl\llvert
m_T(z_1,t)-m_T(z_2,t)\bigr
\rrvert <\delta, \xi_T(z_1)>\eta/2t,
\xi_T(z_2)>\eta/2t,
\\
&&\hspace*{89pt}{} h_T(z_1)<t-\eta/2\Upsilon, h_T(z_2)<t-
\eta/2\Upsilon \bigr\}.
\end{eqnarray*}
This is because, provided $\delta<\frac{\eta}{2t}\wedge\frac{\eta
}{2\Upsilon} \wedge\frac{\eta}{2}$,
%
%e11 #&#
\begin{eqnarray}
&&\P \left(\exists z_1,z_2\in L_T\dvtx
\matrix{\tilde m_T(z_1,t)\geq
\tilde m_T(z_2,t)\geq\tilde m_T(y,t) ~\forall y\neq z_1,
\vspace*{3pt}\cr
\bigl\llvert \tilde m_T(z_1,t) - \tilde
m_T(z_2,t)\bigr\rrvert <\delta } \right)
\nonumber
\\
&&\qquad \leq\P\bigl(\exists z\notin B(0,R)\dvtx  \tilde m_T(z,t) > 0\bigr) +
\P\bigl(\exists z\in L_T(0,R)\dvtx  \xi_T(z)>\Upsilon
\bigr)\label{eqcalM}
\\
&&\quad\qquad{} + \P\bigl(\tilde m_T(z,t) \leq\eta~\forall z\bigr) + \sum_{z_1,z_2\in L_T(0,R)}\P\bigl(
\calM_T(z_1,z_2)\bigr).
\nonumber
\end{eqnarray}
We now estimate $\P(\calM_T(z_1,z_2))$. Suppose first that
$h_T(z_1)\leq h_T(z_2)$. Then
\begin{eqnarray*}
&& \calM_T(z_1,z_2)
\\
&&\qquad \subseteq \biggl\{ \xi_T(z_2)\in \biggl[
\frac
{\xi_T(z_1)(t-h_T(z_1))-\delta}{t-h_T(z_2)}\vee\frac{\eta}{2t}, \frac{\xi_T(z_1)(t-h_T(z_1))+\delta}{t-h_T(z_2)} \biggr],
\\
&&\hspace*{179pt} \xi_T(z_1) > \eta/2t, t-
h_T(z_2)>\eta/2\Upsilon \biggr\}.
\end{eqnarray*}
Now, $\xi_T(z_2)$ is independent of $\{h_T(z_1)\leq h_T(z_2)\}$,
$h_T(z_2)$ and $\xi_T(z_1)$. Further, given $\{h_T(z_1)\leq h_T(z_2)\}
$, $\xi_T(z_2)$ is conditionally independent of $h_T(z_1)$. Also $\P
(\xi_T(z_2)\in[\mu,\mu+\nu])$ is decreasing in $\mu$ for $\mu
\geq1$ and increasing in $\nu$. Thus
\begin{eqnarray*}
&&\P\bigl(\calM_T(z_1,z_2)\cap\bigl
\{h_T(z_1)\leq h_T(z_2)\bigr\}
\bigr)
\\
&&\qquad  \leq\P\bigl(\xi_T(z_2)\in[\eta/2t,\eta/2t+4\delta
\Upsilon/\eta]\bigr) \P\bigl(\xi_T(z_1)\geq\eta/2t\bigr)
\\
&&\qquad  = (\eta/2t)^{-2\alpha} a(T)^{-2\alpha}\biggl(1-\biggl(1+
\frac{8
t\Upsilon\delta}{\eta^2}\biggr)^{-\alpha}\biggr)
\\
&&\qquad  \leq\frac{2^{2\alpha+3}t^{2\alpha+1}\alpha\Upsilon\delta
}{a(T)^{2\alpha}\eta^{2\alpha+2}}.
\end{eqnarray*}
By symmetry we also have
\[
\P\bigl(\calM_T(z_1,z_2)\cap\bigl
\{h_T(z_2)\leq h_T(z_1)\bigr\}
\bigr) \leq\frac
{2^{2\alpha+3}t^{2\alpha+1}\alpha\Upsilon\delta}{a(T)^{2\alpha
}\eta^{2\alpha+2}},
\]
and therefore
\[
\P\bigl(\calM_T(z_1,z_2)\bigr)\leq
\frac{2^{2\alpha+4}t^{2\alpha+1}\alpha
\Upsilon\delta}{a(T)^{2\alpha}\eta^{2\alpha+2}}.
\]
Plugging this and our previous estimates into (\ref{eqcalM}), we see that
\begin{eqnarray*}
&&\P \left(\exists z_1,z_2\in L_T\dvtx
\matrix{\tilde m_T(z_1,t)\geq
\tilde m_T(z_2,t)\geq\tilde m_T(y,t) ~\forall y\neq z_1,
\vspace*{3pt}\cr
\bigl\llvert \tilde m_T(z_1,t) - \tilde
m_T(z_2,t)\bigr\rrvert <\delta}  \right)
\\
&&\qquad < 3\eps/4 + C_d^2 R^{2d}
r(T)^{2d} \frac{2^{2\alpha+4}t^{2\alpha
+1}\alpha\Upsilon\delta}{a(T)^{2\alpha}\eta^{2\alpha+2}},
\end{eqnarray*}
which, by choosing $\delta$ sufficiently small, we may ensure is at
most $\eps$.
\end{pf}

\begin{pf*}{Proof of Theorem~\ref{teointermittency}}
By Proposition~\ref{prgrowth}, for large $T$ we have
\[
\P\bigl(M_T(z,t)\geq m_T(z,t) - \log^{-1/4}T
~\forall z\bigr) > 1 - \eps/5.
\]
By Proposition~\ref{prtoomany}, for large $T$ we have
\[
\P\bigl(M_T(z,t)\leq m_T(z,t)+\log^{-1/4}T
~\forall z\bigr) > 1- \eps/5.
\]
By Lemma~\ref{le0805-1}, we may choose $R>0$ such that for $T>e$ we have
\[
\P\bigl(h_T(z) > t+1 ~\forall z\notin B(0,R)\bigr) > 1-
\eps/5,
\]
and then by Proposition~\ref{prearly}, since if $H_T(z)>t$, then
$M_T(z,t)=0$, for large $T$ we have
\[
\P\bigl(h_T(z) > t+1\mbox{ and } M_T(z,t) = 0
~\forall z\notin B(0,R)\bigr) > 1-2\eps/5.
\]
Finally, by Lemma~\ref{lelocal}, there exists $\delta>0$ such that
for large $T$ we have
\[
\P\bigl(\exists z_1\in L_T\dvtx  \tilde
m_T(z_1,t) \geq\tilde m_T(z,t) + \delta
~\forall z\in L_T\bigr) > 1-\eps/5.
\]
Therefore, with probability at least $1-\eps$, all of the above events
hold. Assume that they do all hold, and fix $z\in L_T(0,R)$. Note that
\[
m_T(z,t) = \sup_y\bigl\{\tilde
m_T(y,t) - q\llvert z-y\rrvert \bigr\},
\]
so either $m_T(z,t) \leq\tilde m_T(z_1,t)-\delta$ or $m_T(z,t) =
\tilde m_T(z_1,t)-q\llvert  z-z_1\rrvert  $. Thus if $\llvert  z-z_1\rrvert   > \frac{3}q \log^{-1/4}T$
and $T$ is large, then
\begin{eqnarray*}
M_T(z,t) &\leq& m_T(z,t) + \log^{-1/4}T \leq
m_T(z_1,t) - 2\log^{-1/4}T
\\
&\leq& M_T(z_1,t)-\log^{-1/4}T.
\end{eqnarray*}
We deduce that if $\llvert  z-z_1\rrvert   > \frac{3}q \log^{-1/4}T$ and $T$ is large, then
\[
N\bigl(r(T)z,tT\bigr) \leq e^{a(T)T M_T(z_1,t) - a(T)T \log^{-1/4}T} = N\bigl(r(T)z_1,tT\bigr)e^{-a(T)T\log^{-1/4}T}.
\]
Summing up, with probability at least $1-\eps$ we have a point $z_1\in
L_T$ such that
\begin{eqnarray*}
\sum_{z\in L_T} N\bigl(r(T)z,tT\bigr)
&\leq&\sum_{z\in L_T(z_1,\sklfrac{3}q \log^{-1/4} T)} N\bigl(r(T)z,tT\bigr)
\\
&&{} +
C_d R^d r(T)^d e^{-a(T)T\log^{-1/4}T}
N\bigl(r(T)z_1,tT\bigr).
\end{eqnarray*}
The result follows.
\end{pf*}

%s7 #&#
\section{The parabolic Anderson model}\label{sePAM}
Our aim in this section is to prove Theorem~\ref{teoPAM}. Recall that
we defined
\[
\Lambda_T(z,t) = \frac{1}{a(T)T}\log_+E^{\xi}\bigl[N
\bigl(r(T)z, tT\bigr)\bigr]
\]
and
\[
\lambda_T(z,t) = \sup_y \bigl\{
\xi_T(y)t - q\llvert y\rrvert - q\llvert z-y\rrvert \bigr\} \vee0,
\]
and that we claimed that these two objects are similar in size for all
$z$ when $T$ is large. The idea is that (in the rescaled picture) the
size of the population at $z$ is dominated by particles that look for
the site $y$ that maximizes $\xi_T(y)t - q\llvert  y\rrvert   - q\llvert  z-y\rrvert  $, run quickly
to $y$ at cost $q\llvert  y\rrvert  $, sit there breeding until just before time $t$
and thus gain a reward of $\xi_T(y)t$, and then run quickly to $z$ at
cost $q\llvert  z-y\rrvert  $.

We first rule out unfriendly environments. Define the event
\[
\Delta_T(t) = \biggl\{\exists k\geq0, \exists z\in B
\bigl(0,2^{k+1}\log\log T\bigr)\dvtx  \xi_T(z) \geq
\frac{q}{2t}2^k\log\log T \biggr\}.
\]
By Lemma~\ref{le0306-1} we know that
\begin{eqnarray*}
&&\P\biggl(\exists z\in B\bigl(0,2^{k+1}\log\log T\bigr)\dvtx
\xi_T(z) \geq\frac
{q}{2t}2^k\log\log T\biggr)
\\
&&\qquad \leq C_d e 2^{(k+1)d}(\log\log T)^d \biggl(
\frac{q}{2t} \biggr)^{-\alpha} 2^{-\alpha k}(\log\log
T)^{-\alpha}
\\
&&\qquad\leq C t^{\alpha} 2^{k(d-\alpha)}(\log\log T)^{d-\alpha}
\end{eqnarray*}
for some constant $C$, and thus for any fixed $t$, $\P(\Delta
_T(t))\to0$ as $T\to\infty$. We view $\Delta_T(t)$ as a bad event
and work on the complement, $\Delta_T(t)^c$.

%s7.1 #&#
\subsection{Upper bound}

%
%le7.1 #&#
\begin{lem}\label{lePAMupper}
For any $t_\infty>0$, there exists $T_0>0$ such that for any $T\geq
T_0$, for all $z\in L_T$ and all $t\leq t_\infty$, on $\Delta
_T(t_\infty)^c$,
\begin{eqnarray*}
&& E^{\xi}\bigl[ N\bigl(r(T)z,tT\bigr)\bigr]
\\
&&\qquad \leq\frac{1}2
+ \sum_{y\in L_T(0,\log\log T)}\! \exp
\biggl(a(T)T\biggl(\xi _T(y)t - q\llvert y\rrvert - q\llvert z-y
\rrvert
\\
&&\hspace*{220pt}{} +  \frac{(\log\log T)^3}{\log
T}\biggr) \biggr).
\end{eqnarray*}
\end{lem}

\begin{pf}
Take $z\in L_T$ and $t\leq t_\infty$. We apply the Feynman--Kac
formula and split the probability space according to the supremum of $\llvert  X(s)\rrvert  $.
\begin{eqnarray*}
E^{\xi}\bigl[N\bigl(r(T)z, tT\bigr)\bigr] &=& E^{\xi}
\bigl[e^{\int_0^{tT}\xi(X(s))\,ds} \1_{\{
X(tT)=r(T)z\}}\bigr]
\\
& =& E^{\xi}\bigl[e^{\int_0^{tT}\xi(X(s))\,ds} \1_{\{X(tT)=r(T)z, \sup_{s\leq tT} \llvert  X(s)\rrvert   \geq r(T)\log\log T\}}\bigr]
\\
&&{} + E^{\xi}\bigl[e^{\int_0^{tT}\xi(X(s))\,ds} \1_{\{
X(tT)=r(T)z, \sup_{s\leq tT} \llvert  X(s)\rrvert   < r(T)\log\log T\}}
\bigr].
\end{eqnarray*}

We check first that the term in which $\sup_{s\leq tT} \llvert  X(s)\rrvert   \geq
r(T)\log\log T$ is small.
\begin{eqnarray*}
&& E^{\xi}\bigl[e^{\int_0^{tT}\xi(X(s))\,ds} \1_{\{X(tT)=r(T)z, \sup_{s\leq tT} \llvert  X(s)\rrvert   \geq r(T)\log\log T\}}\bigr]
\\
&&\qquad  = \sum_{k=0}^\infty E^{\xi}
\bigl[e^{\int_0^{tT}\xi(X(s))\,ds} \1_{\{
X(tT)=r(T)z, \sup_{s\leq tT} \llvert  X(s)\rrvert  /(r(T)\log\log T) \in
[2^k,2^{k+1})\}}\bigr]
\\
&&\qquad  \leq\sum_{k=0}^\infty\exp \Bigl(tT
\max_{x\in
B(0,2^{k+1}r(T)\log\log T)}\xi(x)
\Bigr)
\\
&&\hspace*{46pt}{}\times P^{\xi} \Bigl(\sup_{s\leq tT} \bigl\llvert X(s)\bigr
\rrvert \geq2^k r(T)\log\log T \Bigr)
\\
&&\qquad \leq\sum_{k=0}^\infty\exp \Bigl(tT \max
_{x\in B(0,2^{k+1}r(T)\log
\log T)}\xi(x) \Bigr) J_T\bigl(t,2^k
\log\log T\bigr).
\end{eqnarray*}
Now, we know from Lemma~\ref{rwlem2} that
\[
J_T\bigl(t,2^k\log\log T\bigr) \leq\exp \bigl(-a(T)T
\bigl(2^k q\log\log T - \mathcal{E}^2_T
\bigl(t,2^k\log\log T\bigr)\bigr) \bigr)
\]
and thus on $\Delta_T(t_\infty)^c$,
\begin{eqnarray*}
&& E^{\xi}\bigl[e^{\int_0^{tT}\xi(X(s))\,ds} \1_{\{X(tT)=r(T)z, \sup_{s\leq tT} \llvert  X(s)\rrvert   \geq r(T)\log\log T\}}\bigr]
\\
&&\qquad \leq\sum_{k=0}^\infty\exp \biggl(a(T)T
\biggl(2^k \frac{q}{2}\log \log T - 2^k q\log\log
T + \mathcal{E}^2_T\bigl(t,2^k\log\log T
\bigr) \biggr) \biggr).
\end{eqnarray*}
However,
\[
\calE^2_T\bigl(t,2^k\log\log T\bigr) \leq
\frac{2^k\log\log T}{\log T}\bigl(\log t + 1 + \log(2d) + (q+1)\log\log T\bigr),
\]
so
\[
E^{\xi}\bigl[e^{\int_0^{tT}\xi(X(s))\,ds} \1_{\{X(tT)=r(T)z,
\sup_{s\leq tT} \llvert  X(s)\rrvert   \geq r(T)\log\log T\}}\bigr]\leq
\tfrac{1}2
\]
for large $T$. In particular this shows that on $\Delta_T(t_\infty
)^c$, if $\llvert  z\rrvert  \geq\log\log T$, then $\E[N(r(T)z, tT)]\leq\frac{1}2$.
We may therefore assume that $\llvert  z\rrvert  <\log\log T$.

We are now left with the term when $\sup_{s\leq tT} \llvert  X(s)\rrvert   < r(T)\log
\log T$. We further split our probability space depending on the site
of maximal potential that we visit before time $t$.
\begin{eqnarray*}
&&E^{\xi}\bigl[e^{\int_0^{tT}\xi(X(s))\,ds} \1_{\{X(tT)=r(T)z,
\sup_{s\leq tT} \llvert  X(s)\rrvert   < r(T)\log\log T\}}\bigr]
\\
&&\qquad  \leq\sum_{y\in L_T(0,\log\log T)} E^{\xi}
\bigl[e^{\int_0^{tT}\xi
(X(s))\,ds} \1_{ \biggl\{\fontsize{8.36}{12}\selectfont{\matrix{
X(tT)=r(T)z, \exists s\leq t\dvtx  X(sT)=r(T)y, \cr \sup_{s\leq t} \xi(X(sT)) = \xi
(r(T)y)}}  \biggr\}} \bigr]
\\
&&\qquad  \leq\sum_{y\in L_T(0,\log\log T)} \exp \bigl(a(T)T
\xi_T(y)t \bigr)
\\
&&\hspace*{95pt}{}\times P^{\xi}\bigl(\exists s\leq t\dvtx  X(sT)=r(T)y,
X(tT)=r(T)z\bigr)
\\
&&\qquad  \leq\sum_{y\in L_T(0,\log\log T)} \exp \bigl(a(T)T
\xi_T(y)t \bigr)J_T\bigl(t,\llvert y\rrvert
\bigr)J_T\bigl(t,\llvert z-y\rrvert \bigr)
\\
&&\qquad  \leq\sum_{y\in L_T(0,\log\log T)} \exp \bigl(a(T)T\bigl(
\xi_T(y)t - q\llvert y\rrvert - q\llvert z-y\rrvert \bigr)
\\
&&\hspace*{95pt}{}\times \exp \bigl(\mathcal{E}^2_T
\bigl(t,\llvert y\rrvert \bigr) + \mathcal {E}^2_T
\bigl(t,\llvert z-y\rrvert \bigr)\bigr) \bigr),
\end{eqnarray*}
where the last inequality uses Lemma~\ref{rwlem2}. For large $T$ and
$y\in L_T(0,\log\log T)$, we have $\mathcal{E}^2_T(t,\llvert  y\rrvert  ) \leq(\log
\log T)^3/(2\log T)$ and similarly for $\mathcal{E}^2_T(t,\llvert  z-y\rrvert  )$
since we are assuming that $\llvert  z\rrvert  <\log\log T$. This gives the result.
\end{pf}

%s7.2 #&#
\subsection{Lower bound}

%
%le7.2 #&#
\begin{lem}\label{lePAMlower}
For any $t_\infty>0$, there exists $T_0>0$ such that for any $T\geq
T_0$, for all $y,z\in L_T$ and all $t\leq t_\infty$, on $\Delta
_T(t_\infty)^c$,
\begin{eqnarray*}
&& E^{\xi}\bigl[N\bigl(r(T)z, tT\bigr)\bigr]\vee1
\\
&&\qquad \geq\exp(a(T)T\biggl(
\xi_T(y)t - q\llvert y\rrvert - q\llvert z-y\rrvert - 6
\frac{(\log\log T)^2}{\log T}\biggr).
\end{eqnarray*}
\end{lem}

\begin{pf}
Fix $t\leq t_\infty$. On $\Delta_T(t_\infty)^c$, if $\llvert  y\rrvert  \geq\log
\log T$ or if $\llvert  y\rrvert   \leq\log\log T$ and either $\llvert  z\rrvert  \geq\log\log T$
or $t\leq2/\log T$, then for large $T$,
\[
\xi_T(y)t - q\llvert y\rrvert - q\llvert z-y\rrvert - 6(\log\log
T)^2/\log T \leq0,
\]
so there is nothing to prove. We may therefore assume that $\llvert  y\rrvert  <\log
\log T$, $\llvert  z\rrvert  <\log\log T$ and $t>2/\log T$. Then by the Feynman--Kac formula,
\begin{eqnarray*}
&&E^{\xi}\bigl[N\bigl(r(T)z, tT\bigr)\bigr]
\\
&&\qquad  = E^{\xi}\bigl[e^{\int_0^{tT} \xi(X(s)) \,ds} \1_{\{X(tT)=r(T)z\}}\bigr]
\\
&&\qquad  \geq E^{\xi}\bigl[e^{\int_0^{tT} \xi(X(s)) \,ds} \1_{\{X(sT)=r(T)y
~\forall s\in[1/\log T, t-1/\log T], X(tT)=r(T)z\}}\bigr]
\\
&&\qquad \geq e^{a(T)T\xi_T(y)(t-2/\log T)}
\\
&&P^{\xi}\bigl(X(sT)=r(T)y ~\forall s\in[1/\log T, t-1/\log T],
X(tT)=r(T)z\bigr).
\end{eqnarray*}
By the Markov property,
\begin{eqnarray*}
&&P^{\xi}\bigl(X(sT) =r(T)y ~\forall s\in[1/\log T, t-1/\log T],
X(tT)=r(T)z\bigr)
\\
&&\qquad = P^{\xi}\bigl(X(T/\log T) = r(T)y\bigr)P^{\xi}\bigl(X(sT)=0 ~\forall s\in [0,t-2/\log
T]\bigr)
\\
&&\quad\qquad{}  \times P^{\xi}\bigl(X(T/\log T) = r(T) (z-y)\bigr).
\end{eqnarray*}
By Lemma~\ref{rwlem} and the fact that the probability our random
walk remains at its current location for time $s$ is the probability
that a Poisson random variable of parameter $2\,ds$ is zero, this is at least
\begin{eqnarray*}
&&\exp\bigl(-a(T)T\bigl(q\llvert y\rrvert + \calE^1_T
\bigl(1/\log T,\llvert y\rrvert \bigr)\bigr)\bigr)
\\
&&\qquad{} \times\exp\bigl(-2d(t-2/\log T)T\bigr)
\\
&&\qquad{} \times\exp\bigl(-a(T)T\bigl(q\llvert z-y\rrvert + \calE^1_T\bigl(1/\log T,\llvert z-y\rrvert \bigr)
\bigr)\bigr).
\end{eqnarray*}
It is easy to check that since $y,z\in B(0,\log\log T)$, we have
\[
\calE^1_T\bigl(1/\log T,\llvert y\rrvert \bigr)\leq2(
\log\log T)^2/\log T
\]
and
\[
\calE^1_T\bigl(1/\log T,\llvert z-y\rrvert \bigr)
\leq3(\log\log T)^2/\log T
\]
for large $T$.
Thus if $T$ is large,
\begin{eqnarray*}
&& E^{\xi}\bigl[N\bigl(r(T)z, tT\bigr)\bigr]
\\
&&\qquad \geq\exp(a(T)T\bigl(
\xi_T(y)t - q\llvert y\rrvert - q\llvert z-y\rrvert - 6(\log\log
T)^2/\log T\bigr).
\end{eqnarray*}
\upqed
\end{pf}

%s7.3 #&#
\subsection{Proof of Theorem \texorpdfstring{\protect\ref{teoPAM}}{1.4}}
The two estimates given by the previous two lemmas are the tools we
need to complete the proof of Theorem~\ref{teoPAM}.

\begin{pf*}{Proof of Theorem~\ref{teoPAM}}
We begin with part (i). Fix $t_\infty>0$. For an upper bound, we know
from Lemma~\ref{lePAMupper} that for all large $T$, for any $z\in
L_T$ and $t\leq t_\infty$, on $\Delta_T(t_\infty)^c$,
\begin{eqnarray*}
&& E^{\xi}\bigl[ N\bigl(r(T)z,tT\bigr)\bigr]
\\
&&\qquad \leq\frac{1}2 + \sum_{y\in L_T(0,\log\log T)} \exp
\biggl(a(T)T \biggl(\xi_T(y)t - q\llvert y\rrvert - q\llvert z-y
\rrvert
\\
&&\hspace*{221pt}{}+  \frac{(\log\log T)^3}{\log
T} \biggr) \biggr).
\end{eqnarray*}
Since $\lambda_T(z,t) \geq\sup_{y\in L_T} \{\xi_T(y)t - q\llvert  y\rrvert   -
q\llvert  z-y\rrvert  \}$, we immediately see that
\begin{eqnarray*}
&& E^{\xi}\bigl[N\bigl(r(T)z,tT\bigr)\bigr]
\\
&&\qquad  \leq\frac{1}2 + C_d (\log\log T)^d
r(T)^d \exp \biggl(a(T)T\biggl(\lambda _T(z,t) +
 \frac{(\log\log T)^3}{\log T}\biggr) \biggr),
\end{eqnarray*}
and thus
\[
\Lambda_T(z,t):= \frac{1}{a(T)T} \log_+ E^{\xi}
\bigl[N\bigl(r(T)z,tT\bigr)\bigr] \leq \lambda_T(z,t) +
\frac{(\log\log T)^3}{\log T} + \frac{1}{T}.
\]
For a lower bound, we know from Lemma~\ref{lePAMlower} that on
$\Delta_T(t_\infty)^c$, for any $y,z\in L_T$ and $t\leq t_\infty$,
\begin{eqnarray*}
&& E^{\xi}\bigl[N\bigl(r(T)z,  tT\bigr)\bigr]\vee1
\\
&&\qquad \geq\exp \bigl(a(T)T\bigl(\xi_T(y)t - q\llvert y\rrvert - q
\llvert z-y\rrvert - 6(\log\log T)^2/\log T\bigr)\bigr).
\end{eqnarray*}
Without loss of generality, we can assume $\la_T(z,t) > 0$. Then
choosing $y$ such that $\lambda_T(z,t) = \xi_T(y)t - q\llvert  y\rrvert   - q\llvert  z-y\rrvert  $,
we have
\[
E^{\xi}\bigl[N\bigl(r(T)z, tT\bigr)\bigr]\vee1 \geq\exp\bigl(a(T)T
\bigl(\lambda_T(z,t) - 6(\log \log T)^2/\log T\bigr)
\bigr)
\]
and thus
\[
\Lambda_T(z,t) \geq\lambda_T(z,t) - 6
\frac{(\log\log T)^2}{\log T}.
\]
%
% - \frac{1}{a(T)T}.\]
Since
$\P(\Delta_T(t_\infty))\to0$ as $T\to\infty$, we have the desired result.

We now move on to part (ii). Fix $R,\eps,\delta>0$. It is easy to see
by the triangle inequality
and Lemma~\ref{leincrease} that $\tau_T(z)\leq h_T(z)$ for all $z\in
L_T$ and $T>e$, so by Lemma~\ref{lesmalllily}, there exists
$t_\infty$ such that
\[
\P\bigl(\tau_T(z)< t_\infty-\delta ~\forall z
\in L_T(0,R)\bigr) > 1- \eps/4.
\]
Moreover, by Lemma~\ref{lesmalllily}, again since $\tau_T(z) \leq h_T(z)$,
\[
\P\bigl( \tau_T(z) \leq\delta~\forall z \in
L_T\bigl(0, (\log \log T)^4/\log T\bigr) \bigr) > 1 -
\eps/4.
\]
By Lemma~\ref{le0306-1}, we can find a large $K$ such that
%
%e12 #&#
\begin{equation}
\label{eqy0K} \P\bigl(\exists y_0 \in L_T
\bigl(0,2^{-K}\bigr)\dvtx  \xi_T(y_0)
\geq4q2^{-K}/\delta\bigr) > 1 - \eps/4.
\end{equation}
Also we can choose $T$ large enough such that $\P( \Delta_T(t_\infty
)^c) > 1 - \eps/4$.
Then, with probability at least $1-\eps$, we may assume that all of
the above events hold, and
it suffices to show that then for any $z$ with $\tau_T(z)< t_\infty
-\delta$, we have
%
%e13 #&#
%e14 #&#
\begin{eqnarray}
\Lambda_T(z,t) &=& 0 \qquad\forall t \leq\tau_T(z)-
\delta\quad\mbox{and} \label{eq1605-1}
\\
\Lambda_T\bigl(z,\tau_T(z)+\delta\bigr) &>&
0.\label{eq1605-2}
\end{eqnarray}

To show~(\ref{eq1605-1}), note that if $\llvert  z\rrvert  \leq(\log\log T)^4/\log
T$, then
$\tau_T(z) \leq\delta$ by assumption, and the statement is trivial.
Therefore we may assume that $\llvert  z\rrvert  >(\log\log T)^4/\log T$. By
Lemma~\ref{lePAMupper}, on $\Delta_T(t_\infty)^c$ we have for any
$t \leq t_\infty$,
\begin{eqnarray*}
&& E^{\xi}\bigl[N\bigl( r(T)z,tT\bigr)\bigr]
\\
&&\qquad \leq\frac{1}2 + \sum_{y\in L_T(0,\log\log T)} \exp
\biggl(a(T)T\biggl(\xi _T(y)t - q\llvert y\rrvert - q\llvert z-y
\rrvert
\\
&&\hspace*{221pt}{}+  \frac{(\log\log T)^3}{\log
T}\biggr) \biggr).
\end{eqnarray*}
We claim that for every $y$,
\[
\xi_T(y) \bigl(\tau_T(z)-\delta\bigr) - q\llvert y
\rrvert - q\llvert z-y\rrvert \leq-3(\log\log T)^3/\log T,
\]
which\vspace*{1pt} is enough to guarantee that $\Lambda_T(z,t) = 0$ for all $t \leq
\tau_T(z) - \delta$. Indeed, if $\xi_T(y)\geq3(\log\log
T)^3/(\delta\log T)$, then by the definition of $\tau_T(z)$ we have
\[
\xi_T(y) \bigl(\tau_T(z)-\delta\bigr) - q\llvert y
\rrvert - q\llvert z-y\rrvert \leq- \xi_T(y)\delta \leq-3(\log\log
T)^3/ \log T.
\]
On the other hand if $\xi_T(y)< 3(\log\log T)^3/(\delta\log T)$,
since $\tau_T(z) \leq t_\infty$, we have
by the triangle inequality
\begin{eqnarray*}
&&\xi_T(y) \bigl(\tau_T(z)-\delta\bigr) - q\llvert y
\rrvert - q\llvert z-y\rrvert \leq \xi_T(y)\tau_T(z) -
q\llvert z\rrvert
\\
&&\qquad  \leq 3t_\infty(\log\log T)^3/ (\delta\log T) - q(\log\log
T)^4/\log T,
\end{eqnarray*}
which is smaller than $-3(\log\log T)^3/ \log T$ when $T$ is large.

For~(\ref{eq1605-2}), by Lemma~\ref{lePAMlower} we have that on
$\Delta_T(t_\infty)^c$, for any $y\in L_T$,
\begin{eqnarray*}
&& E^{\xi}\bigl[N\bigl(r(T)z, \bigl(\tau_T(z)+\delta\bigr)T
\bigr)\bigr]\vee1
\\
&&\qquad \geq\exp\bigl(a(T)T\bigl(\xi_T(y) \bigl(\tau_T(z)+
\delta\bigr) - q\llvert y\rrvert - q\llvert z-y\rrvert - 6(\log\log
T)^2/\log T\bigr)\bigr).
\end{eqnarray*}
If $\llvert  z\rrvert  <2^{-K}$, then choose $y = y_0$ from (\ref{eqy0K}), and note that
\[
\xi_T(y_0) \bigl(\tau_T(z)+\delta
\bigr) - q\llvert y_0\rrvert - q\llvert z-y_0\rrvert
\geq4q 2^{-K} - q 2^{-K} - 2q 2^{-K} =
q2^{-K},
\]
so $\Lambda_T(z,\tau_T(z)+\delta)\geq q2^{-K} - 6\frac{(\log\log
T)^2}{\log T}$. On the other hand if $\llvert  z\rrvert  \geq2^{-K}$, choose $y$ such
that $\tau_T(z) = \frac{q}{\xi_T(y)}(\llvert  y\rrvert  +\llvert  z-y\rrvert  )$. By assumption\vspace*{1pt} we
know that $\tau_T(z) \leq t_\infty$,
and since by the triangle inequality $\tau_T(z) \geq q\llvert  z\rrvert  /\xi_T(y)$,
we deduce that $\xi_T(y) \geq q2^{-K}/t_\infty$. Then
\[
\xi_T(y) \bigl(\tau_T(z)+\delta\bigr) - q\llvert y
\rrvert - q\llvert z-y\rrvert = \xi_T(y)\delta\geq q
\delta2^{-K}/t_\infty,
\]
and thus $\Lambda_T(z,(\tau_T(z)+\delta)T)\geq\frac{q\delta
2^{-K}}{\log\log T} - 6\frac{(\log\log T)^2}{\log T} >0$ provided
$T$ is large.

Finally, the proof of part (iii) of the theorem is essentially the same
as the proof of Theorem~\ref{teosupport}, checking that with high
probability the PAM lilypad model does not grow too fast, as in Lemma
\ref{lelilysupport}, and combining this with our knowledge of the
hitting times from part (ii) above.
\end{pf*}

%s8 #&#
\section{Comparing the BRW with the PAM}\label{secompare}

In this section we prove Theorem~\ref{teonotconnected}. We start by
showing the corresponding statements for the maximizers of the lilypad
models. As a first step, we construct conditions on the potential under
which we can control the maximum of the BRW lilypad. We will see in
Lemma~\ref{leposprob} that these conditions occur simultaneously
with positive probability uniformly in~$T$.

\subsection*{BRW Setup} Fix $T, t, \kappa>0$. Suppose that $r>0$ and $\eta
\geq8qr/t$ (we will choose $r$ and $\eta$ later) and that $R>(\frac
{2\eta t}{q})\vee3 \kappa$.
Assume that the potential $(\xi_T(z), z \in L_T)$ satisfies the
following conditions:%
\begin{longlist}[(A)]
\item[(A)] there exists a site $x$ in $L_T(0,r)$ such that $\xi_T(x)
\in[\eta, 2\eta)$;
\item[(B)] for all sites $y \in L_T(0,R) \setminus\{ x\}$, we have
that $\xi_T(y) \leq\eta/2$;
\item[(C)] $\max_{z\in L_T(0,r)}h_T(z) \leq t/8$.
\end{longlist}

%
%pr8.1 #&#
\begin{prop}\label{propBRWpeaks}
Under assumptions \textup{(A)}, \textup{(B)}, \textup{(C)} and for $T$ sufficiently large, if $z
\in L_T$ is such that $h_T(z) \leq t$, then
\[
m_T(z,t) = m_T(x,t) - q \llvert z-x\rrvert =
\xi_T(x) \bigl( t - h_T(x)\bigr) - q \llvert z-x
\rrvert,
\]
and otherwise $m_T(z,t) = 0$. In particular, $x$ is the unique
maximizer of $m_T(\cdot,t)$.
\end{prop}

\begin{pf}
The idea of the proof is first to check that all sites outside the ball
$L_T(0,R)$ are hit after time $t$. Then we make sure that the site $x$
with large potential is hit so early that by time $t$ the lilypad has
grown far enough to ``overtake'' all other lilypads.

We first show that any site outside $L_T(0,R)$ is hit after time $t$.
Indeed, by Lemma~\ref{le0805-1} we have that
\[
\bigl\{\exists z \in L_T\setminus B(0,R)\dvtx  h_T(z)
\leq t \bigr\} \subseteq \Bigl\{ \max_{y\in L_T(0,R)}
\xi_T(y)\geq qR/t \Bigr\}.
\]
However, $qR/t > 2\eta$, so we have $h_T(z) > t$ for all $z\in
L_T\setminus B(0,R)$.

For our next task, first suppose that $\llvert  z\rrvert  <\frac{\eta t}{8q} + r$.
Then since $x\in L_T(0,r)$,
\begin{eqnarray*}
\xi_T(x) \bigl(t-h_T(x)\bigr) - q\llvert z-x\rrvert &
>& \xi_T(x)t - \xi _T(x)t/8 - \eta t/8 - 2qr
\\
& \geq&7\xi_T(x)t/8 - 3\eta t/8 \geq\xi_T(x)t/2.
\end{eqnarray*}
Since $\xi_T(x)\geq2\xi_T(y)$ for any $y\in L_T(0,R)\setminus\{x\}
$, we therefore have
\begin{eqnarray*}
\xi_T(x) \bigl(t-h_T(x)\bigr) - q\llvert z-x\rrvert &
>& \sup_{y\neq x} \bigl\{\xi_T(y)
\bigl(t-h_T(y)\bigr) \bigr\}
\\
& \geq&\sup_{y\neq x} \bigl\{\xi_T(y)
\bigl(t-h_T(y)\bigr) - q \llvert z-y\rrvert \bigr\}
\end{eqnarray*}
and hence
\[
m_T(z,t) = \xi_T(x) \bigl(t-h_T(x)
\bigr) - q\llvert z-x\rrvert = m_T(x,t) - q\llvert z-x\rrvert.
\]

Now suppose that $\llvert  z\rrvert  \geq\frac{\eta t}{8q} + r$. We claim that
$h_T(z) = h_T(x) + q\llvert  z-x\rrvert  /\xi_T(x)$. Indeed,
\[
h_T(x) + q\frac{\llvert  z-x\rrvert  }{\xi_T(x)} \leq\frac{t}{8} + q
\frac{r}{\eta
} + q\frac{\llvert  z\rrvert  }{\eta},
\]
and for any $y_0,\ldots, y_n\in L_T(0,R)\setminus\{x\}$ with $y_0=z$
and $y_n=0$, by the triangle inequality we have
\[
\sum_{j=1}^n q\frac{\llvert  y_j-y_{j-1}\rrvert  }{\xi_T(y_{j-1})} \geq q
\frac
{2}{\eta} \llvert z\rrvert \geq q\frac{\llvert  z\rrvert  }{\eta} +
\frac{t}{8} + q\frac
{r}{\eta},
\]
which proves the claim. As a result we have that for any $y \in L_T
\setminus\{z\}$,
\[
h_T(x)+ q \frac{\llvert  x-z\rrvert  }{\xi_T(x)} = h_T(z) \leq
h_T(y) + q \frac
{\llvert  y-z\rrvert  }{\xi_T(y)}.
\]
Thus, for any $y \in L_T(0,R)\setminus\{x\}$ such that $h_T(y) \leq
t$, we have that since $\xi_T(y) \leq\frac{1}2 \xi_T(x)$,
\begin{eqnarray*}
\xi_T(y) \bigl( t- h_T(y)\bigr) - q\llvert y-z
\rrvert & \leq&\xi_T(y) \biggl( t - h_T(x) - q
\frac{\llvert  x-z\rrvert  }{\xi_T(x)} \biggr)
\\
& \leq&\frac{1}2 \bigl( \xi_T(x) \bigl( t-
h_T(x) \bigr) - q \llvert x-z\rrvert \bigr).
\end{eqnarray*}
We deduce that in this case too we have
\[
m_T(z,t) = \xi_T(x) \bigl( t- h_T(x)
\bigr) - q\llvert x-z\rrvert = m_T(x,t) - q\llvert z-x\rrvert.
\]
\upqed
\end{pf}

On top of conditions \textup{(A)}, \textup{(B)} and \textup{(C)} outlined above, we now construct
two further scenarios in which the maximizers of the BRW and PAM
lilypad models agree, respectively, do not agree, at time $t$.
\begin{longlist}[(S2)]
\item[(S1)] Suppose that for all $y \notin L_T(0,R)$, we have that
\[
\xi_T(y) < \eta+ q\bigl(\llvert y\rrvert - r\bigr)/t.
\]
\item[(S2)] Suppose that there exists a point $x'\in L_T(0,R+1)
\setminus L_T(0,R)$ such that
\[
\xi_T\bigl(x'\bigr) > 2\eta+ q( R+1 )/t,
\]
and that for all $y\neq x$ such that $y\notin L_T(0,R)$, we have that
\[
\xi_T(y) < \eta+ q\bigl(\llvert y\rrvert - r\bigr)/t.
\]
\end{longlist}

Recall that
\[
\la_T(z,t) = \sup_{y \in L_T} \bigl\{
\xi_T(y) t - q \llvert y\rrvert - q\llvert y-z\rrvert \bigr\} \vee0.
\]

%
%le8.2 #&#
\begin{lem}\label{lelilypaddiff} Suppose \textup{(A)}, \textup{(B)}, \textup{(C)} hold.
\begin{longlist}[(ii)]
\item[(i)] If in addition \textup{(S1)} holds, then $x$ [as defined in \textup{(A)}] is
the unique maximizer of $\la_T(\cdot, t)$ and of $m_T(\cdot, t)$.
\item[(ii)] If in addition \textup{(S2)} holds, then $x$ [as defined in \textup{(A)}] is
the unique maximizer of $m_T(\cdot, t)$, whereas $x'$ is the unique
maximizer of $\la_T(\cdot,t)$ and $\llvert  x-x'\rrvert   > 2 \kappa$.
\end{longlist}
\end{lem}

\begin{pf}
Under \textup{(A)}, \textup{(B)}, \textup{(C)} we already know from Proposition~\ref
{propBRWpeaks} that $x$ is the unique maximizer of the BRW lilypad
$m_T(\cdot,t)$. Moreover, $\llvert  x\rrvert  <r$ and $\llvert  x'\rrvert  \geq R > (16r)\vee3 \kappa
$ so $\llvert  x-x'\rrvert  > 2 \kappa$. Thus it suffices to prove the statements
about $\lambda_T(\cdot,t)$. Note from the definition of $\lambda
_T(\cdot,t)$ that in particular
\[
\sup_{z \in L_T} \la_T(z,t) = \sup
_{z \in L_T} \bigl\{\xi_T(z)t - q\llvert z\rrvert \bigr
\}.
\]

\begin{longlist}[(ii)]
\item[(i)] By the above it suffices to show that $t\xi_T(x) -q\llvert  x\rrvert   > t\xi
_T(y)-q\llvert  y\rrvert  $ for all $y\in L_T\setminus\{x\}$. Take $y \in L_T
\setminus\{ x\}$. If $y \in L_T(0,R)$, then $\xi_T(y) \leq\eta/2$
and thus
\[
t \xi_T(y) - q \llvert y\rrvert \leq\eta t/2 \leq t
\xi_T(x) - \eta t/2 \leq t \xi _T(x) - 4q r < t
\xi_T(x) - q\llvert x\rrvert.
\]
On the other hand, if $\llvert  y\rrvert   \geq R$, then by assumption \textup{(S1)},
\[
\xi_T(y)t - q \llvert y\rrvert < \eta t + q\bigl(\llvert y\rrvert
-r \bigr) - q \llvert y\rrvert \leq\xi_T(x)t - q\llvert x\rrvert.
\]
Thus $x$ is the unique maximizer of $\la_T(\cdot, t)$.

\item[(ii)] Arguing as in part (i), we already know that $\la_T(y,t) \leq\la
_T(x,t)$ for all $y \neq x,x'$.
Thus we only need to show that $\lambda_T(x',t) > \lambda_T(x,t)$.
Indeed, by the assumption on $\xi_T(x')$, we have
\[
\xi_T\bigl(x'\bigr)t - q \bigl\llvert
x'\bigr\rrvert > 2\eta t + q (R+1) - q\bigl\llvert x'
\bigr\rrvert \geq\xi_T(x)t - q \llvert x\rrvert.
\]
Thus $x'$ is the unique maximizer of $\la_T(\cdot,t)$.\quad\qed
\end{longlist}\noqed
\end{pf}

Next we construct a scenario in which the support of the PAM lilypad is
disconnected.
\begin{longlist}[(S3)]
\item[(S3)] For $R$ as above, suppose there exists $x' \in L_T$ with
$2R \leq\llvert  x'\rrvert   \leq2R+1$ such that $\xi_T(x') \in((2R+1)q/t,
5Rq/2t)$. Moreover, assume that for any $y \notin L_T(0,R) \cup\{ x'\}
$ we have $\xi_T(y) < q \llvert  y\rrvert  /t$.
\end{longlist}

%
%le8.3 #&#
\begin{lem}\label{le0504-2}
If events \textup{(A)}, \textup{(B)} and \textup{(S3)} hold, then the support of the PAM lilypad
model is not connected at time $t$.
\end{lem}

\begin{pf}
Recall that for $z\in L_T$, we defined
\[
\tau_T(z) = \inf_{y\in L_T} \biggl\{
\frac{q}{\xi_T(y)}\bigl(\llvert y\rrvert + \llvert z-y\rrvert \bigr) \biggr\}.
\]
Suppose that $\tau(z) = q(\llvert  x'\rrvert  +\llvert  z-x'\rrvert  )/\xi_T(x')\leq t$. Then
\[
\bigl\llvert x'-z\bigr\rrvert \leq\frac{\xi_T(x')t}{q} - \bigl
\llvert x'\bigr\rrvert < \frac{5R}{2} - \bigl\llvert
x'\bigr\rrvert \leq \frac{R}{2}.
\]
In particular if $\tau(z) = q(\llvert  x'\rrvert  +\llvert  z-x'\rrvert  )/\xi_T(x')\leq t$, then
$z\notin B(0,3R/2)$. Moreover,
\[
\tau_T\bigl(x'\bigr) \leq\frac{q \llvert  x'\rrvert  }{\xi_T(x')} < q
\bigl\llvert x'\bigr\rrvert \frac{t}{(2R+1)q} \leq t.
\]

Now suppose that $\tau_T(z) = q(\llvert  y\rrvert  +\llvert  z-y\rrvert  )/\xi_T(y)\leq t$ for some
site $y\in L_T(0,R)$. By assumptions \textup{(A)} and \textup{(B)} we have $\xi_T(y)\leq
2\eta$, which combined with the triangle inequality yields
\[
\tau_T(z) \geq\frac{q\llvert  z\rrvert  }{2\eta}.
\]
We deduce that for such $z$ we have $\llvert  z\rrvert  \leq2\eta t/q < R$, and thus
$z\in B(0,R)$.

Finally, since under \textup{(S3)} for any $y \notin L_T(0,R) \cup\{x'\}$,
$\frac{q \llvert  y\rrvert  }{\xi_T(y)} > t$, we can conclude that
\begin{eqnarray*}
&& \bigl\{ z \in L_T\dvtx  \tau_T(z) \leq t \bigr\}
\\
&&\qquad =
\biggl\{ z \in L_T\dvtx  \tau_T(z) = \frac{q}{\xi_T(x')}
\bigl(\bigl\llvert x'\bigr\rrvert +\bigl\llvert z-x'
\bigr\rrvert \bigr) \leq t \biggr\}
\\
&&\quad\qquad{}\cup \biggl\{ z \in L_T\dvtx  \tau_T(z) = \inf
_{y\in L_T(0,R)} \biggl\{\frac{q}{\xi_T(y)}\bigl(\llvert y
\rrvert +\llvert z-y\rrvert \bigr) \biggr\} \leq t \biggr\}.
\end{eqnarray*}
However, by the above, the two sets on the right-hand side are
nonempty and are separated by distance at least $R/2$, which
immediately implies that the support of the PAM lilypad model at time
$t$ is disconnected.
\end{pf}

Finally, we show that our conditions on the potential are fulfilled
with positive probability.

%
%le8.4 #&#
\begin{lem}\label{leposprob}
For\vspace*{1pt} any $t>0$, there exist $r>0$, $\eta\geq8qr/t$ and $R>(\frac
{2\eta t}{q})\vee3 \kappa$ such that for any $i=1,2,3$, the
probability that events \textup{(A)}, \textup{(B)}, \textup{(C)} and \textup{(Si)} occur simultaneously is
bounded away from $0$ for all large $T$.
\end{lem}

\begin{pf}
To show that there exist $r,\eta,R$ such that events \textup{(A)}, \textup{(B)} and \textup{(C)}
occur simultaneously with probability bounded away from $0$, we use
similar tactics to the proof of Lemma~\ref{lesmalllily}. As in Lemma
\ref{lesmalllily} we take $\gamma\in(d/\alpha,1)$, let $B_k =
L_T(0,2^{-k})$ for $k\geq1$ and set
\[
A_k = \bigl\{\exists z\in B_k\dvtx  \xi_T(z)
\geq2^{-\gamma k} \bigr\}.
\]
Then by Lemma~\ref{le0306-1},
\[
\P\bigl( A_k^c\bigr) \leq e^{-c_d 2^{(\alpha\gamma-d)k}}.
\]
Thus we may choose $K$ such that
\[
\P \Biggl(\bigcap_{k=K}^\infty
A_k \Biggr) > 1- \frac{c_d}{8} e^{-C_d
2^\alpha} \quad\mbox{and}\quad 2^{-K} \leq \biggl(
\frac{t}{8q} \biggr)^{q+1}.
\]
As in the proof of Lemma~\ref{lesmalllily}, on the event $\bigcap_{k=K}^\infty A_k$ we have $\bar h_T(2^{-K})\leq\frac{4q}{1-2^{\gamma
-1}}2^{(\gamma-1)(K-1)}$, which by increasing $K$ if necessary we may
assume is at most $t/8$.
Now let $r = 2^{-K}$ and $\eta= r^{d/\alpha}$. By the second
condition on $K$ it is easy to check that $\eta\geq8qr/t$ as
required. Define
\[
A'_K = \bigl\{\exists x\in B_K\dvtx
\xi_T(x)\in[\eta,2\eta), \xi_T(y)\leq
\eta/2~\forall y\in B_K\setminus\{x\} \bigr\}.
\]
Then
\begin{eqnarray*}
\P\bigl( A'_K\bigr) & \geq& c_d
2^{-dK}r(T)^d\bigl(\eta^{-\alpha}a(T)^{-\alpha} -
(2\eta )^{-\alpha}a(T)^{-\alpha}\bigr)
\\
&&{}\times \bigl(1-(\eta/2)^{-\alpha}a(T)^{-\alpha}
\bigr)^{C_d
2^{-dK}r(T)^d},
\end{eqnarray*}
which for large $T$ is at least
\[
\frac{c_d}{4} r^d\eta^{-\alpha} \exp\bigl(-C_d
2^\alpha r^d \eta ^{-\alpha}\bigr) =
\frac{c_d}{4} e^{-C_d 2^\alpha},
\]
by our choice of $\eta$.
Thus
\[
\P \Biggl( A'_K \cap\bigcap
_{k=K}^\infty A_k \Biggr) >
\frac{c_d}{8} e^{-C_d 2^\alpha}.
\]
On the event $ A'_K \cap\bigcap_{k=K}^\infty A_k$, conditions \textup{(A)} and
\textup{(C)} are satisfied. Since the potential on $L_T(0,R)\setminus L_T(0,r)$
is independent of that on $L_T(0,r)$, and
\begin{eqnarray*}
\P\biggl(\xi_T(y)\leq\frac{\eta}{2} ~\forall y\in
L_T(0,R)\setminus L_T(0,r)\biggr) &\geq& \biggl(1-
\biggl(\frac{\eta}{2}\biggr)^{-\alpha
}a(T)^{-\alpha}
\biggr)^{C_d R^d r(T)^d}
\\
&\geq& c_{R,\eta}
\end{eqnarray*}
for some constant $c_{R,\eta}$ depending on $R$ and $\eta$,
conditions \textup{(A)}, \textup{(B)} and \textup{(C)} occur simultaneously with probability at
least $c_{R,\eta} c_d e^{-C_d 2^\alpha}/8$.

Since \textup{(S1)}, \textup{(S2)} and \textup{(S3)} only involve sites outside $L_T(0,R)$, they
are independent of the events above, and so it suffices to show that
for some $R>(\frac{2\eta t}{q})\vee3 \kappa$, each occurs with
positive probability. Note that for any $k\geq1$,
\begin{eqnarray*}
&&\P\bigl(\exists y\in L_T\bigl(kR,(k+1)R\bigr)\dvtx
\xi_T(y)\geq q(kR-r)/t\bigr)
\\
&&\qquad \leq C_d\bigl((k+1)^d - k^d
\bigr)R^d r(T)^d\bigl(q(kR-r)/t\bigr)^{-\alpha}a(T)^{-\alpha
}
\\
&&\qquad \leq C_d R^{d-\alpha} t^\alpha q^{-\alpha}
\frac{((k+1)^d -
k^d)}{(k-r/R)^\alpha}
\\
&&\qquad \leq C_d d 2^{d+\alpha} R^{d-\alpha} t^\alpha
q^{-\alpha} k^{d-1-\alpha}.
\end{eqnarray*}
Thus
\[
\P\bigl(\exists y\notin L_T(0,R)\dvtx  \xi_T(y) \geq q
\bigl(\llvert y\rrvert -r\bigr)/t\bigr) \leq C_d d 2^d
R^{d-\alpha}t^\alpha q^{-\alpha} \sum
_{k=1}^\infty k^{d-1-\alpha},
\]
which we can make arbitrarily small by choosing $R$ large. This in
particular establishes that \textup{(S1)} occurs with positive probability. The
fact that \textup{(S2)} and \textup{(S3)} each occurs with positive probability then
follows by essentially repeating the calculation of $\P( A'_K)$ above.
\end{pf}

%
%re2 #&#
\begin{rmk*}
We could have proved Lemma~\ref{leposprob} in a more elegant way by
introducing a scaling limit for the potential as in \cite{HMS08}, Section~2.2. We chose the more hands-on route in order to
avoid introducing a new tool at the very end of the article.
\end{rmk*}

\begin{pf*}{Proof of Theorem~\ref{teonotconnected}}
(i) Combining Lemmas~\ref{le0504-2} and~\ref{leposprob} we
know that with positive probability the support $s^{\mathrm{PAM}}_T(t)$ of
the PAM lilypad model at time $t$ is contained in two disjoint sets
that are separated by distance at least $\frac{R}{2}$. Together with
Theorem~\ref{teoPAM}(iii) this implies that the actual support
$S_T^{\mathrm{PAM}}(t)$ is also disconnected with positive probability.

(ii) By Lemma~\ref{leposprob}, there exists $\eps> 0$ such that the
probability that \textup{(A)}, \textup{(B)}, \textup{(C)} and either \textup{(S1)} or \textup{(S2)} occurs is
bounded below
by $\eps$.

From~\cite{KLMS09}, Theorem~1.3, we know that with probability at least
$1- \eps/4$, the PAM is concentrated in a single site which they call
$Z_{tT}$, in the sense that
\[
\frac{u(Z_{tT}, tT)}{ \sum_{z \in\Z^d} u(z,tT) } \geq\frac{3}4.
\]
This immediately implies that $u(\cdot,tT)$ is maximal in $Z_{tT}$,
that is, in our notation that $W^{\mathrm{PAM}}_T(t) = Z_{tT}/r(T)$.

The site $Z_{tT}$ is the maximizer of a functional $\Phi_{tT}(z)$
[defined in terms of $\llvert  z\rrvert  $, the potential $\xi(z)$ and the number of
paths leading to $z$].
Rather than stating the explicit definition, we recall the following
simplficiation.
By~\cite{MOS11}, Lemma 3.3, we can choose $N$ large enough such that
with probability at least $1 - \eps/4$, $\frac{Z_{tT} }{r(T)}
\in L_T(0, N)$ and
$\frac{\phi_{tT}(Z_{tT})}{a_T} \in[\frac{1}N, N]$.
In particular, by~\cite{MOS11}, Lemma 3.2, we know that there exists a
constant $C = C(N,q,t)$,
such that for $z = \frac{Z_{tT}}{r(T)}$,
%
%e15 #&#
\begin{equation}
\label{eq1504-1} \biggl\llvert \frac{\Phi_{tT}(r_T
z)}{a(T)} - \biggl( \xi_T(z)
- \frac{q}{t} \llvert z\rrvert \biggr)\biggr\rrvert \leq
\frac{ C
\log\log T}{\log T}.
\end{equation}
On the scenarios \textup{(A)}, \textup{(B)}, \textup{(C)} and either \textup{(S1)} or \textup{(S2)}, by the same
argument (\ref{eq1504-1}) also holds for $z = w_T^{\mathrm{PAM}}(t)$.
However, we have already seen in Lemma~\ref{lelilypaddiff} that
there exists $\delta> 0$ such that on either event \textup{(S1)} or \textup{(S2)},
we have that for all $y \in L_T$,
\[
t \xi_T(y) - q\llvert y\rrvert \leq t \xi_T\bigl(
w_T^{\mathrm{PAM}}(t) \bigr) - q \bigl\llvert w_T^{\mathrm{PAM}}(t)
\bigr\rrvert - \delta.
\]
Comparing\vspace*{1pt} with~(\ref{eq1504-1}) and using that $Z_{tT}$ is the
maximizer of $\Phi_{tT}$ shows that necessarily $w_T^{\mathrm{PAM}}(t) =
Z_{tT} / r(T)$.

Moreover, the proof of Theorem~\ref{teointermittency} shows that with
probability at least $1 - \eps/4$, the maximizers of the lilypad model and
the BRW are close, that is, $\llvert  w_T(t) - W_T(t)\rrvert   \leq\frac{3} q \log^{-1/4}(T)$.\vspace*{1pt}

Hence, by combining all of the above, with probability at least $\eps
/4$ we have that $w_T^{\mathrm{PAM}}(t) = W_T^{\mathrm{PAM}}$ and $\llvert  w_T(t) -
W_T(t)\rrvert   \leq\frac{3} q \log^{-1/4}(T)$, while at the same time \textup{(A)},
\textup{(B)}, \textup{(C)} and either \textup{(S1)} or \textup{(S2)} hold. The proof of statement (ii) is
then completed by Lemma~\ref{lelilypaddiff}, which implies that on
\textup{(S1)}, $\llvert  W_T(t) - W_T^{\mathrm{PAM}}(t)\rrvert   \leq\frac{3}q \log^{-1/4} T$, and
on \textup{(S2)}, $\llvert   W_T(t) - W_T^{\mathrm{PAM}}(t) \rrvert   \geq\kappa$.
\end{pf*}

\section*{Glossary of notation}

\mbox{}

%t1 #&#
\begin{table}[h]
\tabcolsep=0pt
\begin{tabular*}{\tablewidth}{@{\extracolsep{\fill}}@{}ll@{}}
\hline
\textbf{Notation} & \textbf{Definition/description}\\
\hline
$\alpha$ & Constant such that $\P(\xi(0) > x) = x^{-\alpha}$ for $x\geq1$\\[0.9pt]
$d$ & Dimension; we work in $\Z^d$ and $\R^d$\\[0.9pt]
$q$ & $d/(\alpha-d)$\\[0.9pt]
$a(T)$ & $(T/\log T)^q$ (rescaling of potential $\xi$)\\[0.9pt]
$r(T)$ & $(T/\log T)^{q+1}$ (spatial rescaling)\\[0.9pt]
$\xi_T(z)$ & $\xi(r(T)z)/a(T)$ (rescaled potential at $z$)\\[0.9pt]
$L_T$ & $\{z\in\R^d\dvtx  r(T)z\in\Z^d\}$\\[0.9pt]
$L_T(y,R)$ & $L_T\cap B(y,R)$\\[0.9pt]
$\llvert  \cdot\rrvert  $ & $L^1$ norm\\[0.9pt]
\hline
\end{tabular*}
\end{table}
\begin{table}[h]
\tabcolsep=0pt
\begin{tabular*}{\tablewidth}{@{\extracolsep{\fill}}@{}ll@{}}
\hline
\textbf{Notation} & \textbf{Definition/description}\\
\hline
$X(t)$ & Random walk, independent of BRW and environment \\[0.9pt]
$Y(t)$ & Set of all particles at time $t$ \\[0.9pt]
$Y(z,t)$ & Set of particles at $z$ at time $t$ \\[0.9pt]
$N(t)$ & \#Y(t) \\[0.9pt]
$N(z,t)$ & \#Y(z,t) \\[0.9pt]
$H_T(z)$ & Rescaled first time a particle hits $z$, \\[0.9pt]
& $\inf\{t>0\dvtx  N(r(T)z,tT) \geq1\}$ \\[0.9pt]
$H^*_T(z)$ & Rescaled first time $X$ hits $z$, $\inf\{t>0\dvtx  X(tT) =r(T)z\}$ \\[0.9pt]
$H'_T(z)$ & Rescaled first time there are $\geq\exp(\log^{1/4})$ particles at $z$ \\[0.9pt]
$h_T(z)$ & $\inf_{y_0,\ldots, y_n\dvtx y_0=z,y_n=0} (\sum_{j=1}^n q\frac{\llvert  y_{j-1}-y_j\rrvert  }{\xi_T(y_j)} ) \stackrel {z\neq0}{=} \inf_{y\neq z}\{h_T(y)+q\frac{\llvert  z-y\rrvert  }{\xi_T(y)}\}$ \\[0.9pt]
$M_T(z,t)$ & Rescaled $\#$ particles at $z$ at time $t$, \\[0.9pt]
& $M_T(z,t) = \frac{1}{a(T)T}\log_+ N(r(T)z,tT)$\\[0.9pt]
$m_T(z,t)$ & $\sup_{y\in\R^d} \{\xi_T(y)(t-h_T(y))-\llvert  z-y\rrvert  \}$\\[0.9pt]
$J_T(t,R)$ & Probability $X$ jumps $\geq Rr(T)$ times before $tT$\\[0.9pt]
$\calE^1_T(t,R)$ & $\frac{R}{\log T} (\log R - \log t) + \frac {2\,dt}{a(T)}$ (usually small)\\[4pt]
$\calE^2_T(t,R)$ & $\frac{R}{\log T}(\log t - \log R + 1 + \log(2d) + (q+1)\log\log T)$  (usually small)\\[0.9pt]
$\bar\xi_T(R)$ & $\max_{y\in B(0,R)} \xi_T(y)$\\[0.9pt]
$\bar h_T(R)$ & $\max_{y\in B(0,R)} h_T(y)$\\[0.9pt]
$t_\infty$ & Fixed time (we are usually only interested up to time $t_\infty$)\\[0.9pt]
$Z$ & $\{z\in L_T(0,\rho_T)\dvtx  \xi_T(z)>\bar\xi_T(\eta_T)\}$, where $\eta_T$ is small\\[0.9pt]
$\kappa(T)$ & $\#Z$\\[0.9pt]
$z_1,\ldots,z_{\kappa(T)}$ & Elements of $Z$ in increasing order of $\xi$\\[0.9pt]
$Z'$ & $\{z\in L_T\setminus L_T(0,\rho_T)\dvtx  \xi_T(z) > \bar\xi _T(\eta_T)\}$, where $\eta_T$ is small\\[0.9pt]
$z_{\kappa(T)+1},\ldots$ & Elements of $Z'$ in arbitrary order\\[0.9pt]
$t_1,t_2,\ldots$ & Usually $t_i = \gamma_T h_T(z_i) - \delta_T$ where $\gamma_T\approx1$ and $\delta_T$ is small\\[0.9pt]
$A^*_T(j,z,t)$ & $ \left\{\matrix{
H^*_T(z)\leq t, H^*_T(z_i)\geq
H^*_T(z)\wedge t_i ~\forall i\leq j, \cr  H^*_T(z_i)\geq H^*_T(z)~\forall i>j}\right\} $\\[8pt]
$G_T(j,z,s,t)$ & $E^{\xi} [\exp (T\int_0^{H^*_T(z)\wedge s} \xi(X(uT)) \,du ) \1_{A^*_T(j,z,t)} ]$\\[0.9pt]
$\bar G_T$ & $\max_{k\leq\kappa(T)} G_T(\kappa(T),z_k,t_k,t_\infty )$\\[2pt]
$\mu_T$ & $\log^{1/4} T$\\
\hline
\end{tabular*}\vspace*{-22pt}
\end{table}

% AOS,AOAS: If there are supplements please fill:
%\begin{supplement}[id=suppA]
% \sname{Supplement A}
% \stitle{Title}
% \slink[doi]{10.1214/00-AOASXXXXSUPP}
% \sdatatype{.pdf}"
% \sdescription{Some text}
%\end{supplement}

%\begin{appendix}
%\section{}
%\end{appendix}

%\section*{Acknowledgments}
\section*{Acknowledgment}
Both authors would like to thank Wolfgang K\"onig for several helpful
discussions.
%MR was supported partially by an EPSRC postdoctoral fellowship
%(EP/K007440/1).

%\pagebreak

%\begin{supplement}[id=suppA]
%\slink[doi]{10.1214/00-AOPXXXXSUPP} %[doi,text={...}] - jei reikia
%suskaldyti doi
%\sfilename{aopXXXX\_supp.pdf}
%\end{supplement}

% imsref loaded by linak, 2015-04-23 10:10:55
%

\printaddresses
\end{document}